\documentclass[reqno]{amsart}%
\usepackage[utf8]{inputenc}%
\usepackage[english]{babel}%
\usepackage{amsmath,amssymb,amsthm,amsfonts}%
\usepackage{hyperref}%
\usepackage{enumerate}%
\usepackage{graphicx}
\usepackage{mathrsfs}
\allowdisplaybreaks%
\numberwithin{equation}{section}%
\newcommand{\la}{\lambda}
\newcommand{\Ga}{\mathfrak{C}}
\newcommand{\ga}{\mathfrak{c}}

\newcommand{\al}{\alpha}
\newcommand{\si}{\sigma}
\newcommand{\be}{\beta}

\newcommand{\Z}{\mathbb{Z}}
\newcommand{\Yb}{\mathbb{Y}}

\newcommand{\R}{\mathbb{R}}
\newcommand{\C}{\mathbb{C}}
\newcommand{\GT}{\mathbb{GT}}

\newcommand{\V}{\mathsf{V}}
\newcommand{\q}{{}_{q}}

\renewcommand{\i}{\mathrm{i}}
\newcommand{\vol}{\mathsf{vol}}

\newcommand{\x}{\mathsf{x}}
\newcommand{\X}{\mathsf{X}}
\newcommand{\om}{\mathsf{w}_c}
\newcommand{\Om}{\Omega}
\newcommand{\tg}{\mathcal{T}}
\newcommand{\omb}{\bar{\mathsf{w}}_c}
\newcommand{\B}{{\mathrm{B}}}
\newcommand{\D}{\mathcal{D}}
\newcommand{\Pc}{\mathbf{P}}
\newcommand{\Pl}{\mathcal{P}}
\newcommand{\Pp}{\mathbb{P}}
\newcommand{\da}{\downarrow}
\newcommand{\ua}{\uparrow}

\renewcommand{\d}{\Delta}
\newcommand{\szzz}{S^{(3)}}
\newcommand{\A}{\mathcal{A}}
\newcommand{\Kast}{\mathsf{Kast}}
\newcommand{\wt}{\includegraphics[width=6.5pt]{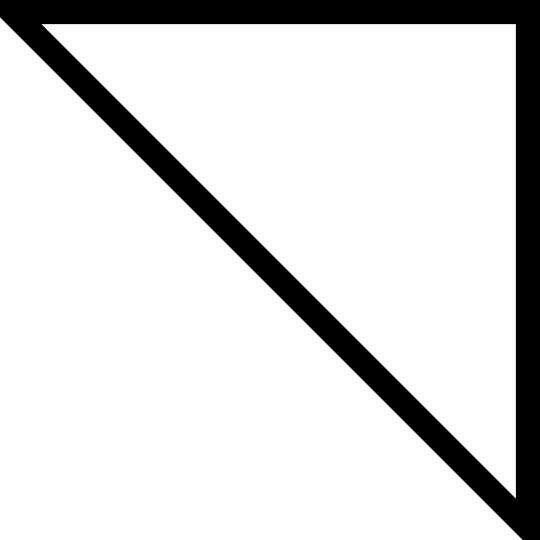}}
\newcommand{\bt}{\includegraphics[width=6.5pt]{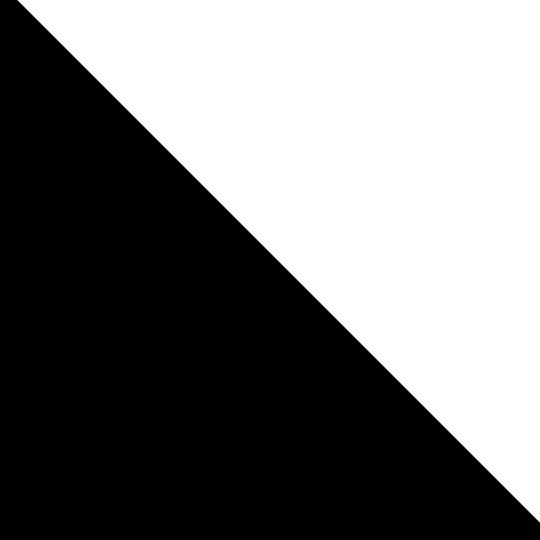}}
\newcommand{\lozv}{\includegraphics[width=5pt]{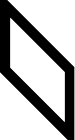}}

\newcommand{\lozvs}{\includegraphics[width=3pt]{lozv.pdf}}
\newcommand{\lozls}{\includegraphics[width=6pt]{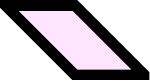}}
\newcommand{\lozss}{\includegraphics[width=3pt]{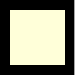}}
\newtheorem{proposition}{Proposition}[section]
\newtheorem{lemma}[proposition]{Lemma}

\newtheorem{theorem}[proposition]{Theorem}
\newtheorem{maintheorem}{Theorem}
\theoremstyle{definition}

\newtheorem{definition}[proposition]{Definition}
\newtheorem{remark}[proposition]{Remark}

\newtheorem{remarks}[proposition]{Remarks}
\begin{document}
\title{Asymptotics of Random Lozenge Tilings via Gelfand-Tsetlin schemes}
\author{Leonid Petrov}
\date{\today}
\address{Department of Mathematics, Northeastern University, 360 Huntington ave., Boston, MA 02115, USA}
\address{Dobrushin Mathematics Laboratory, Kharkevich Institute for Information Transmission Problems, Moscow, Russia}
\email{lenia.petrov@gmail.com}

\begin{abstract}
   A Gelfand-Tsetlin scheme of depth $N$ is a triangular array with $m$ integers at level $m$, $m=1,\ldots,N$, subject to certain interlacing constraints. We study the ensemble of uniformly random Gelfand-Tsetlin schemes with arbitrary fixed $N$th row. We obtain an explicit double contour integral expression for the determinantal correlation kernel of this ensemble (and also of its $q$-deformation).

   This provides new tools for asymptotic analysis of uniformly random lozenge tilings of polygons on the triangular lattice; or, equivalently, of random stepped surfaces. We work with a class of polygons which allows arbitrarily large number of sides. We show that the local limit behavior of random tilings (as all dimensions of the polygon grow) is directed by ergodic translation invariant Gibbs measures. The slopes of these measures coincide with the ones of tangent planes to the corresponding limit shapes described by Kenyon and Okounkov \cite{OkounkovKenyon2007Limit}. We also prove that at the edge of the limit shape, the asymptotic behavior of random tilings is given by the Airy process.

   In particular, our results cover the most investigated case of random boxed plane partitions (when the polygon is a hexagon).
\end{abstract}

\maketitle

% \setcounter{tocdepth}{1}
% \tableofcontents
% \setcounter{tocdepth}{2}

\section{Introduction} % (fold)
\label{sec:introduction}

We study uniformly random tilings (by lozenges of three types) of polygons drawn on the triangular lattice with sides parallel to one of three possible directions. Equivalently, one can speak either of random stepped surfaces --- that is, continuous 3-dimensional surfaces glued out of $1\times1\times1$ boxes with sides parallel to coordinate lines; or dimer coverings (= perfect matchings) of the hexagonal bipartite graph which is dual to the triangular lattice (see Fig.~\ref{fig:tiling_dimers}).\begin{figure}[htbp]
  \begin{tabular}{cc}
    \includegraphics[width=135pt]{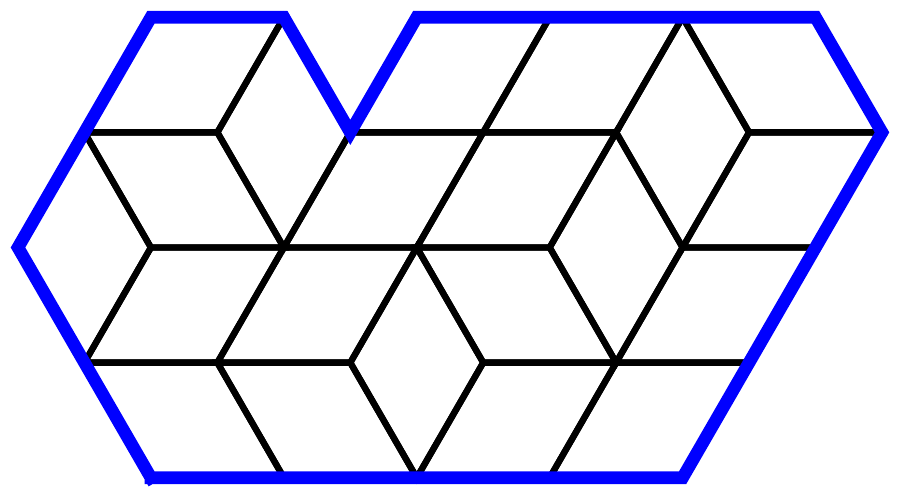}&
    \includegraphics[width=135pt]{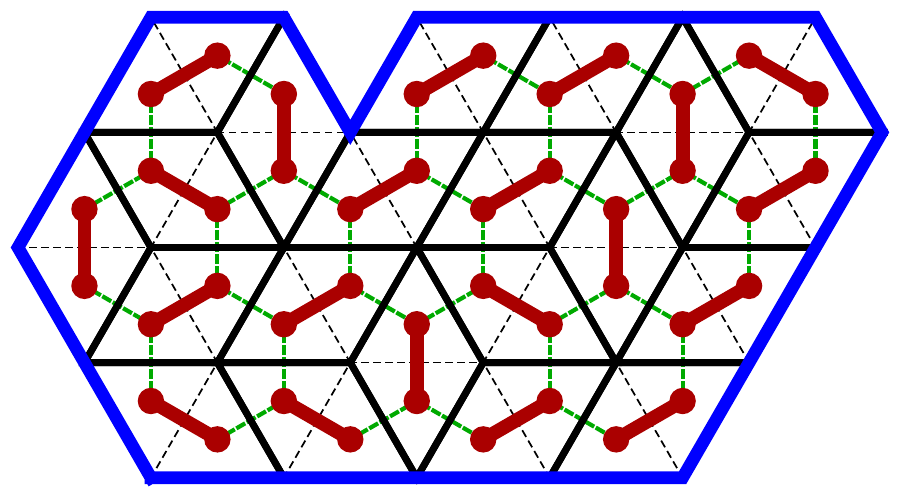}
  \end{tabular}
  \caption{A tiling of a polygon by lozenges (= rhombi) of three types, and its dimer interpretation.}
  \label{fig:tiling_dimers}
\end{figure} This model has received a significant attention \cite{CohnKenyonPropp2000}, \cite{OkounkovKenyon2007Limit}, \cite{Kenyon2004Height}. See also the lecture notes \cite{Kenyon2007Lecture} for a detailed exposition of the subject and more references.

The most well-studied particular case is the model of uniformly random tilings of the hexagon having sides $a,b,c,a,b,c$ (objects also referred to as uniformly random boxed plane partitions). Let us briefly indicate results obtained for the hexagon model, and how we generalize them in the present paper. A detailed description of the model and our results can be found in \S \ref{sec:model_and_results}.

\subsection{Limit shapes for tilings of general polygons} % (fold)
\label{sub:limit_shape_intro}

The law of large numbers under the global scaling (when the polygon is fixed and the mesh of the lattice is going to zero) is well-studied for general polygons. In this limit, the random stepped surface (which is the same as the height function of the corresponding random dimer covering) approaches a non-random limit shape which can be obtained as a unique solution to a suitable variational problem. See \cite{CohnLarsenPropp}, \cite{CohnKenyonPropp2000}, and also \cite{DMB1997Lozenge}, \cite{Destainville1998Lozenge}. This solution was described in \cite{OkounkovKenyon2007Limit} by means of the complex Burgers equation. For a wide class of polygons, the limit shape is an algebraic surface.

Typically, such a limit shape forms facets, i.e., areas where asymptotically only one kind of lozenges is present. There is a curve in the plane (called the frozen boundary) that separates the liquid phase (where asymptotically one can see a random mixture of lozenges of all types) from the facets. This curve is also algebraic and is tangent to all sides of the polygon (or lines containing the sides). See \cite[Fig.~2]{CohnLarsenPropp}, \cite[Fig.~5]{BorodinGorin2008}, \cite[Fig.~1]{OkounkovKenyon2007Limit} for illustrations of random tilings with small mesh where the limit shape and the frozen boundary are clearly seen.
  
% subsection limit_shape (end)

\subsection{Asymptotics in the bulk} % (fold)
\label{sub:asymptotics_in_the_bulk_intro}

In the case of the hexagon it was established in \cite{BKMM2003} and \cite{Gorin2007Hexagon} (see also \cite{Kenyon1997LocalStat}, \cite{KOS2006} and \cite{johansson2002non}, \cite{johansson2005non}, \cite{johansson2006eigenvalues}) that locally near every point of the limit shape, the asymptotics of the random surface are directed by (uniquely defined \cite{Sheffield2008}) ergodic translation invariant Gibbs measure on tilings of the whole plane. The slope of this measure coincides with that of the tangent plane to the corresponding limit shape of \cite{OkounkovKenyon2007Limit} at the chosen point.

The limiting translation invariant Gibbs measures are believed to be universal; they are present in random tiling models of, e.g., \cite{okounkov2003correlation}, \cite{Okounkov2005}, \cite{BorodinKuan2007U}, \cite{BMRT2009BackWall}. More general measures on tilings of the hexagon also exhibit the same local asymptotics \cite{borodin-gr2009q}. Furthermore, in \cite{Kenyon2004Height} such behavior was established for rather general random stepped surfaces with no asymptotically frozen zones.

In the present paper we obtain the conjectural bulk limit asymptotics for uniformly random tilings of a wide class of polygons (described in \S \ref{sub:lozenge_tilings_of_polygons}). In particular, our approach allows polygons with arbitrarily many sides, which means that the frozen boundary curve and the limit shape surface can have an arbitrarily large degree. The slopes of the limiting Gibbs measures agree with the limit shapes of \cite{OkounkovKenyon2007Limit}.

% subsection asymptotics_in_the_bulk (end)

\subsection{Asymptotics at the edge} % (fold)
\label{sub:asymptotics_at_the_edge_intro}

In the present paper we show that the asymptotics of random tilings at the edge of the limit shape are directed by another universal law --- the Airy process (introduced in \cite{PhahoferSpohn2002} in connection with a polynuclear growth model). 
Various results on Airy-type asymptotics of 
lozenge tilings of unbounded regions were previously established in
\cite{ferrari2003step}, \cite{okounkov2003correlation}, \cite{BorodinKuan2007U}.
For lozenge tilings of bounded polygons, only a partial result for the hexagon
was proven in 
\cite{BKMM2003} via subtle analytical results about asymptotic properties of discrete orthogonal polynomials. Namely, the Airy process asymptotics were shown to hold only in one direction, i.e., for the statical ensemble of random particle configurations on a line, see \cite[Fig.~7]{BKMM2003}. We manage to obtain the (space-time) Airy asymptotics for tilings of polygons in our class.

It should be possible to go further and show (in a way similar to \cite{Okounkov2005}, \cite{BorodinKuan2007U}, and \cite{BMRT2009BackWall}) that at cusps of the limit shape the asymptotics of random tilings are governed by the Pearcey process.

% subsection asymptotics_at_the_edge (end)

\subsection{Connection with orthogonal polynomials} % (fold)
\label{sub:connection_with_orthogonal_polynomials_intro}

It is worth noting that almost all results cited above in \S \ref{sub:asymptotics_in_the_bulk_intro} and \S \ref{sub:asymptotics_at_the_edge_intro} are obtained using in one form or another the technique of orthogonal polynomial ensembles. Because random tilings of polygons other than the hexagon lack a direct connection with orthogonal polynomial ensembles, no similar results about more general polygons were known.

% subsection connection_with_orthogonal_polynomials_intro (end)

\subsection{Technique of the present paper} % (fold)
\label{sub:technique_of_the_present_paper}

We overcome the lack of orthogonal polynomials for more general polygons by establishing a new double contour integral formula (Theorem \ref{thm:K_intro}) for the determinantal correlation kernel of random tilings. This provides necessary tools for the desired asymptotic analysis. 

We believe that our technique could also be applied to a more detailed study of various other models of random tilings such as, e.g., the ones considered recently in \cite{BreuerDuits2011StaircasePaths}, \cite{NordenstamYoung2011HalfHex} and \cite{NordYoung2011}.

% subsection technique_of_the_present_paper (end)

\subsection*{Acknowledgements} % (fold)

The author would like to thank Alexei Borodin for fruitful discussions, and Vadim Gorin for helpful comments. 
I am also grateful to the anonymous referee for remarks which helped to improve the
presentation.
The work was partially supported by the RFBR-CNRS grants 10-01-93114 and 11-01-93105.

% subsection acknowledgements (end)

% section introduction (end)

\section{Model and results} % (fold)
\label{sec:model_and_results}

\subsection{Lozenge tilings of polygons} % (fold)
\label{sub:lozenge_tilings_of_polygons}

For convenience, we perform a simple affine transform of lozenges which were present in Fig.~\ref{fig:tiling_dimers}:
\begin{align}\label{lozenges}
  \begin{array}{c}
  {\includegraphics[width=270pt]{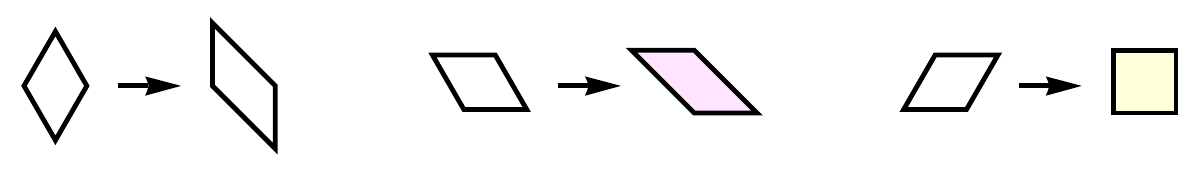}}
  \end{array}
\end{align}
After this transform, the polygons will be drawn on the standard square grid with all sides parallel to one of the coordinate axes or the vector $(-1,1)$. We denote the horizontal and the vertical coordinates by $x$ and $n$, respectively. 

\begin{figure}[htbp]
  \begin{tabular}{c}
    \includegraphics[width=225pt]{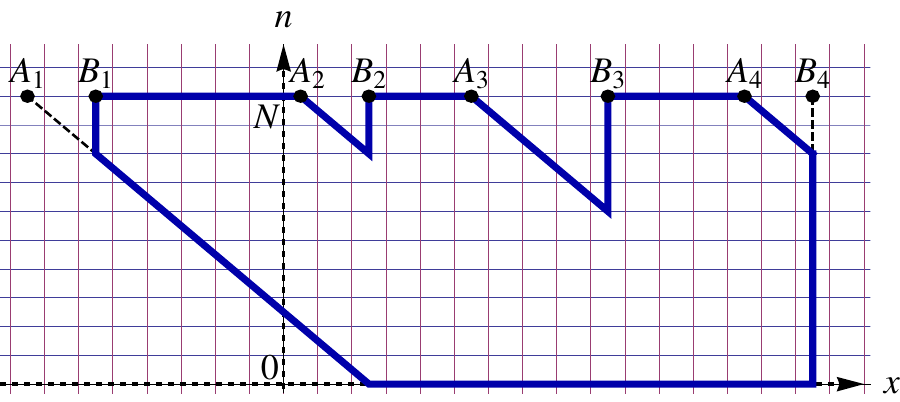}
  \end{tabular}
  \caption{A polygon drawn on the ruled paper. In this example $k=4$, and the polygon has $3k=12$ sides.}
  \label{fig:polygonal_region}
\end{figure}

We will restrict ourselves to polygons of a special kind as shown on Fig.~\ref{fig:polygonal_region}. Every polygon $\Pc$ we consider can be parametrized by two integers $N=1,2,\ldots$, and $k=2,3,\ldots$ (the polygon has $3k$ sides) and by the (proper) half-integers
\begin{align*}
  A_1<B_1<A_2<B_2<\ldots<A_k<B_k,\qquad
  A_i,B_i\in\Z':=\Z+\tfrac12,
\end{align*}
subject to the condition $\sum_{i=1}^{k}(B_i-A_i)=N$. The bottom side of $\Pc$ lies on the horizontal axis $n=0$, and all the $k-1$ top sides lie on one and the same line $n=N$, so $N$ is the height of $\Pc$.% We can also regard $\Pc$ as a ``long'' hexagon with $k-2$ triangles cut on top of it.

The main object of the present paper is the uniform measure $\Pp_\Pc$ on the set of all lozenge tilings of the polygon $\Pc$ by lozenges of three types (the transformed ones in (\ref{lozenges})). An example of such a tiling is given in Fig.~\ref{fig:polygonal_region_tiling} (we trivially extend the tiling to the whole strip $0\le n\le N$ with $N$ small triangles added on top). 
\begin{figure}[htbp]
  \begin{tabular}{cc}
    \includegraphics[width=320pt]{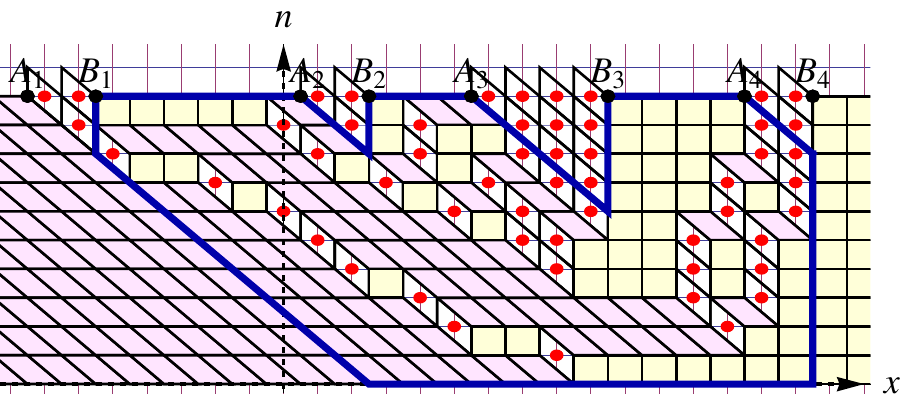}
  \end{tabular}
  \caption{A tiling of $\Pc$. Particles correspond to lozenges of one type.}
  \label{fig:polygonal_region_tiling}
\end{figure}

For $k=2$, our measure $\Pp_\Pc$ becomes the uniform measure on all tilings of the hexagon with sides of lengths $(A_2-B_1),(B_2-A_2)\sqrt2,(B_1-A_1),(A_2-B_1),(B_2-A_2)\sqrt2,(B_1-A_1)$. The total number of tilings of such a hexagon is given by the classical MacMahon product formula, e.g., see \cite[\S7.21]{Stanley1999}. Existing results about $\Pp_{\Pc}$ in this case are discussed in the Introduction.

There is also a product formula for the total number of tilings of a polygon $\Pc$ for arbitrary $k$, i.e., for the partition function of our model:
\begin{align}\label{ZPc}
  Z_{\Pc}=\prod\nolimits_{1\le i< j\le N}\frac{\x_{i}^{N}-\x_{j}^{N}}{j-i},
\end{align}
where
\begin{align}\label{fixed_top_row}
  \{\x_{N}^{N}<\ldots<\x_{1}^{N}\}&=\{A_{1}+\tfrac12<A_{1}+\tfrac32<\ldots<B_1-\tfrac32<B_1-\tfrac12<\\&\qquad<A_2+\tfrac12<\ldots<B_2-\tfrac12<\ldots<A_{k}+\tfrac12<\ldots<B_k-\tfrac12\}
  \nonumber
\end{align}
are the positions of particles on the top line $n=N$, see Fig.~\ref{fig:polygonal_region_tiling}. One way to obtain (\ref{ZPc}) is to interpret $Z_{\Pc}$ as a dimension of a certain irreducible representation of the unitary group $U(N)$, and then use the classical Weyl dimension formula \cite{Weyl1946}. See \S \ref{sec:combinatorics_of_the_model_and_random_gelfand_tsetlin_schemes} for more detail. 

% subsection lozenge_tilings_of_polygons (end)

\subsection{Particle configurations and the correlation kernel} % (fold)
\label{sub:particle_configurations_and_the_correlation_kernel}

Another way to look at the number of tilings $Z_\Pc$ (\ref{ZPc}) is provided by the dimer interpretation of our model (Fig.~\ref{fig:tiling_dimers}). Namely, $Z_{\Pc}$ is equal to the determinant of the Kasteleyn matrix for the honeycomb graph (dual to the original triangular lattice) located inside $\Pc$ \cite{Kasteleyn1967}. In particular, this implies that the uniform measure $\Pp_\Pc$ on tilings of $\Pc$ is determinantal (e.g., see \cite[Cor. 3]{Kenyon2007Lecture}). There are several (very similar) ways to express the determinantal property of $\Pp_\Pc$, we choose the most convenient for us.

We put a particle in the center of every lozenge of type \lozv{} (see Fig.~\ref{fig:polygonal_region_tiling}). Note that the coordinates of all such particles are integers. Thus, one sees a configuration of particles $\X:=\{\x_{j}^{m}\colon m=1,\ldots,N;\ j=1,\ldots,m\}\in\Z^{N(N+1)/2}$ with precisely $m$ particles at the $m$th horizontal line, $m=0,\ldots,N$. Because we have a tiling of $\Pc$, these particles must satisfy the \emph{interlacing constraints}:
\begin{align}\label{interlacing_intro}
  \x_{j+1}^{m}<\x_{j}^{m-1}\le \x_{j}^{m}
\end{align}
(for all $j$'s and $m$'s for which these inequalities can be written out). Clearly, lozenge tilings of $\Pc$ and such interlacing arrays $\X$ with fixed top row as in (\ref{fixed_top_row}) are in a bijective correspondence. These arrays are also in a bijection with Gelfand-Tsetlin schemes (\S \ref{sub:connection_to_measures_on_tilings}). 

In this way, the measure $\Pp_\Pc$ on tilings becomes a probability measure on interlacing particle arrays $\X=\{\x_{j}^{m}\}$. We denote it also by $\Pp_\Pc$. 

\begin{definition}\label{def:correlation_functions_intro} 
  Let $(x_1,n_1),\ldots,(x_s,n_s)$ be pairwise distinct points, $x_i\in\Z$, $1\le n_i\le N-1$. The \emph{correlation functions} of the measure $\Pp_\Pc$ are defined as
  \begin{align*}
    \rho_s(x_1,n_1;\ldots;x_s,n_s):=\Pp_\Pc&\big(\mbox{there is a particle of the random configuration $\{\x_j^{m}\}$}\\
    &\qquad\mbox{ at position $(x_j,n_j)$ for every $j=1,\ldots,s$}\big).
  \end{align*}
\end{definition}
The dimer interpretation mentioned above implies that the measure $\Pp_{\Pc}$ on interlacing particle arrays is \emph{determinantal}. That is, there exists a function $K(x,n;y,m)$ called the \emph{correlation kernel}, such that 
\begin{align}\label{correlation_kernel_intro}
  \rho_s(x_1,n_1;\ldots;x_s,n_s)=
  \det[K(x_i,n_i;x_j,n_j)]_{i,j=1}^{s}
\end{align}
for any $s$ and any collection of positions $(x_1,n_1),\ldots,(x_s,n_s)$. About determinantal point processes in general see the surveys \cite{Soshnikov2000}, \cite{peres2006determinantal}, \cite{Borodin2009}.

The first main result of the present paper is an explicit formula for the correlation kernel $K$ of $\Pp_\Pc$ in terms of double contour integrals:
\begin{maintheorem}[Correlation kernel]
  \label{thm:K_intro}
  For $1\le n_1\le N$, $1\le n_2\le N-1$, and $x_1,x_2\in\Z$, the correlation kernel of the point process $\Pp_\Pc$ has the form\footnote{Here and below $1_{\{\cdot\cdot\cdot\}}$ denotes the indicator of a set, and $(y)_m:=y(y+1)\ldots(y+m-1)$, $m=1,2,\ldots$ (with $(y)_0:=1$) is the Pochhammer symbol.}
  \begin{align}&
    K(x_1,n_1;x_2,n_2)=
    -1_{n_2<n_1}1_{x_2\le x_1}\frac{(x_1-x_2+1)_{n_1-n_2-1}}{(n_1-n_2-1)!}
    +\frac{(N-n_1)!}{(N-n_2-1)!}
    \times
    \nonumber
    \\&\qquad\quad\times
    \frac1{(2\pi\i)^{2}}
    \oint\limits_{\{z\}}dz\oint\limits_{\{w\}}dw
    \frac{(z-x_2+1)_{N-n_2-1}}{(w-x_1)_{N-n_1+1}}
    \frac{1}{w-z}
    \prod_{i=1}^{k}
    \frac{(A_i+\frac12-w)_{B_i-A_i}}{(A_i+\frac12-z)_{B_i-A_i}}.
    \label{K_intro}
  \end{align}
  The contours in $z$ and $w$ are positively (counter-clockwise) oriented and do not intersect. The contour $\{z\}$ encircles the integer points $x_2,x_2+1,\ldots,B_k-\tfrac12$ in $z$ and only them (i.e., does not contain $x_2-1,x_2-2,\ldots$). The contour $\{w\}$ contains $\{z\}$ and all the points $x_1,x_1-1,\ldots,x_1-(N-n_1)$.
\end{maintheorem}

We prove Theorem \ref{thm:K_intro} in \S\S \ref{sec:correlation_kernel_for_the_measure_q_vol_}--\ref{sec:correlation_kernel_for_uniform_gelfand_tsetlin_schemes} using an Eynard-Mehta type theorem which allows the number of particles to vary. Its statement can be found in, e.g., \cite[\S4]{Borodin2009}. However, the application of that theorem is not straightforward. Namely, to get the result about the uniform measure $\Pp_\Pc$, we first consider its $q$-deformation $\q\Pp_\Pc$ obtained by weighting every tiling proportional to $q^{vol}$, where $vol$ is the volume under the corresponding stepped surface. Such deformed measures were considered in \cite{OkounkovKenyon2007Limit} and (as a particular case) in \cite{borodin-gr2009q}. %, where bulk asymptotics for a family of non-uniform measures on tilings of the hexagon were investigated. 
The correlation kernel $\q K$ of $\q\Pp_\Pc$ is given by a double contour integral expression similar to (\ref{K_intro}), but with a $q$-hypergeometric function inside the integral (see Theorem \ref{thm:q_kernel}). 

Then in a limit as $q\to1$ we obtain the formula for the kernel $K$ of Theorem \ref{thm:K_intro}. The complexity of the expression for $\q K$ is the reason why restrict our asymptotic analysis to the case $q=1$. The disappearance of the $q$-hypergeometric function in the $q\to 1$ limit may seem rather surprising. A possible explanation can be given in the course of computations (see Remark \ref{rmk:q_disappearance} and Proposition \ref{prop:oint_principal_term}): the limit $q\to1$ kills many additional negligible terms, and this significantly simplifies the post-limit ($q=1$) formulas.

\begin{remark}[Kasteleyn matrix]
  \label{rmk:Kasteleyn_intro}
  Using Theorem \ref{thm:K_intro}, in \S \ref{sub:inverse_kasteleyn_matrix} we manage to express the inverse Kasteleyn matrix for the honeycomb graph inside $\Pc$ (see Fig.~\ref{fig:tiling_dimers}, right) through the correlation kernel $K$ (\ref{K_intro}). Similar results for other models can be found in \cite[\S5.1]{Ferrari2008} and \cite[\S7.2]{borodin-gr2009q} (the latter model includes the hexagon case).
\end{remark}

\begin{remarks}[Degeneration to some known kernels]
  {\bf1.} A similar kernel for the continuous Gelfand-Tsetlin patterns was obtained (using similar methods) recently in \cite[Prop. 2.4]{Metcalfe2011GT}. The model of \cite{Metcalfe2011GT} describes distributions of eigenvalue minor processes of random Hermitian $N\times N$ matrices. It can be shown that the kernel of \cite{Metcalfe2011GT} is a limit of ours when one: 1) embeds every horizontal line $n=0,\ldots,n=N$ (see Fig.~\ref{fig:polygonal_region_tiling}) into $\R$ as $\Z\mapsto \xi\Z$; 2) scales the positions of particles as $x_{j}^{n}=\xi^{-1}\tilde x_{j}^{n}$ with new continuous $\tilde x_{j}^{n}$'s; and then lets $\xi\to0$. Note that in this continuous limit the connection with random tilings is lost.

  {\bf2.} In \cite{BorodinKuan2007U}, a correlation kernel for so-called ergodic measures on (infinite) Gelfand-Tsetlin schemes was obtained. Equivalently, one can speak about random tilings of the whole upper half plane when the tiling spreads all the way upwards ($N=\infty$ on Fig.~\ref{fig:polygonal_region_tiling}). These measures are limits (in some sense) of the measures $\Pp_\Pc$ as $N\to\infty$. This limit transition is related to classification of characters of the infinite-dimensional unitary group. See the recent paper \cite{BorodinOlsh2011GT} and references therein for more detail on the subject. It can be shown that the kernel of \cite{BorodinKuan2007U} is a limit of our kernel $K$ in this regime.
\end{remarks}

% subsection particle_configurations_and_the_correlation_kernel (end)

\subsection{Asymptotics in the bulk} % (fold)
\label{sub:asymptotics_in_the_bulk}

We consider $N\to\infty$ asymptotics as all dimensions of the polygon $\Pc(N)$ grow. We assume that the parameters $A_i(N),B_i(N)$ of $\Pc(N)$ are scaled as
\begin{align}\label{scale_Ai_Bi}
  A_i(N)=[a_i N]+\frac12+\delta_i,\qquad B_i(N)=[b_i N]+\frac12+\delta_i'
  \qquad\mbox{($i=1,\ldots,k$)}
  .
\end{align}
Here $a_1<b_1<\ldots<a_k<b_k$ are new continuous parameters which satisfy $\sum_{i=1}^{k}(b_i-a_i)=1$. The integer constants 
$\delta_i,\delta_i'\in\{-1,0,1\}$
above are needed to ensure that $\sum_{k=1}^{N}(B_i(N)-A_i(N))=N$,
they are not relevant in the scaling limit.
Clearly, one has $A_i(N),B_i(N)\in\Z'$. 

One can then rescale the growing polygons $\Pc(N)$ by $N^{-1}$ in both directions, so they will approach some fixed polygon $\Pl$ drawn on the new $(\chi,\eta)$ coordinate plane. The polygon $\Pl$ is parametrized by $\{a_i,b_i\}_{i=1}^{k}$ in the same way as it was for $\Pc(N)$ and $\{A_i(N),B_i(N)\}_{i=1}^{k}$, and $\Pl$ is located inside the strip $0\le \eta\le 1$. See the polygon on Fig.~\ref{fig:frozen_boundary} (we preserve the proportions of the polygon on Fig.~\ref{fig:polygonal_region}--\ref{fig:polygonal_region_tiling}). 

Alternatively, one could say that we consider tilings of a fixed polygon $\Pl$ with refining mesh. In terms of particle configurations, this asymptotic regime means that the $N$ particles on the top row form $k$ macroscopic clusters in which the particles are densely packed.

\begin{figure}[htbp]
  \begin{tabular}{c}
    \includegraphics[width=225pt]{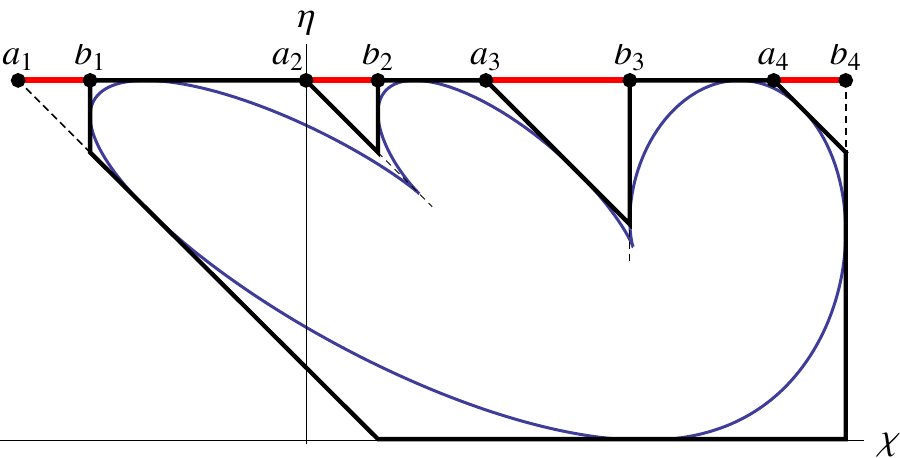}
  \end{tabular}
  \caption{The limiting polygon $\Pl$ on the $(\chi,\eta)$ plane and the frozen boundary curve. The segments $[a_i,b_i]$ indicate densely packed clusters of particles on the top row.}
  \label{fig:frozen_boundary}
\end{figure}

Let us discuss the limiting object which will describe local asymptotics of random tilings.

\begin{definition}[Incomplete beta kernel \cite{okounkov2003correlation}]
\label{def:incomplete_beta}
  Let $\Om\in\C\setminus\{0,1\}$, $\Im\Om\ge0$, be a parameter called the \emph{complex slope}. The \emph{incomplete beta kernel} $\B_{\Om}(m,l)$ is defined as
  \begin{align*}
    \B_{\Om}(m,l):=
    \frac{1}{2\pi\i}
    \int_{\bar\Om}^{\Om}
    (1-u)^{m}u^{-l-1}du,
    \qquad
    m,l\in\Z,
  \end{align*}
  where the path of integration crosses $(0,1)$ for $m\ge0$ and $(-\infty,0)$ for $m<0$. 
\end{definition}
This kernel was introduced in \cite{okounkov2003correlation} in connection with local asymptotics of random plane partitions. As was shown in \cite{Sheffield2008} and \cite{KOS2006}, for every complex slope $\Om$ there is a unique ergodic translation invariant Gibbs measure on tilings of the whole plane, and this is the determinantal point process corresponding to the kernel $\B_\Om(m_1-m_2,l_1-l_2)$ (as in (\ref{correlation_kernel_intro}), the correlation kernel depends on two points $(m_1,l_1),(m_2,l_2)$ in the plane). 
Under this measure, the proportions of lozenges seen in a large box 
(denote these proportions by $p_{\lozvs},p_{\lozss},p_{\lozls}$, 
with $p_{\lozvs}+p_{\lozss}+p_{\lozls}=1$)
are determined by the complex slope $\Om$ as follows.\begin{figure}[htbp]
  \begin{tabular}{c}
    \includegraphics[width=140pt]{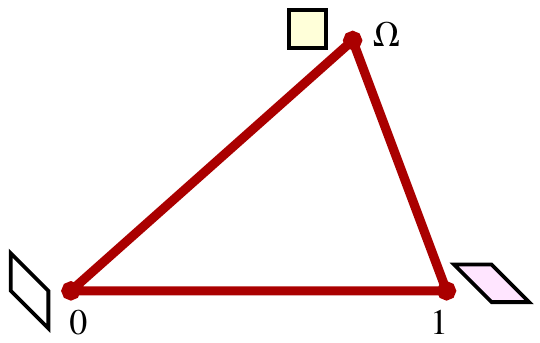}
  \end{tabular}
  \caption{Lozenge densities defined by a complex slope $\Om$, $\Im\Om>0$.}
  \label{fig:lozenge_types}
\end{figure} For $\Im\Om>0$, $p_{\lozvs},p_{\lozss},p_{\lozls}$ are proportional to angles of the triangle in the complex plane with vertices $0,\Om,1$ (see Fig.~\ref{fig:lozenge_types}). For $\Om\in\R$, $\Om\ne 0,1$, the Gibbs measure is supported by a single configuration with all lozenges of one type. This may be seen as a degeneration of Fig.~\ref{fig:lozenge_types} when $\Om$ approaches the real line. 

We note that the cases $\Om=0,1$, or $\infty$ in our model correspond to points where two different frozen facets meet (see Proposition \ref{prop:tangent_points_intro}). At these points the local asymptotic picture is not translation invariant.

\begin{maintheorem}[Asymptotics in the bulk]\label{thm:bulk_intro}
  Let $(\chi,\eta)$ encode a scaled global position in the interior of the 
  limiting polygon $\Pl$ (see Fig.~\ref{fig:frozen_boundary}). As $N\to\infty$, locally around $(\chi,\eta)$, the measure $\Pp_{\Pc(N)}$ on random tilings converges to an ergodic translation invariant Gibbs measure on tilings of the whole plane with some complex slope $\Om=\Om(\chi,\eta)$. In terms of correlation functions, this means that for any points $(x_1(N),n_1(N)),\ldots,(x_s(N),n_s(N))$ such that
  \begin{align*}
    {x_i(N)}/{N}\to\chi,\qquad {n_i(N)}/{N}\to \eta,
    \qquad
    N\to\infty \qquad (i=1,\ldots,s)
  \end{align*}
  while the differences stabilize as
  \begin{align*}
    x_i-x_j=\lim_{N\to\infty}(x_i(N)-x_j(N))\in\Z,
    \qquad
    n_i-n_j=\lim_{N\to\infty}(n_i(N)-n_j(N))\in\Z,
  \end{align*}
  we have the convergence
  \begin{align*}
    \rho_s(x_1(N),n_1(N);\ldots;x_s(N),n_s(N))
    \to
    \det[\B_\Om(n_i-n_j;x_j-x_i)]_{i,j=1}^{s}.
  \end{align*}
\end{maintheorem}

This theorem is proved in \S \ref{sec:asymptotics_in_the_bulk_limit_shape_and_frozen_boundary} using the double contour integral representation (\ref{K_intro}) for the correlation kernel $K$. The proof uses the saddle point analysis and mainly follows the lines of \cite{Okounkov2002}, \cite{okounkov2003correlation}, \cite{Okounkov2005}, and \cite{BorodinKuan2007U}.

% subsection asymptotics_in_the_bulk (end)

\subsection{Complex Burgers equation, limit shape, and frozen boundary} % (fold)
\label{sub:limit_shape_complex_burgers_equation_frozen_boundary}

Let us describe how the complex slope $\Om(\chi,\eta)$ at a global point $(\chi,\eta)$ (Theorem \ref{thm:bulk_intro}) is characterized.

There is an (open) domain $\D$ inside the polygon $\Pl$ (the so-called \emph{liquid region}) such that for every $(\chi,\eta)\in\D$ we have $\Im\Om(\chi,\eta)>0$. This means that inside $\D$ all types of lozenges are asymptotically present (in proportions determined by $\Om(\chi,\eta)$, see \S \ref{sub:asymptotics_in_the_bulk}). 

On the other hand, everywhere in $\Pl\setminus\D$, the random configuration is asymptotically frozen, i.e., it contains only one type of lozenges. The liquid region is separated from frozen ones by a curve $\partial\D$ called the \emph{frozen boundary} (see Fig.~\ref{fig:frozen_boundary}). For our model, $\partial\D$ is an algebraic curve of degree $k$ inscribed in $\Pl$: it is tangent to all sides of the polygon $\Pl$ (or lines containing them). The curve $\partial \D$ has $k-2$ turning points (corresponding to cusps of the limit shape). In the $k=2$ case when $\Pl$ is a hexagon, $\partial\D$ reduces to the ellipse inscribed in $\Pl$.

\begin{proposition}[The complex slope $\Om(\chi,\eta)$]\label{prop:complex_Burgers_intro}
  For $(\chi,\eta)\in\D$, the slope $\Om(\chi,\eta)$ is the only solution in the upper half plane of the following algebraic equation:\footnote{This equation has degree $k+1$, but it always has the root $\Om=1$, so in fact $\Om$ satisfies a degree $k$ equation. There are other (sometimes more suitable) complex parameters of the limit shape and the frozen boundary --- $\om$ (\ref{omc_intro}) and $\tg$ (Remark \ref{rmk:tg_intro}). See also \S\S \ref{sub:limit_shape_and_the_complex_burgers_equation}--\ref{sub:frozen_boundary}.}
  \begin{align}\label{Algebraic_equation_Omega}
    \Om\cdot\prod\limits_{i=1}^{k}
    \big(
    (a_i-\chi+1-\eta)\Om-(a_i-\chi)
    \big)=
    \prod\limits_{i=1}^{k}
    \big(
    (b_i-\chi+1-\eta)\Om-(b_i-\chi)
    \big).
  \end{align}
  The slope $\Om(\chi,\eta)$ also satisfies a differential equation (called the complex Burgers equation \cite{OkounkovKenyon2007Limit}):
  \begin{align}\label{complex_Burgers}
    \Om(\chi,\eta)\frac{\partial\Om(\chi,\eta)}{\partial\chi}=
    -(1-\Om(\chi,\eta))\frac{\partial\Om(\chi,\eta)}{\partial\eta}.
  \end{align}
\end{proposition}
In particular, this implies that the complex slope $\Om(\chi,\eta)$ of the limiting ergodic translation invariant Gibbs measure describing the local asymptotics around the global position $(\chi,\eta)$ (Theorem \ref{thm:bulk_intro}) coincides with the complex slope of the tangent plane to the limit shape of \cite{CohnKenyonPropp2000}, \cite{OkounkovKenyon2007Limit} at the point $(\chi,\eta)\in\D$. 
Indeed, both complex slopes satisfy the same complex Burgers equation in the liquid region $\D$, 
and also coincide on frozen facets and in particular on the
frozen boundary (which is evident from 
Theorem \ref{thm:bulk_intro}).
Thus, the normal vector to the limit shape at every point $(\chi,\eta)\in\Pl\setminus \partial\D$ looks as $(p_{\lozvs},p_{\lozss},p_{\lozls})$ (in the coordinates shown on Fig.~\ref{fig:normal_vectors}). 
(On the frozen boundary $\partial\D$ the limit shape surface does not have a normal vector.) This interpretation is the reason why we call the parameter $\Om$ the complex slope.
\begin{figure}[htbp]
  \begin{tabular}{cc}
    \includegraphics[width=200pt]{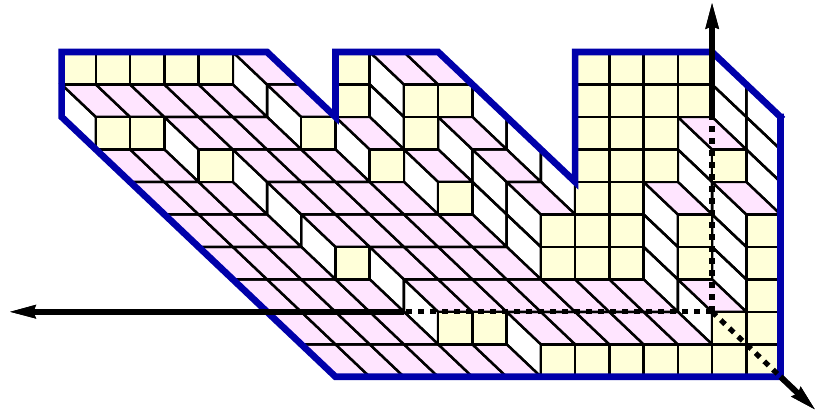}
  \end{tabular}
  \caption{3-dimensional coordinate system for a random stepped surface.}
  \label{fig:normal_vectors}
\end{figure} 

Our analysis also produces an explicit rational parametrization of the frozen boundary $\partial\D$ (which, in particular, is used to draw the curve on Fig.~\ref{fig:frozen_boundary}, see also Fig.~\ref{fig:frozen_boundary_big}). For $(\chi,\eta)\in\partial\D$, the two complex roots $\Om,\bar\Om$ of (\ref{Algebraic_equation_Omega}) coincide. Thus, one could take $\Om\in\R$ as a parameter for the curve. However, it is more convenient to use another parameter $\om$, where
\begin{align}\label{omc_intro}
  \om:=\chi+\frac{(1-\eta)\Om}{1-\Om},\qquad
  \Om=\frac{\om-\chi}{\om-\chi+1-\eta}.
\end{align}
Note that $\Om$ is monotonically increasing in $\om$.
\begin{proposition}[Frozen boundary]\label{prop:frozen_boundary_intro}
  The frozen boundary curve can be para\-met\-ri\-zed as
  \begin{align}\label{frozen_param1}
    \chi(\om)=\om+\frac{\Pi(\om)-1}{\Sigma(\om)};\qquad
    \eta(\om)=\frac{\Pi(\om)(\Sigma(\om)-\Pi(\om)+2)-1}{\Pi(\om)\Sigma(\om)},
  \end{align}
  where
  \begin{align}\label{frozen_param2}
    \Pi(\om):=\prod\nolimits_{i=1}^{k}\frac{\om-b_i}{\om-a_i},\qquad
    \Sigma(\om):=\sum\nolimits_{i=1}^{k}\Big(\frac1{\om-b_i}-\frac1{\om-a_i}\Big),
  \end{align}
  with parameter $-\infty\le\om<\infty$. As $\om$ increases, the frozen boundary is passed in the clockwise direction (so that $\D$ is to the right of $\partial\D$).
\end{proposition}

We can also identify the tangent points on the frozen boundary:

\begin{proposition}[Tangent points]\label{prop:tangent_points_intro}
  The tangent vector $(\dot\chi(\om),\dot\eta(\om))$ to the frozen boundary curve has slope
  \begin{align}\label{tangent_vector}
    \frac{\dot\chi(\om)}{\dot\eta(\om)}=\frac{\om-\chi(\om)}{1-\eta(\om)}=
    \frac{\Pi(\om)}{1-\Pi(\om)}=:\tg(\om).
  \end{align}
  There are sides of $\Pl$ of three directions to which the curve $\partial\D$ is tangent:
  \begin{enumerate}[(D)]
    \item The values $\om=a_i$, $i=1,\ldots,k$, correspond to points of $\partial\D$ where it is tangent to sides of the polygon $\Pl$ parallel to the ``diagonal'' 
    vector $(-1,1)$. Here $\Om(a_i)=\infty$.
  \end{enumerate}
  \begin{enumerate}[(V)]
    \item When $\om=b_i$, $i=1,\ldots,k$, the frozen boundary $\partial\D$ is tangent to vertical sides of $\Pl$. We have $\Om(b_i)=0$.
  \end{enumerate}
  \begin{enumerate}[(H)]
    \item At points $\om=h_1,\ldots,h_k$, where $\Om(h_i)=1$, the frozen boundary $\partial\D$ is tangent to horizontal sides of $\Pl$. This includes the case $h_1=-\infty$ where the curve is tangent to the bottom horizontal side of $\Pl$. Clearly, $h_i\in(b_{i-1},a_i)$, $i=2,\ldots,k$.
  \end{enumerate}
  Finally, in each segment $(h_2,h_3),\ldots,(h_{k-1},h_k)$ there is one turning point of the frozen boundary (corresponding to a cusp of the limit shape).
\end{proposition}
See also Fig.~\ref{fig:tau_om_S3} for graphs of $\tg(\om)$ and $\Om(\om)$ along the frozen boundary curve.

\begin{remark}\label{rmk:tg_intro}
  Let us extend the definition of $\tg$ (\ref{tangent_vector}) inside the liquid region (using (\ref{omc_intro})) as $\tg(\chi,\eta):=\frac{\om(\chi,\eta)-\chi}{1-\eta}$. Observe that the complex Burgers equation (\ref{complex_Burgers}) can be rewritten in the following form inside $\D$:
  \begin{align}\label{complex_Burgers_tg}
    \tg(\chi,\eta)\frac{\partial\om(\chi,\eta)}{\partial\chi}
    =-\frac{\partial\om(\chi,\eta)}{\partial\eta}.
  \end{align}
\end{remark}

% subsection limit_shape_complex_burgers_equation_frozen_boundary (end)

\subsection{Asymptotics at the edge of the limit shape} % (fold)
\label{sub:asymptotics_at_the_edge_of_the_limit_shape}

Here we explain our results about the asymptotics of random tilings when the global position $(\chi,\eta)$ we are looking at is located on the frozen boundary. See \S \ref{sec:asymptotics_at_the_edge} for precise formulations. 

An intuitive way of thinking about the edge asymptotics governed by the Airy process (introduced in \cite{PhahoferSpohn2002}) is to consider nonintersecting paths. Namely, observe first that the frozen boundary at any chosen point can have one of three ``directions'' (named according to Proposition \ref{prop:tangent_points_intro} with the understanding that the frozen 
boundary is passed in the clockwise direction): 
``VH'' from tangent point to the vertical side to tangent point to the horizontal side of the polygon; 
``HD'' from the horizontal side to the diagonal side of the polygon;
``DV'' from the diagonal side to the vertical of the polygon.
See Fig.~\ref{fig:Airy}. Note that this is the same as to classify frozen regions by the type of lozenges they are built of.

\begin{figure}[htbp]
  \begin{center}
    \begin{tabular}{ccc}
      \includegraphics[width=111pt]{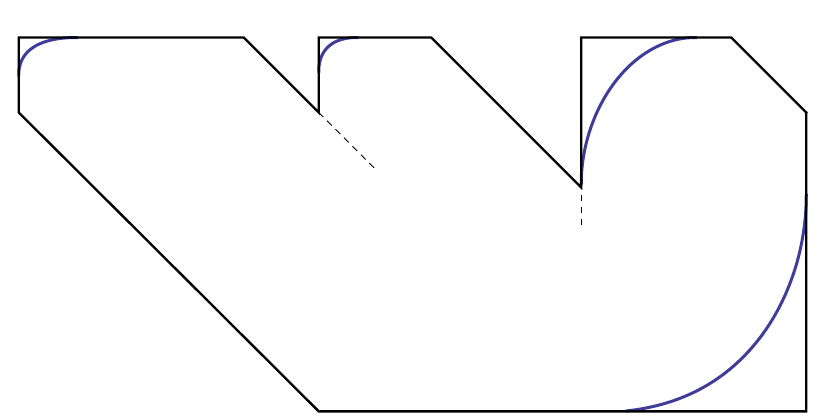}&
      \includegraphics[width=111pt]{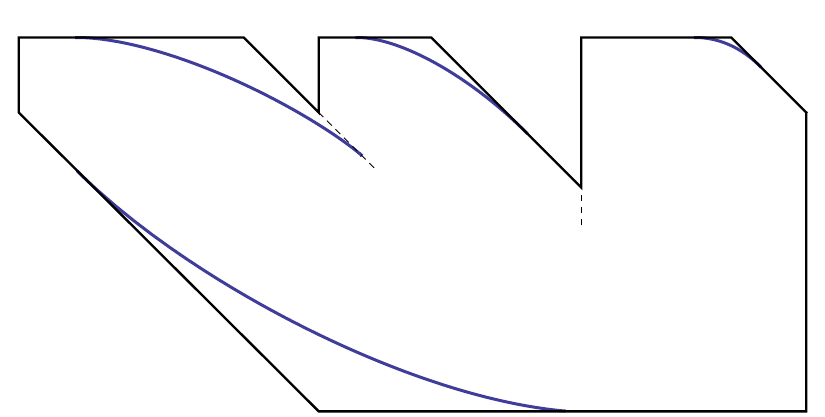}&
      \includegraphics[width=111pt]{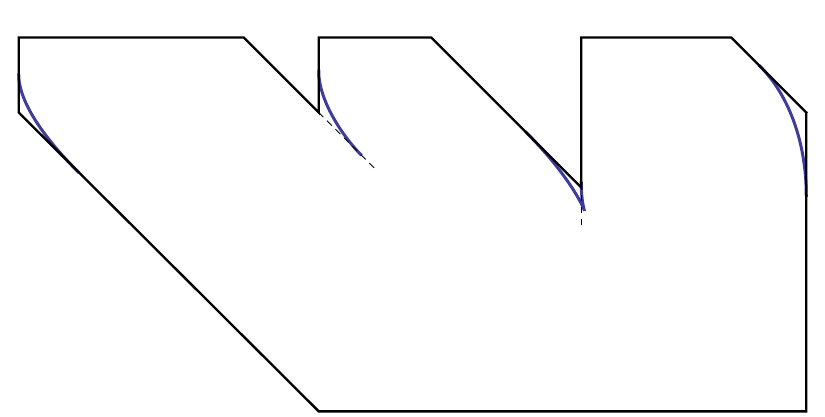}\\
      \includegraphics[width=111pt]{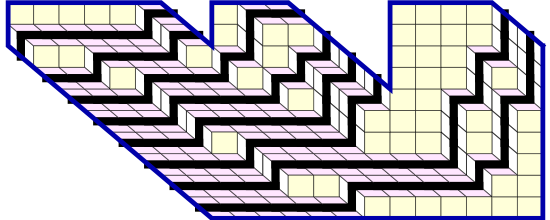}&
      \includegraphics[width=111pt]{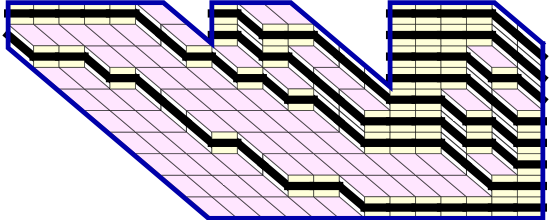}&
      \includegraphics[width=111pt]{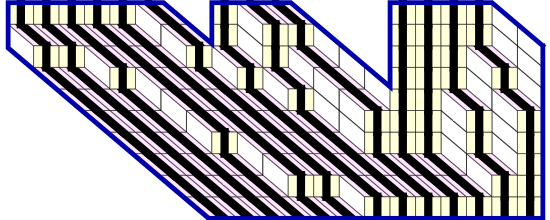}\\
      \hspace{50pt}\includegraphics[width=40pt]{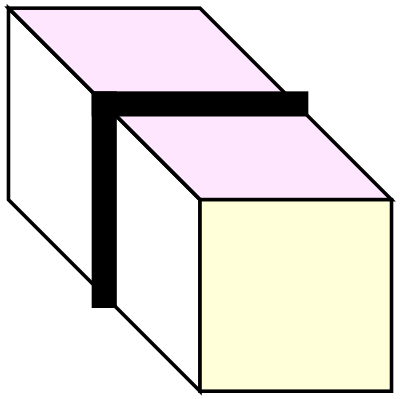}&
      \hspace{50pt}\includegraphics[width=40pt]{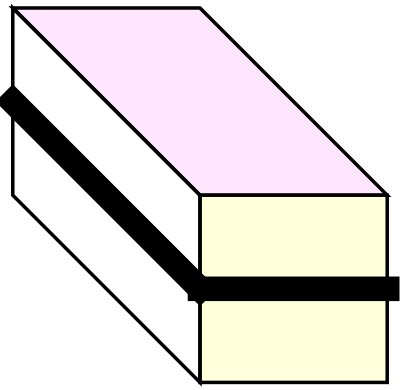}&
      \hspace{50pt}\includegraphics[width=40pt]{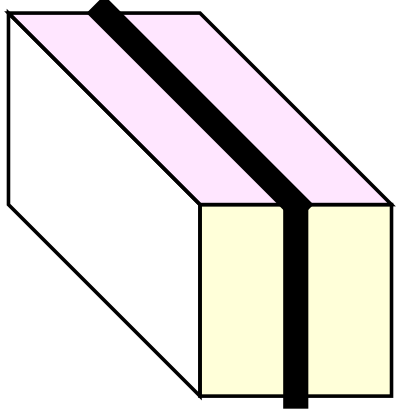}
    \end{tabular}
  \end{center}
  \caption{VH (top left), HD (top center), and DV (top right) directions of the frozen boundary
  and
  the corresponding nonintersecting path ensembles (center).
  The ensembles are 
  constructed by drawing line segments
  on two of the three types of lozenges (bottom).}
  \label{fig:Airy}
\end{figure} 
Close to a point at every direction of the frozen boundary, let us look at the corresponding ensemble of nonintersecting paths constructed using our random tiling (according to three cases on Fig.~\ref{fig:Airy}). We consider distribution of several such paths starting with the outer ones. That is, we look at first appearances of lozenges of two new types when the frozen region turns into the liquid one. 
Each of these first nonintersecting paths 
behaves as follows. 
A part of it inside the liquid region concentrates
around the corresponding part of the frozen boundary 
(when $N\to\infty$, after a global scaling by $N^{-1}$ in both directions).
The remaining parts of this nonintersecting path 
which are inside the frozen facets 
follow the boundary of the polygon. 

The Airy-type asymptotics we establish for these paths means that at every given position $(\chi,\eta)\in\partial\D$ (we assume that $(\chi,\eta)$ is neither a tangent nor a 
turning point)\footnote{In the context of random tilings,
``turning point'' sometimes stays for the point where two types of 
frozen regions meet (e.g., see
\cite{OkounkovReshetikhin2006RandomMatr}); thus, 
one can think that tangent point is a particular case of 
a turning point.}, 
the fluctuations of paths in direction tangent to $\partial\D$ are of order $N^{2/3}$, and in normal direction they have order $N^{1/3}$ (without a scaling). This can be restated in terms of our point processes $\Pp_{\Pc(N)}$ as follows (see Theorem \ref{thm:Airy_full} for a full and detailed statement):

\begin{maintheorem}[Asymptotics at the edge]\label{thm:Airy_intro}
  After a proper rescaling (and additional replacement of particles by holes and vice versa for the DV part of $\partial\D$), the correlation functions of the point process $\Pp_{\Pc(N)}$ converge to those of the Airy process. The latter are given by minors of the extended Airy kernel:
  \begin{align}\label{extended_Airy}
    \A(\tau_1,\si_1;\tau_2,\si_2)=\begin{cases}
      \int_{0}^{\infty}e^{-u(\tau_1-\tau_2)}Ai(\si_1+u)Ai(\si_2+u)du,&
      \mbox{if $\tau_1\ge\tau_2$};\\
      -\int_{-\infty}^{0}e^{-u(\tau_1-\tau_2)}Ai(\si_1+u)Ai(\si_2+u)du,&
      \mbox{if $\tau_1<\tau_2$}.
    \end{cases}
  \end{align}
  Here $Ai(x)$ is the Airy function 
  \begin{align}\label{Airy_function}
    Ai(x)=\frac{1}{2\pi}\int_{-\infty}^{\infty}e^{it^{3}/3+ixt}dt.
  \end{align}
\end{maintheorem}

We establish this Theorem also using saddle point analysis in a way similar to \cite{Okounkov2002}, \cite{Okounkov2005}, and \cite{BorodinKuan2007U}.

\medskip

% subsection asymptotics_at_the_edge_of_the_limit_shape (end)

% \subsection{Organization of the rest of the paper} % (fold)
% \label{sub:organization_of_the_rest_of_the_paper}

% % subsection organization_of_the_rest_of_the_paper (end)

The rest of the paper is devoted to proving Theorems \ref{thm:K_intro}, \ref{thm:bulk_intro}, and \ref{thm:Airy_intro}.

% section model_and_results (end)

\section{Combinatorics of the model and random Gelfand-Tsetlin schemes} % (fold)
\label{sec:combinatorics_of_the_model_and_random_gelfand_tsetlin_schemes}

We will argue in terms of measures on interlacing integer (particle) arrays $\X=\{\x_{j}^{m}\colon m=1,\ldots,N;\ j=1,\ldots,m\}$ (\S \ref{sub:particle_configurations_and_the_correlation_kernel}) with an arbitrary fixed top row. In this section we explain the connection of such arrays with Gelfand-Tsetlin schemes and discuss the partition function of our model (\ref{ZPc}) and of its $q$-deformation.

\subsection{Signatures and branching of Laurent-Schur polynomials} % (fold)
\label{sub:signatures_and_branching_of_schur_functions}

By a \emph{signature} of length $N$ we will mean a nonincreasing $N$-tuple of integers $\la=(\la_1\ge \ldots\ge\la_N)\in\Z^{N}$. Let $\GT_N$ denote the set of all signatures of length $N$. Because $\GT_N$ parametrizes irreducible representations of the unitary group $U(N)$ \cite{Weyl1946}, in the literature signatures are also referred to as \emph{highest weights}. The irreducible characters of $U(N)$ are given by the \emph{Laurent-Schur polynomials} \cite{Weyl1946}, \cite{Macdonald1995}
\begin{align}\label{schur_polynomial}
  s_\la(u_1,\ldots,u_N)=
  \frac{\det[u_i^{\la_j+N-j}]_{i,j=1}^{N}}{\det[u_i^{N-j}]_{i,j=1}^{N}},
  \qquad \la\in\GT_N.
\end{align}
Here $u_1,\ldots,u_N$ are eigenvalues of a unitary matrix $U\in U(N)$. Every $s_\la(u_1,\ldots,u_N)$ is a rational function which is symmetric in $u_1,\ldots,u_N$. More precisely, it is a homogeneous Laurent polynomial (i.e., a polynomial in $u_1^{\pm1},\ldots,u_N^{\pm1}$) of degree $|\la|:=\la_1+\ldots+\la_N\in\Z$ (this number is not necessary nonnegative). Note that the denominator in (\ref{schur_polynomial}) is simply the Vandermonde determinant
\begin{align}\label{Vandermonde}
  V(u_1,\ldots,u_N):=\det[u_i^{N-j}]_{i,j=1}^{N}
  =\prod\nolimits_{1\le i<j\le N}(u_i-u_j).
\end{align}

The \emph{branching} of the Laurent-Schur polynomials reflects the branching of the irreducible characters of unitary groups:
\begin{align*}
  s_\la(u_1,\ldots,u_{N-1};u_N=1)=
  \sum\nolimits_{\mu\in\GT_{N-1}\colon\mu\prec\la}
  s_{\mu}(u_1,\ldots,u_{N-1}),\qquad \la\in\GT_N,
\end{align*}
where $\mu\prec\la$ means \emph{interlacing} of signatures:
\begin{align*}
  \la_1\ge\mu_1\ge\la_2\ge \ldots\ge\la_{N-1}\ge\mu_{N-1}\ge\la_{N}.
\end{align*}
In fact, more can be said:
\begin{align}\label{cutting_out_uN}
  s_\la(u_1,\ldots,u_{N-1};u_N)=
  \sum\nolimits_{\mu\colon\mu\prec\la}
  s_{\mu}(u_1,\ldots,u_{N-1})u_N^{|\la|-|\mu|}.
\end{align}
Continuing expansion (\ref{cutting_out_uN}) for $u_{N-1},u_{N-2},\ldots,u_{1}$, we arrive at
\begin{definition}\label{def:GT_scheme}
  A \emph{Gelfand-Tsetlin scheme} of depth $N$ is an interlacing sequence of signatures
  \begin{align*}
    \varnothing\prec\la^{(1)}\prec\la^{(2)}
    \ldots\prec\la^{(N-1)}\prec\la^{(N)}.
  \end{align*}
  One can also view this as a triangular array of integers $\{\la_{j}^{(m)}\}\in\Z^{N(N+1)/2}$ with the following interlacing constraints:
  \begin{center}
    \includegraphics[width=135pt]{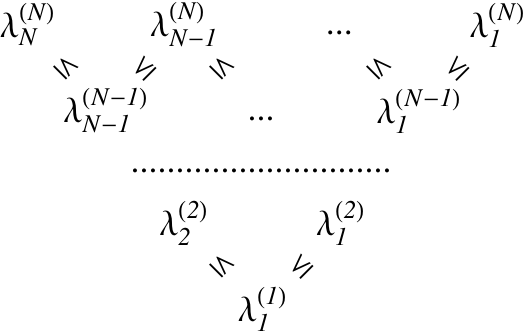}
  \end{center}
\end{definition}
From (\ref{cutting_out_uN}) we get the following combinatorial formula for the Laurent-Schur polynomials \cite[Ch. I, (5.12)]{Macdonald1995}:
\begin{equation}\label{schur_combinatorial_formula}
  s_\la(u_1,\ldots,u_N)=
  \sum_{\varnothing
  \prec\la^{(1)}\prec \ldots\prec\la^{(N)}=\la}
  u_1^{|\la^{(1)}|}u_2^{|\la^{(2)}|-|\la^{(1)}|}
  \ldots
  u_N^{|\la^{(N)}|-|\la^{(N-1)}|}.
\end{equation}
Here the sum is taken over all Gelfand-Tsetlin schemes with fixed top row $\la^{(N)}=\la$. We will use this identity in the present section to compute the partition function in our model and in its $q$-deformation.

% subsection signatures_and_branching_of_schur_functions (end)

\subsection{Volume of a Gelfand-Tsetlin scheme} % (fold)
\label{sub:volume_of_a_gelfand_tsetlin_scheme}

In any signature $\la=(\la_1,\ldots,\la_N)$, one can separate the positive and the negative components, and write it as a pair of partitions (= Young diagrams \cite[Ch. I.I]{Macdonald1995}) $\la=(\la^{+},\la^{-})$:
\begin{equation*}
  \la=(\la_1^{+},\ldots,\la_{\ell^{+}}^{+},0,\ldots,0,
  -\la_{\ell^{-}}^{-},\ldots,-\la_{1}^{-}),
\end{equation*}
where $\la_1^{\pm}\ge \ldots\ge \la_{\ell^{\pm}}^{\pm}>0$. Clearly, $|\la|=|\la^{+}|-|\la^{-}|$.

Each Gelfand-Tsetlin scheme of depth $N$ with $N$th row $\la^{(N)}=\la\in\GT_N$ can be viewed as a pair of sequences of Young diagrams $(\la^{(m)})^{\pm}$, $m=1,\ldots,N$, where each $(\la^{(m)})^{\pm}$ is obtained from $(\la^{(m-1)})^{\pm}$ by adding a horizontal strip \cite[Ch. I.1]{Macdonald1995}. We can make this into a pair of 3-dimensional Young diagrams (plane partitions) by placing $N-m$ boxes of dimensions $1\times1\times1$ on top of each 2-dimensional horizontal strip $(\la^{(m)})^{\pm}/(\la^{(m-1)})^{\pm}$. Let $\vol^\pm=\vol^{\pm}({\varnothing\prec\la^{(1)}\prec \ldots\prec\la^{(N)}})$ be volumes of the resulting two 3-dimensional Young diagrams.

More formally, 
\begin{align*}
  \vol^{\pm}:=\sum\nolimits_{m=1}^{N}(N-m)\big(|(\la^{(m)})^{\pm}|-|(\la^{(m-1)})^{\pm}|\big)
  =\sum\nolimits_{m=1}^{N-1}|(\la^{(m)})^{\pm}|
\end{align*}
(by agreement, $|(\la^{(0)})^{\pm}|=|\varnothing|=0$). Let us call the quantity
\begin{equation}\label{volume_of_GT_scheme}
  \vol:=\vol^+-\vol^-=
  \sum\nolimits_{m=1}^{N-1}|\la^{(m)}|=
  \sum\nolimits_{m=1}^{N}(N-m)\big(|\la^{(m)}|-|\la^{(m-1)}|\big)
\end{equation}
the \emph{volume} of the Gelfand-Tsetlin scheme $\varnothing\prec\la^{(1)}\prec \ldots\prec\la^{(N-1)}\prec\la^{(N)}=\la$.

% subsection volume_of_a_gelfand_tsetlin_scheme (end)

\subsection{Measure $q^{-\vol}$ on Gelfand-Tsetlin schemes} % (fold)
\label{sub:measure_q_vol_on_gelfand_tsetlin_schemes}

\emph{Everywhere in the paper we assume that $0<q<1$.}

Let us fix any signature $\nu\in\GT_N$. We consider the following family of probability measures (depending on the parameter $q$) on the set of all Gelfand-Tsetlin schemes of depth $N$ with fixed $N$th row $\nu$:
\begin{align}\label{measure_q_-vol}
  \q\Pp_{N,\nu}
  (\varnothing\prec\nu^{(1)}\prec 
  \ldots\prec\nu^{(N-1)}\prec\nu):=
  \frac{q^{-\vol({\varnothing\prec\nu^{(1)}\prec 
    \ldots\prec\nu^{(N-1)}\prec\nu})}}{\q Z_{N,\nu}}.
\end{align}
We will sometimes abbreviate $\q\Pp_{N,\nu}$ as $q^{-\vol}$.

From the combinatorial formula (\ref{schur_combinatorial_formula}), we readily conclude that the partition function is
\begin{align}\label{qZNnu}
  \q Z_{N,\nu}=s_\nu(q^{1-N},\ldots,q^{-1},1).
\end{align}
This specialization of Schur polynomials is known \cite[Ch.~I,~\S3, Ex.~1]{Macdonald1995}:
\begin{align}\label{qZNnu2}
  s_\nu(q^{1-N},\ldots,q^{-1},1)=
  q^{|\nu|(1-N)}
  \frac{V(q^{\nu_1-1},\ldots,q^{\nu_N-N})}
  {V(q^{-1},\ldots,q^{-N})},
\end{align}
where $V$ is the Vandermonde determinant (\ref{Vandermonde}).

% subsection measure_q_vol_on_gelfand_tsetlin_schemes (end)

\subsection{Uniform measure} % (fold)
\label{sub:uniform_measure}

Taking the $q\ua1$ limit of $\q\Pp_{N,\nu}$, we arrive at the uniform measure on all Gelfand-Tsetlin schemes of depth $N$ with the fixed $N$th row $\nu\in\GT_N$:
\begin{equation}\label{uniform_measure}
  \Pp_{N,\nu}(\varnothing\prec\nu^{(1)}\prec 
  \ldots\prec\nu^{(N-1)}\prec\nu):=
  \frac1{Z_{N,\nu}}.
\end{equation}
It is known that for $\Pp_{N,\nu}$ the partition function has the form (cf. (\ref{qZNnu})--(\ref{qZNnu2})) 
\begin{equation}\label{ZNnu}
  Z_{N,\nu}=s_\nu(\underbrace{1,\ldots,1}_{N})=
  \frac{V(\nu_1-1,\ldots,\nu_N-N)}
  {V(-1,\ldots,-N)}.
\end{equation}
From the branching rule (\ref{cutting_out_uN}) one can in fact conclude that $Z_{N,\nu}$ is equal to the dimension of an irreducible representation of $U(N)$ corresponding to the signature $\nu\in\GT_N$ \cite{Weyl1946}.

We need the $q$-deformation $\q\Pp_{N,\nu}$ 
in order to obtain the correlation kernel for the uniform
measure $\Pp_{N,\nu}$ because our technique does not straightforwardly 
apply to the latter,
cf. Remark \ref{rmk:q_needed} below.

% subsection uniform_measure (end)

\subsection{Connection to random tilings} % (fold)
\label{sub:connection_to_measures_on_tilings}

Gelfand-Tsetlin schemes $\varnothing\prec\nu^{(1)}\prec 
\ldots\prec\nu^{(N-1)}\prec\nu^{(N)}$ bijectively correspond to interlacing integer arrays $\X=\{\x_{j}^{m}\}$ (discussed in \S \ref{sub:particle_configurations_and_the_correlation_kernel} in connection with tilings) by means of a simple transformation:
\begin{align*}
  \x_{j}^{m}=\nu_j^{(m)}-j,\qquad m=1,\ldots,N,\qquad
  j=1,\ldots,m.
\end{align*}
The interlacing conditions for Gelfand-Tsetlin schemes of Definition \ref{def:GT_scheme} are clearly translated into the constraints (\ref{interlacing_intro}) which are satisfied by $\X$. 

We will think of the measures $\q\Pp_{N,\nu}$ and $\Pp_{N,\nu}$ introduced in \S\S \ref{sub:measure_q_vol_on_gelfand_tsetlin_schemes}--\ref{sub:uniform_measure} as measures on such interlacing arrays (and do not change the notation for the measures).

In the language of \S \ref{sub:lozenge_tilings_of_polygons}, the uniform measure $\Pp_{N,\nu}$ clearly becomes the corresponding uniform measure on all tilings of the polygon obtained by fixing the top row particles (see Fig.~\ref{fig:polygonal_region_tiling}) at positions $\x_j^{N}=\nu_{j}-j$, $j=1,\ldots,N$. In particular, formula (\ref{ZPc}) for the partition function of random tilings is the same as (\ref{ZNnu}).

\begin{figure}[htbp]
  \begin{tabular}{c}
    \includegraphics[width=140pt]{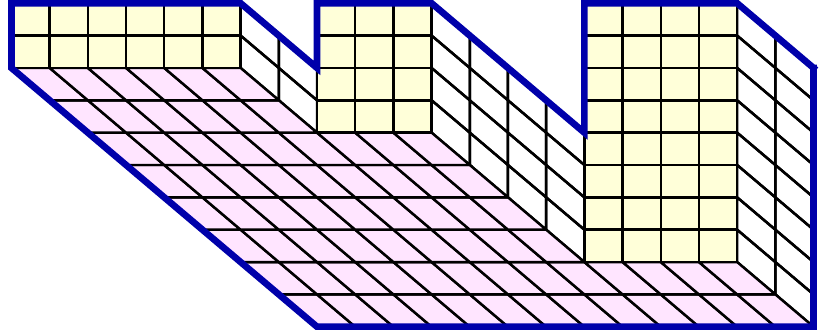}
  \end{tabular}
  \caption{Tiling with zero volume under the stepped surface.}
  \label{fig:empty_tiling}
\end{figure}

Now let us consider $q$-deformed measures. For random tilings, one way to define volume $vol$ under the corresponding stepped surface is to postulate that the tiling with zero volume is given on Fig.~\ref{fig:empty_tiling}. Then the surface corresponding to every other tiling is obtained from the one on Fig.~\ref{fig:empty_tiling} by adding $vol$ boxes of dimensions $1\times 1\times 1$. It is readily checked that thus defined volume $vol$ under the stepped surface is (up to a constant summand) equal to $(-\vol)$, where $\vol$ is the volume of the corresponding Gelfand-Tsetlin scheme~(\S\ref{sub:volume_of_a_gelfand_tsetlin_scheme}). This implies that weighting of tilings proportional to $q^{vol}$ (as mentioned in \S \ref{sub:particle_configurations_and_the_correlation_kernel}) is the same as considering the measure $\q\Pp_{N,\nu}$ (\ref{measure_q_-vol}) on interlacing arrays.

% subsection connection_to_measures_on_tilings (end)

% section combinatorics_of_the_model_and_random_gelfand_tsetlin_schemes (end)

\section{Correlation kernel for the measure $q^{-\vol}$} % (fold)
\label{sec:correlation_kernel_for_the_measure_q_vol_}

\subsection{Formulation of the result} % (fold)
\label{sub:formulation_of_the_result_q}

In this section we establish a $q$-deformation of Theorem \ref{thm:K_intro}, that is, compute the correlation kernel $\q K$ of the measure $\q\Pp_{N,\nu}$ on interlacing arrays with fixed $N$th row (\S \ref{sec:combinatorics_of_the_model_and_random_gelfand_tsetlin_schemes}). %The kernel $\q K$ is defined in the same way as in (\ref{correlation_kernel_intro}). 

To formulate the result, we will need the standard notation of the $q$-Pochhammer symbol
\begin{equation*}
  (a;q)_{k}:=\prod\nolimits_{i=0}^{k-1}(1-aq^{i})
  \qquad
  (k=1,2,\ldots),
  \qquad
  (a;q)_{0}:=1,
\end{equation*}
and the $q$-hypergeometric function
\begin{equation*}
  {}_{2}\phi_{1}(a,b;c\mid q;z):=\sum_{i=0}^{\infty}
  \frac{(a;q)_{i}(b;q)_{i}}{(c;q)_{i}}\frac{z^{i}}{(q;q)_{i}}.
\end{equation*}

\begin{theorem}\label{thm:q_kernel}
  The correlation kernel $\q K$ of the measure $\q\Pp_{N,\nu}$ on interlacing particle arrays $\{\x_j^m\colon m=1,\ldots,N,\; j=1,\ldots,m\}$ with fixed top row $\x_j^N=\nu_j-j$ ($j=1,\ldots,N$) is given for $1\le n_1\le N$, $1\le n_2\le N-1$, and $x_1,x_2\in\Z$ by
  \begin{align}\label{q_kernel_formula}
    \q K&(x_1,n_1;x_2,n_2)=
    -1_{n_2<n_1}1_{x_2\le x_1}q^{n_2(x_1-x_2)}
    \frac{(q^{x_1-x_2+1};q)_{n_1-n_2-1}}{(q;q)_{n_1-n_2-1}}
    \\&
    \nonumber
    +\frac{(q^{N-1};q^{-1})_{N-n_1}}{(2\pi\i)^2}
    \oint\limits_{\Ga_q(x_2-x_1)}dz
    \oint\limits_{\ga(\infty)} \frac{dw}{w}
    \frac{q^{n_2(x_1-x_2)}z^{n_2}}{w-z}
    \times\\&\quad\times
    {}_2\phi_1(q^{-1},q^{n_1-1};q^{N-1}\mid q^{-1};w^{-1})
    \frac{(zq^{1-x_2+x_1};q)_{N-n_2-1}}{(q;q)_{N-n_2-1}}
    \prod_{r=1}^{N}\frac{w-q^{\nu_r-r-x_1}}{z-q^{\nu_r-r-x_1}}.
    \nonumber
  \end{align}
  The contours in $z$ and $w$ are counter-clockwise and do not intersect. The contour $\Ga_q(x_2-x_1)$ in $z$ encircles the points $q^{x_2-x_1},q^{x_2-x_1+1},\ldots,q^{\nu_1-1-x_1}$ in $z$ and only them (i.e., does not contain $q^{x_2-x_1-1},q^{x_2-x_1-2},\ldots$). 
  The contour $\ga(\infty)$ in $w$ 
  must be sufficiently large and contain
  $\Ga_q(x_2-x_1)$.
\end{theorem}

% subsection formulation_of_the_result (end)

\subsection{Eynard-Mehta-type theorem} % (fold)
\label{sub:eynard_mehta_type_theorem}

In the proof of Theorem \ref{thm:q_kernel} we follow an Eynard-Mehta type formalism which allows the number of particles to vary. 
Let us recall (see Theorem \ref{thm:EM} below) 
the precise statement 
which we will then apply. The present subsection 
is essentially a citation from
\cite[\S4]{Borodin2009},
see also \cite{BorFerr08push}, \cite[\S4.4]{ForrNord2008}, \cite[Lemma 3.4]{BorodinFPS2007}. The application of the Eynard-Mehta-type theorem 
to our situation is in some aspects similar to \cite{BorodinKuan2007U}. 

Let $\mathfrak{X}^{1},\ldots,\mathfrak{X}^{N}$ be discrete sets. Denote 
$\mathfrak{X}:=\mathfrak{X}^{1}\sqcup \ldots \sqcup \mathfrak{X}^{N}$.
Let 
\begin{align}\nonumber
  \varphi_n(\cdot,\cdot)\colon\mathfrak{X}^{n-1}\times 
  \mathfrak{X}^{n}&\to\C,\qquad n=2,\ldots,N;
  \\
  \varphi_n(virt,\cdot)\colon 
  \mathfrak{X}^{n}&\to\C,\qquad n=1,\ldots,N;
  \label{phi_EM}
  \\
  \psi_j(\cdot\mid N)\colon \mathfrak{X}^{N}&\to\C,\qquad
  j=1,\ldots,N,
  \nonumber
\end{align}
be arbitrary functions. Assign the following 
weight
to any configuration
$\X\in\mathrm{Conf}(\mathfrak{X})$
for which $|\X\cup\mathfrak{X}^{n}|=n$
(i.e., $\X$ contains precisely $n$ points
in each of the parts $\mathfrak{X}^{n}$; denote these points by 
$\x^{n}_{1},\ldots,\x^{n}_{n}$):
\begin{align}\label{EM_W}
  W(\X):=\det[\psi_i(\x_j^N\mid N)]_{i,j=1}^{N}
  \prod_{n=1}^{N}
  \det [\varphi_n(\x_i^{n-1},\x_{j}^{n})]_{i,j=1}^{n},
\end{align}
otherwise $W(\X)=0$. By agreement, $\x^{m}_{m+1}=virt$ for each $m=0,\ldots,N-1$.
Assume that the partition function $\sum_{\X\in\mathrm{Conf}(\mathfrak{X})}
W(\X)$ is finite and does not vanish. Normalizing the weights, one obtains a 
(generally speaking, complex valued) point process on~$\mathfrak{X}$.

We need further notation. For any functions $f(x,y), g(x,y)$, 
and $h(x)$, define convolutions as follows:
\begin{align*}
  (f*g)(x,y):=\sum_{z}f(x,z)g(z,y),\qquad
  (g*h)(x):=\sum_{y}g(x,y)h(y),
\end{align*}
where the sums are taken over all values of the
intermediate variables $z$ and $y$, respectively.
Denote (for $n_1,n_2,m=1,\ldots,N$)
\begin{align}
  \varphi^{(n_1,n_2)}(x,y)&:=
  1_{n_1<n_2}(\varphi_{n_1+1}*\ldots*\varphi_{n_2})(x,y);
  \label{phi_n1_n2}
  \\
  (\varphi_{n_1}*\varphi^{(n_1,n_2)})(virt,y)&:=1_{n_1\le n_2}
  (\varphi_{n_1}*\varphi_{n_1+1}*\ldots*\varphi_{n_2})(virt,y);
  \label{phi_n1_n2_virt}\\
  \psi_i(x\mid m)&:=
  (\varphi_{m+1}*\ldots*\varphi_{N}*\psi_i(\cdot\mid N))(x),
  \label{psi_x_m}
\end{align}
and also define the ``Gram matrix'' as
\begin{align}\label{Gram_matrix}
  G_{ml}:=(\varphi_m*\ldots*\varphi_N*\psi_l(\cdot\mid N))(virt),
  \qquad
  m,l=1,\ldots,N.
\end{align}
We assume that all the sums in \eqref{phi_n1_n2}--\eqref{Gram_matrix}
converge.
\begin{theorem}\label{thm:EM}
  The point process on $\mathfrak{X}$
  defined by the weight $W$ 
  \eqref{EM_W} is determinantal
  (cf. Definition \ref{def:correlation_functions_intro}). 
  Its correlation kernel
  is given by
  \begin{align}\label{qK_Eynard-Mehta}
    \q K(x_1,n_1;x_2,n_2)&=
    -\varphi^{(n_2,n_1)}(x_2,x_1)
    \\&\qquad+
    \sum_{i=1}^{n_1}\sum_{j=1}^{N}
    [G^{-t}]_{ij}
    (\varphi_{i}*\varphi^{(i,n_1)})(virt,x_1)
    \cdot\psi_j(x_2\mid n_2),
    \nonumber
  \end{align}
  where $G^{-t}$ is the inverse transpose of 
  the matrix $G$ \eqref{Gram_matrix}.
\end{theorem}
Formula \eqref{qK_Eynard-Mehta}
for the correlation kernel becomes especially
useful (e.g., for asymptotic analysis)
if one manages to write the inverse of the ``Gram matrix'' $G$
in an explicit form. This has to be done separately for every particular 
point process as there is no general recipe for explicit
inversion of a matrix. 

For our measure $q^{-\vol}$ we manage to choose the representation
\eqref{EM_W} such that the ``Gram matrix'' becomes diagonal. 
The rest of this section is devoted to proving Theorem 
\ref{thm:q_kernel} via an application of Theorem \ref{thm:EM}.

% subsection eynard_mehta_type_theorem (end)

\subsection{Writing the measure $q^{-\vol}$ as a product of determinants} % (fold)
\label{sub:writing_the_measure_q_vol_as_a_product_of_determinants}

The first step in applying the formalism of \S \ref{sub:eynard_mehta_type_theorem} 
is to write the measure $\q\Pp_{N,\nu}$ on interlacing arrays $\{\x_j^m\}$ as a product of determinants.
We set $\mathfrak{X}^{1}=\ldots=\mathfrak{X}^{N}=\Z$;
we understand that each $\mathfrak{X}^{m}$ represents the 
$m$th row of our interlacing array $\X$.
We allow the number of particles in $\X^{1},\ldots,\X^{N}$ 
to vary, 
and add to every $m$th row of the array $\x_m^m<\ldots<\x_1^m$ 
a virtual particle $\x_{m+1}^{m}=virt$. Informally, one can think that $virt=-\infty$.

Let us denote for $m=1,\ldots,N$ (cf. \eqref{phi_EM}):
\begin{align}\label{phi_m}
  \varphi_m(x,y):=q^{(m-1)(y-x)}1_{x\le y}+q^{(m-1)y}1_{x=virt},\qquad
  x\in\{virt\}\cup\Z,\quad y\in\Z.
\end{align}
The following lemma is well-known (e.g., see \cite[\S3]{BorodinKuan2007U}):
\begin{lemma}\label{lemma:interlacing}
  For any array of integers $\{\x_j^m\colon m=1,\ldots,N;\; j=1,\ldots,m\}$, we have (with the agreement that $\x_{m+1}^{m}=virt$)
  \begin{align*}
    \prod\nolimits_{n=1}^{N}
    \det [\varphi_n(\x_i^{n-1},\x_{j}^{n})]_{i,j=1}^{n}=
    q^{-|\x^{1}|-\ldots-|\x^{N-1}|}q^{(N-1)|\x^{N}|}
  \end{align*}
  if the array is interlacing (i.e., satisfies (\ref{interlacing_intro})), and $0$ otherwise. Here and below we denote $|\x^{m}|:=\x_{1}^{m}+\ldots+\x_{m}^{m}$.
\end{lemma}

We need to define one more object which we will use:

\begin{definition}
  Let $\q\V(\nu)$ denote the $N\times N$ Vandermonde matrix $[(q^{\nu_i-i})^{N-j}]_{i,j=1}^{N}$. Clearly, $\det[\q\V(\nu)]=V(q^{\nu_1-1},\ldots,q^{\nu_N-N})$, where $V$ is the Vandermonde determinant (\ref{Vandermonde}).

  Let $\q\V(\nu)^{-1}$ be the inverse of that Vandermonde matrix. 
  Define the following functions in $x\in\Z$ (cf. \eqref{phi_EM}):
  \begin{align}\label{psi_i}
    \psi_i(x\mid N):=\sum\nolimits_{j=1}^{N}[\q\V(\nu)^{-1}]_{ij}1_{x=\nu_j-j}.
  \end{align}
\end{definition}
\begin{proposition}[Measure $q^{-\vol}$ as a product of determinants]\label{prop:qPp_product_of_determinants}
  The measure $\q\Pp_{N,\nu}$ on arrays of integers $\X=\{\x_j^m\}\in\Z^{N(N+1)/2}$ (with $\x_{m+1}^{m}=virt$) can be written in the following form:
  \begin{align}
    \label{qPp_product_of_determinants}
    \q\Pp_{N,\nu}(\X)={}&
    q^{\frac16N(N+1)(2N+1)}
    V(q^{-1},\ldots,q^{-N})
    \times\\&\times
    \nonumber
    \det[\psi_i(\x_j^N\mid N)]_{i,j=1}^{N}
    \prod_{n=1}^{N}
    \det [\varphi_n(\x_i^{n-1},\x_{j}^{n})]_{i,j=1}^{n}.
    \nonumber
  \end{align}
\end{proposition}
The proposition in particular claims that if the array is not interlacing or its top row $\x_1^N,\ldots,\x_N^N$ does not coincide with $(\nu_1-1,\ldots,\nu_N-N)$, then the measure assigned to such array by the right-hand side of (\ref{qPp_product_of_determinants}) is zero.
\begin{proof}
  By Lemma \ref{lemma:interlacing}, the product of determinants over $n=1,\ldots,N$ in the right-hand side of (\ref{qPp_product_of_determinants}) imposes the interlacing constraints on $\{\x_{j}^{m}\}$. It also provides the desired weighting proportional to $q^{-\vol}$ because $|\x^{m}|=|\nu^{(m)}|-\binom{m+1}2$ (see \S\ref{sub:connection_to_measures_on_tilings}). 

  Then an obvious (but crucial; see Remark \ref{rmk:why_works} below) observation is that for any integers $y_1>\ldots>y_N$ one has
  \begin{align*}
    \det[\psi_i(y_j\mid N)]_{i,j=1}^{N}=
    \frac{1_{y_1=\nu_1-1}\ldots 1_{y_N=\nu_N-N}}
    {V(q^{\nu_1-1},\ldots,q^{\nu_N-N})}.
  \end{align*}
  This allows to take the factor $\frac{1}{V(q^{\nu_1-1},\ldots,q^{\nu_N-N})}$ from the partition function $\q Z_{N,\nu}$ of $\q\Pp_{N,\nu}$ (see \S\ref{sub:measure_q_vol_on_gelfand_tsetlin_schemes}) and make it into the determinant $\det[\psi_i(\x_j^N\mid N)]_{i,j=1}^{N}$. This imposes the desired condition that the top row of the array $\{\x^{m}_{j}\}$ is fixed: $\x_{j}^{N}=\nu_j-j$, $j=1,\ldots,N$.

  The factor $V(q^{-1},\ldots,q^{-N})$ also comes from the partition function $\q Z_{N,\nu}$. Finally, it is readily checked that the power of $q$ in front in (\ref{qPp_product_of_determinants}) is correct. This concludes the proof.
\end{proof}

% subsection writing_the_measure_q_vol_as_a_product_of_determinants (end)

\subsection{Convolutions and the ``Gram matrix''} % (fold)
\label{sub:convolutions_and_the_gram_matrix}

In \S\S \ref{sub:convolutions_and_the_gram_matrix}--\ref{sub:functions_psi_i_xmid_k_} we compute more quantities which are needed for the Eynard-Mehta type formalism (\S \ref{sub:eynard_mehta_type_theorem}).

Formula (\ref{qK_Eynard-Mehta}) for the kernel involves certain convolutions of the functions $\varphi_n$ (\ref{phi_m}). The computation of these convolutions is done similarly to \cite[\S3]{BorodinKuan2007U}.
\begin{lemma}\label{lemma:convolutions}
  We have for $x,y\in\Z$ (cf. \eqref{phi_n1_n2}):
  \begin{align*}
    \varphi^{(n_1,n_2)}(x,y)=
    1_{n_1<n_2}1_{x\le y}q^{n_1(y-x)}
    \frac{(q^{y-x+1};q)_{n_2-n_1-1}}{(q;q)_{n_2-n_1-1}}.
  \end{align*}
\end{lemma}
\begin{proof}
  Denote $F_n(z):=\frac1{1-zq^{n-1}}$. For any $x,y\in\Z$, we clearly have
  \begin{equation*}
    \varphi_n(x,y)=\frac1{2\pi\i}\oint_{|z|=1}\frac{dz}{z^{y-x+1}}F_n(z).
  \end{equation*}
  Then the convolutions take the form
  \begin{align}\nonumber
    \varphi^{(n_1,n_2)}(x,y)&=
    \frac1{2\pi\i}\oint_{|z|=1}\frac{dz}{z^{y-x+1}}F_{n_1+1}(z)\ldots F_{n_2}(z)
    \\&=
    1_{n_1<n_2}
    h_{y-x}(q^{n_1},q^{n_1+1},\ldots,q^{n_2-1}),
    \label{convolutions_proof}
  \end{align}
  where $h_m$ ($m\in\Z$) are the complete homogeneous symmetric polynomials, $h_0=1$, $h_{-1}=h_{-2}=\ldots=0$ \cite[Ch. I.2]{Macdonald1995}. We have used their generating function
  \begin{align}\label{h_m_genfunc}
    \sum_{m=0}^{\infty}
    h_m(y_1,\ldots,y_L)t^{m}
    =\prod_{r=1}^{L}\frac{1}{1-y_rt}.
  \end{align}
  Since the $h_m$'s are particular cases of the Schur polynomials, $h_m=1_{m\ge0}s_{(m)}$ \cite[Ch. I.3]{Macdonald1995}, we can write using (\ref{qZNnu2}) for every $L\ge1$:
  \begin{align*}
    h_m(1,q,\ldots,q^{L-1})&=
    1_{m\ge0}\frac{V(q^{m-1},q^{-2},\ldots,q^{-L})}{V(q^{-1},\ldots,q^{-L})}
    \\&=
    1_{m\ge0}\prod_{r=1}^{L-1}\frac{1-q^{m+r}}{1-q^{r}}\\&=
    1_{m\ge0}\frac{(q^{m+1};q)_{L-1}}{(q;q)_{L-1}}.
  \end{align*}
  Since $h_m$ is a homogeneous symmetric polynomial of degree $m$, we see that the above expression for $h_m(1,q,\ldots,q^{L-1})$ together with (\ref{convolutions_proof}) implies the claim.
\end{proof}
\begin{lemma}\label{lemma:convolutions2}
  We have for $y\in\Z$ (cf. \eqref{phi_n1_n2_virt}):
  \begin{align*}
    (\varphi_{n_1}*\varphi^{(n_1,n_2)})(virt,y)=
    1_{n_1\le n_2}
    q^{(n_1-1)y}\prod_{r=1}^{n_2-n_1}(1-q^{r})^{-1}.
  \end{align*}
  (of course, for $n_1=n_2$ the empty product is interpreted as $1$).
\end{lemma}
\begin{proof}
  For $n_1=n_2$ the claim is trivial. For $n_1>n_2$ from the proof of Lemma~\ref{lemma:convolutions} we have:
  \begin{align*}
    \sum_{x\in\Z}\varphi_{n_1}(virt,x)\varphi^{(n_1,n_2)}(x,y)
    &=1_{n_1<n_2}
    \sum_{x\in\Z}q^{(n_1-1)x}h_{y-x}(q^{n_1},\ldots,q^{n_2-1})
    \\&=1_{n_1<n_2}q^{(n_1-1)y}
    \sum_{x\in\Z}h_{y-x}(q,\ldots,q^{n_2-n_1})
    \\&=
    \nonumber
    1_{n_1< n_2}
    \frac{q^{(n_1-1)y}}{(1-q)\ldots(1-q^{n_2-n_1})}.
  \end{align*}
  In the last summation we used (\ref{h_m_genfunc}) with $t=1$. Note that in this case the series converges because $0<q<1$.
\end{proof}
\begin{remark}\label{rmk:q_needed}
  The need for a converging series in the proof of Lemma \ref{lemma:convolutions2} is the reason why one cannot directly use the Eynard-Mehta type formalism to compute the correlation kernel $K$ of the uniform measure $\Pp_{N,\nu}$. Thus, to get a formula for $q=1$, we first need to deal with the case $0<q<1$, and then take the limit $q\ua 1$. In fact, a similar problem occurs in computations in \cite[\S3]{BorodinKuan2007U}.
\end{remark}
\begin{lemma}\label{lemma:Gram}
  We have (cf. \eqref{Gram_matrix})
  \begin{align*}
    G_{ml}=1_{l=N-m+1}\cdot\prod_{r=1}^{N-m}(1-q^{r})^{-1}.
  \end{align*}
\end{lemma}
\begin{proof}
  Let us write (by Lemma \ref{lemma:convolutions2})
  \begin{align*}
    (\varphi_m*\ldots\varphi_N*\psi_l(\cdot\mid N))(virt)&=
    \prod_{r=1}^{N-m}(1-q^{r})^{-1}
    \sum_{y\in\Z}
    q^{(m-1)y}\psi_l(y\mid N).
  \end{align*}
  Using the definition of $\psi_i$ (\ref{psi_i}), we can readily simplify the sum over~$y$ above:
  \begin{align*}
    \sum_{y\in\Z}
    q^{(m-1)y}\psi_l(y\mid N)&=
    \sum\nolimits_{j=1}^{N}
    [\q\V(\nu)^{-1}]_{l,j}\cdot
    q^{(\nu_j-j)(m-1)}
    \\&=
    \sum\nolimits_{j=1}^{N}
    [\q\V(\nu)^{-1}]_{l,j}\cdot
    [\q\V(\nu)]_{j,N-m+1}
    \\&=
    1_{l=N-m+1}.
  \end{align*}
  This concludes the proof.
\end{proof}

\begin{remark}\label{rmk:why_works}
  We see that the ``Gram matrix'' $G=[G_{ml}]_{m,l=1}^{N}$ turns out to be diagonal. 
  This means that we readily know its inverse which enters (\ref{qK_Eynard-Mehta}). 
  In fact, in various determinantal models inverting the matrix $G$ is the main technical difficulty in obtaining an explicit formula for the correlation kernel. 
  For our measure $\q\Pp_{N,\nu}$, it is the use of the functions $\psi_i(x\mid N)$ in 
  \eqref{qPp_product_of_determinants} that allows to obtain a diagonal ``Gram matrix''.
\end{remark}

% subsection convolutions_and_the_gram_matrix (end)

\subsection{Inverse Vandermonde matrix and double contour integrals} % (fold)
\label{sub:inverse_vandermonde_matrix_and_double_contour_integrals}

Here we show how the elements of the inverse Vandermonde matrix can be written as double contour integrals. In fact, this is the reason of appearance of the double contour integral expression for $\q K$ in (\ref{q_kernel_formula}).
\begin{proposition}\label{prop:V_inverse_ointoint}
  For every $i,j=1,\ldots,N$, we have
  \begin{align}
    [\q\V(\nu)^{-1}]_{ij}=
    \frac{1}{(2\pi\i)^{2}}
    \oint\limits_{\ga(q^{\nu_j-j})}dz
    \oint\limits_{\ga(\infty)}
    \frac{dw}{w^{N+1-i}}\frac{1}{w-z}
    \prod_{r=1}^{N}\frac{w-q^{\nu_r-r}}{z-q^{\nu_r-r}}.
    \label{V_inverse_ointoint}
  \end{align}
  The contour $\ga(q^{\nu_j-j})$ in $z$ is counter-clockwise, sufficiently small, and goes around $q^{\nu_j-j}$ (it does not include any other poles in $z$). The counter-clockwise contour $\ga(\infty)$ in $w$ contains $\ga(q^{\nu_j-j})$ (without intersecting it) and is sufficiently large.
\end{proposition}
\begin{proof}
  Using the elementary symmetric polynomials $e_m$ (e.g., see \cite[Ch. I.2]{Macdonald1995} for definition), one can write for $i,j=1,\ldots,N$:
  \begin{align}\label{V_inverse_e_i}
    [\q\V(\nu)^{-1}]_{ij}=(-1)^{i-1}
    \frac{e_{i-1}(q^{\nu_1-1},\ldots,\widehat{q^{\nu_j-j}},\ldots,q^{\nu_N-N})}
    {\prod_{r\ne j}(q^{\nu_j-j}-q^{\nu_r-r})}
  \end{align}
  (hat means the absence of $q^{\nu_j-j}$). Indeed, this formula follows from the fact that every cofactor of the Vandermonde matrix $\q\V(\nu)$ can be expressed through the numerator in the right-hand side of (\ref{schur_polynomial}) with $\la$ of the form $(1^{m})=(1,\ldots,1)$ ($m$~ones). It remains to recall that $s_{(1^{m})}=e_m$, the $m$th elementary symmetric polynomial \cite[Ch. I.3]{Macdonald1995}.

  Let us perform the integration in the right-hand side of (\ref{V_inverse_ointoint}). We first integrate over $z$, and obtain:
  \begin{align*}
    \frac1{(2\pi\i)^2}&
    \oint\limits_{\ga(q^{\nu_j-j})}dz
    \oint\limits_{\ga(\infty)} \frac{dw}{w^{N+1-i}}
    \frac1{w-z}
    \prod_{r=1}^{N}\frac{w-q^{\nu_r-r}}{z-q^{\nu_r-r}}
    \\&=
    \Big(\prod_{r\ne j}
    \frac1{q^{\nu_j-j}-q^{\nu_r-r}}\Big)
    \frac1{2\pi\i}
    \oint\limits_{\ga(\infty)} \frac{dw}{w^{N+1-i}}
    \prod_{r\ne j}(w-q^{\nu_r-r}).
  \end{align*}
  Next, the integral in $w$ amounts to taking the coefficient by $w^{N-i}$ in $\prod_{r\ne j}(w-q^{\nu_r-r})$. Recalling the generating series for the elementary symmetric polynomials \cite[Ch. I.2]{Macdonald1995}
  \begin{align*}
    \sum_{m=0}^{L}
    e_m(y_1,\ldots,y_L)t^{m}
    =\prod_{r=1}^{L}(1+y_rt),
  \end{align*}
  we see that the right-hand side of (\ref{V_inverse_ointoint}) is equal to that of (\ref{V_inverse_e_i}). This concludes the proof.
\end{proof}

% subsection inverse_vandermonde_matrix_and_double_contour_integrals (end)

\subsection{Functions $\psi_i(x\mid m)$} % (fold)
\label{sub:functions_psi_i_xmid_k_}

Using the result of Proposition \ref{prop:V_inverse_ointoint}, we can compute the last ingredient used in the Eynard-Mehta type formula (\ref{qK_Eynard-Mehta}): 
\begin{proposition}\label{prop:psi_i_x_m}
  We have for $m=0,\ldots,N-1$ (cf. \eqref{psi_x_m}):
  \begin{align*}
    \psi_i(x\mid m)=
    \frac1{(2\pi\i)^2}
    \oint\limits_{\Ga_q(x)}dz
    \oint\limits_{\ga(\infty)} \frac{dw}{w^{N+1-i}}
    \frac{q^{-mx}z^{m}}{w-z}
    \frac{(zq^{1-x};q)_{N-m-1}}{(q;q)_{N-m-1}}
    \prod_{r=1}^{N}\frac{w-q^{\nu_r-r}}{z-q^{\nu_r-r}}.
  \end{align*}
  The contours in $z$ and $w$ are counter-clockwise and do not intersect. The contour $\Ga_q(x)$ in $z$ encircles the points $q^{x},q^{x+1},\ldots$ (and not the points $q^{x-1},q^{x-2},\ldots$). The contour $\ga(\infty)$ in $w$ contains $\Ga_q(x)$ and is sufficiently large.
\end{proposition}
\begin{proof}
  Using Lemma \ref{lemma:convolutions} and the definition of $\psi_i(\cdot\mid N)$ (\ref{psi_i}), we can write
  \begin{align*}
    \psi_i(x\mid m)&
    =\sum_{y\colon y\ge x}
    q^{m(y-x)}\frac{(q^{y-x+1};q)_{N-m-1}}{(q;q)_{N-m-1}}
    \psi_i(y\mid N)
    \\&=
    \sum_{j\colon \nu_j-j\ge x}
    q^{m(\nu_j-j-x)}\frac{(q^{\nu_j-j-x+1};q)_{N-m-1}}{(q;q)_{N-m-1}}
    [\q\V(\nu)^{-1}]_{ij}.
  \end{align*}
  Using the double contour integral formula for the inverse Vandermonde matrix (Proposition \ref{prop:V_inverse_ointoint}), we readily see that the summation over $j\colon \nu_j-j\ge x$ turns into the integration over $z\in\Ga_q(x)$. This concludes the proof.
\end{proof}

\begin{remark}\label{rmk:no_double_cont}
  Observe that the double contour integral expression for $\psi_i(x\mid m)$ of Proposition~\ref{prop:psi_i_x_m} makes no sense for $m=N$. However, there is another similar expression for $\psi_i(x\mid N)$ which directly follows from (\ref{psi_i}) and Proposition \ref{prop:V_inverse_ointoint}:
  \begin{align*}
    \psi_i(x\mid N)=
    \frac{1}{(2\pi\i)^{2}}
    \oint\limits_{\ga(q^{x})}dz
    \oint\limits_{\ga(\infty)}
    \frac{dw}{w^{N+1-i}}\frac{1}{w-z}
    \prod_{r=1}^{N}\frac{w-q^{\nu_r-r}}{z-q^{\nu_r-r}}.
  \end{align*}
  Note that this expression does not use the global contour $\Ga_q(x)$, and here $z$ belongs to a small contour around $q^{x}$ instead.
\end{remark}

% subsection functions_psi_i_xmid_k_ (end)

\subsection{Getting formula for the kernel $\q K$} % (fold)
\label{sub:getting_formula_for_the_kernel_q_k_}

The Eynard-Mehta type formalism (\S \ref{sub:eynard_mehta_type_theorem}) 
produces a formula 
\eqref{qK_Eynard-Mehta}
for the correlation kernel $\q K$ 
based on the representation of the measure $\q\Pp_{N,\nu}$ as a product of determinants (Proposition \ref{prop:qPp_product_of_determinants}). 
The inverse transpose of the ``Gram matrix'' looks as (see Lemma \ref{lemma:Gram}):
\begin{align*}
  [G^{-t}]_{ij}=1_{j=N-i+1}\cdot
  \prod_{r=1}^{N-i}(1-q^{r}).
\end{align*}
Using this fact and Lemmas \ref{lemma:convolutions} and \ref{lemma:convolutions2}, we can rewrite
\eqref{qK_Eynard-Mehta} as
\begin{align*}
  \q K(x_1,n_1;x_2,n_2)&=-
  1_{n_2<n_1}1_{x_2\le x_1}q^{n_2(x_1-x_2)}
  \frac{(q^{x_1-x_2+1};q)_{n_1-n_2-1}}{(q;q)_{n_1-n_2-1}}
  \\&\qquad +
  \sum_{i=1}^{n_1}
  \Big(\prod_{r=n_1+1}^{N}(1-q^{r-i})\Big)
  q^{(i-1)x_1}
  \psi_{N+1-i}(x_2\mid n_2).
\end{align*}
Observe that the function $\psi_{N+1-i}(x_2\mid n_2)$ depends on $i$ only via the term $w^{-i}$ under the integral (Proposition \ref{prop:psi_i_x_m}). It follows that we can compress the sum over $i$ into a $q$-hypergeometric function as follows:
\begin{align*}
  \sum_{i=1}^{n_1}
  \Big(\prod_{r=n_1+1}^{N}&(1-q^{r-i})\Big)
  (q^{x_1}w^{-1})^{i-1}\\&=
  (q^{N-1};q^{-1})_{N-n_1}\cdot
  {}_2\phi_1(q^{-1},q^{n_1-1};q^{N-1}\mid q^{-1};q^{x_1}w^{-1}).
\end{align*}

Plugging in the expression for $\psi_{N+1-i}(x_2\mid n_2)$ (Proposition \ref{prop:psi_i_x_m}) without the factor $w^{1-i}$ which went inside ${}_2\phi_1$, we get the following formula for the kernel:
\begin{align*}
  \q K&(x_1,n_1;x_2,n_2)=-
  1_{n_2<n_1}1_{x_2\le x_1}q^{n_2(x_1-x_2)}
  \frac{(q^{x_1-x_2+1};q)_{n_1-n_2-1}}{(q;q)_{n_1-n_2-1}}
  \\&\quad+
  \frac{(q^{N-1};q^{-1})_{N-n_1}}{(2\pi\i)^2}
  \oint\limits_{\Ga_q(x_2)}dz
  \oint\limits_{\ga(\infty)} \frac{dw}{w}
  \frac{q^{-n_2x_2}z^{n_2}}{w-z}
  \times\\&\qquad\times
  {}_2\phi_1(q^{-1},q^{n_1-1};q^{N-1}\mid q^{-1};q^{x_1}w^{-1})
  \frac{(zq^{1-x_2};q)_{N-n_2-1}}{(q;q)_{N-n_2-1}}
  \prod_{r=1}^{N}\frac{w-q^{\nu_r-r}}{z-q^{\nu_r-r}}.
\end{align*}
Here it is essential that $n_2\le N-1$. For $n_2=N$, one could write another formula for the kernel $\q K(x_1,n_1;x_2,n_2)$ based on Remark \ref{rmk:no_double_cont}, but since the top row of the array $\X=\{\x_{j}^{m}\}$ (corresponding to $m=N$) is fixed, we do not need such a formula.

Finally, changing the variables as $\tilde w=wq^{-x_1}$ and $\tilde z=zq^{-x_1}$ and renaming them back to $w,z$, we see that Theorem \ref{thm:q_kernel} is established.

% subsection getting_formula_for_the_kernel_q_k_ (end)

% section correlation_kernel_for_the_measure_q_vol_ (end)

\section{Correlation kernel for uniformly random\\ Gelfand-Tsetlin schemes} % (fold)
\label{sec:correlation_kernel_for_uniform_gelfand_tsetlin_schemes}

\subsection{The kernel} % (fold)
\label{sub:the_kernel}

In this section we compute the correlation kernel for uniformly random Gelfand-Tsetlin schemes with arbitrary fixed top row $\nu\in\GT_N$:
\begin{theorem}\label{thm:kernel}
  The correlation kernel $K$ of the uniform measure $\Pp_{N,\nu}$ on interlacing particle arrays $\{\x_j^m\colon m=1,\ldots,N,\; j=1,\ldots,m\}$ with fixed top row $\x_j^N=\nu_j-j$ ($j=1,\ldots,N$) is given for $1\le n_1\le N$, $1\le n_2\le N-1$, and $x_1,x_2\in\Z$ by
  \begin{align}&
    K(x_1,n_1;x_2,n_2)=
    -1_{n_2<n_1}1_{x_2\le x_1}\frac{(x_1-x_2+1)_{n_1-n_2-1}}{(n_1-n_2-1)!}
    +\frac{(N-n_1)!}{(N-n_2-1)!}
    \times
    \nonumber\\&
    \qquad\qquad
    \qquad
    \times
    \frac1{(2\pi\i)^{2}}
    \oint\limits_{\Ga(x_2)}dz\oint\limits_{\ga(\infty)}dw
    \frac{(z-x_2+1)_{N-n_2-1}}{(w-x_1)_{N-n_1+1}}
    \frac{1}{w-z}
    \prod_{r=1}^{N}\frac{w+r-\nu_r}{z+r-\nu_r}.
    \label{kernel_formula}
  \end{align}
  The counter-clockwise contour $\Ga(x_2)$ in $z$ contains the points $x_2,x_2+1,\ldots,\nu_1-1$ and not $x_2-1,x_2-2,\ldots,\nu_N-N$. The counter-clockwise contour $\ga(\infty)$ contains $\Ga(x_2)$ (without intersecting it) and is sufficiently large to include all the points $x_1,x_1-1,\ldots,x_1-(N-n_1)$.
\end{theorem}

By taking appropriate top row particles as in (\ref{fixed_top_row}), we see that Theorem \ref{thm:kernel} readily implies Theorem \ref{thm:K_intro}.

% subsection the_kernel (end)

\subsection{Preliminaries} % (fold)
\label{sub:preliminaries}

In the rest of this section we prove Theorem \ref{thm:kernel} by taking the $q\ua 1$ limit in Theorem \ref{thm:q_kernel}. That is, we obtain $K$ (\ref{kernel_formula}) as
\begin{align}\label{K=lim_of_qK}
  K(x_1,n_1;x_2,n_2)=\lim_{q\ua1}\q K(x_1,n_1;x_2,n_2).
\end{align}
Note that since the measures $\q\Pp_{N,\nu}$ tend to $\Pp_{N,\nu}$ as $q\ua1$, the correlation functions of $\Pp_{N,\nu}$ must be limits of those of $\q\Pp_{N,\nu}$. But because the correlation kernel of a point process is not defined uniquely,\footnote{For example, one can conjugate the kernel as $K(x_1,n_1;x_2,n_2)\mapsto\frac{f(x_1,n_1)}{f(x_2,n_2)}K(x_1,n_1;x_2,n_2)$ with a nonvanishing function $f(x,n)$, which does not affect the correlation functions (\ref{correlation_kernel_intro}).} relation (\ref{K=lim_of_qK}) for any two correlation kernels of $\Pp_{N,\nu}$ and $\q\Pp_{N,\nu}$ does not a priori have to hold.

\begin{remark}
  We do not know if it is possible to establish (\ref{K=lim_of_qK}) directly by looking at the asymptotics of the integrand in (\ref{q_kernel_formula}) and transforming the contours in some way. Our proof involves breaking the integral for $\q K$ in (\ref{q_kernel_formula}) into much smaller pieces and looking at their Taylor expansions at $q=1$. In such expansions almost all terms disappear in the $q\ua 1$ limit. This is the reason why a rather complicated formula for the $q$-deformed kernel $\q K$ of Theorem \ref{thm:q_kernel} (with a $q$-hypergeometric function inside) turns into a simpler expression for $K$ (\ref{kernel_formula}). 
\end{remark}

Observe that the additional summand in $\q K$ simply tends to the corresponding summand in $K$:
\begin{equation*}
  \lim_{q\ua1}
  q^{n_2(x_1-x_2)}
  \frac{(q^{x_1-x_2+1};q)_{n_1-n_2-1}}{(q;q)_{n_1-n_2-1}}
  =\frac{(x_1-x_2+1)_{n_1-n_2-1}}{(n_1-n_2+1)!}
\end{equation*}
(this a well-known property of the $q$-Pochhammer symbols). The rest of this section is devoted to establishing the convergence of the remaining double contour integrals.

% subsection preliminaries (end)

\subsection{Residues} % (fold)
\label{sub:residues}

The integrand in (\ref{q_kernel_formula}) has (possible) poles in the variable $z$ at the points $q^{\nu_1-1-x_1},\ldots,q^{\nu_N-N-x_1}$. For (\ref{K=lim_of_qK}) to hold, the residue at every point $z=q^{\nu_j-j-x_1}$ ($j=1,\ldots,N$) must have a limit as $q\ua1$. Here we are allowed to consider
individual residues instead of their combination $\q K(x_1,n_1;x_2,n_2)$ because one 
can express each such residue as a linear combination of the $\q K(x_1,n_1;x_2,n_2)$'s with various values of $x_2$. 

Fix $j=1,\ldots,N$. The residue at $z=q^{\nu_j-j-x_1}$ looks as
\begin{align}\label{Residue_j}
  &
  \q Res_j:=\frac{(q^{\nu_j-j-x_2+1};q)_{N-n_2-1}}{(q;q)_{N-n_2-1}}
  \times\\&\qquad\qquad\qquad\times
  q^{n_2(\nu_j-j-x_2)}(q^{N-1};q^{-1})_{N-n_1}
  \nonumber
  \prod_{r\ne j}\frac1{q^{\nu_j-j-x_1}-q^{\nu_r-r-x_1}}
  \times\\&\qquad \qquad\qquad\times
  \nonumber
  \frac{1}{2\pi\i}
  \oint\limits_{\ga(\infty)} \frac{dw}{w}
  {}_2\phi_1(q^{-1},q^{n_1-1};q^{N-1}\mid q^{-1};w^{-1})
  \prod_{r\ne j}({w-q^{\nu_r-r-x_1}}).
\end{align}
For shorter notation, set
\begin{equation}
  \label{alphas}
  (\al_1,\ldots,\al_{N-1}):=
  (\nu_1-1-x_1,\ldots,\widehat{\nu_j-j-x_1},\ldots,\nu_N-N-x_1)\in\Z^{N-1}.
\end{equation}
Now let us compute the integral in $w$ in (\ref{Residue_j}). This amounts to simply taking the free term in the product of the generating series ${}_2\phi_1$ and $\prod_{r=1}^{N-1}(w-q^{\al_r})$. Thus, our next goal is to understand the $q\ua1$ asymptotics of the following expression:
\begin{align}\nonumber
  &
  \frac{(q^{N-1};q^{-1})_{N-n_1}}{2\pi\i}
  \oint\limits_{\ga(\infty)} \frac{dw}{w}
  {}_2\phi_1(q^{-1},q^{n_1-1};q^{N-1}\mid q^{-1};w^{-1})
  \prod_{r=1}^{N-1}(w-q^{\al_r})
  \\&\qquad \qquad=\sum_{i=0}^{n_1-1}
  \Big(\prod_{r=n_1+1}^{N}(1-q^{r-i-1})\Big)
  (-1)^{N-i-1}e_{N-i-1}(q^{\al_1},\ldots,q^{\al_{N-1}}).
  \label{things_we_need_to_understand_q->1}
\end{align}
Denote the right-hand side of this expression by $\q E_{n_1}(\al_1,\ldots,\al_{N-1})$.

% subsection residues (end)

\subsection{Taylor expansion of $q$-specialized elementary symmetric polynomials} % (fold)
\label{sub:taylor_expansion_of_elementary_symmetric_polynomials}

One cannot simply plug $q=1$ in $\q E_{n_1}(\al_1,\ldots,\al_{N-1})$ (\ref{things_we_need_to_understand_q->1}) to get its $q\ua1$ limit; in particular, that would destroy the dependence on $\al_1,\ldots,\al_{N-1}$. Thus, we need to consider the whole Taylor expansion of $e_{N-i-1}(q^{\al_1},\ldots,q^{\al_{N-1}})$ at $q=1$ to see which terms would matter in the $q\ua1$ asymptotics in (\ref{things_we_need_to_understand_q->1}) and therefore in the residue $\q Res_j$ (\ref{Residue_j}).

We will need to use the following symmetric polynomials:
\begin{equation*}
  \mathbf{m}_{\la}^*(y_1,\ldots,y_{N-1}):=
  \sum_{i_1,\ldots,i_{\ell(\la)}}y_{i_1}^{\da\la_1}\ldots y_{i_{\ell(\la)}}^{\da\la_{\ell(\la)}},
\end{equation*}
where the sum is taken over all pairwise distinct indices $i_1,\ldots,i_{\ell(\la)}$ from 1 to $N-1$, and $\la\in\Yb$ is a Young diagram \cite[Ch. I.I]{Macdonald1995} with number of parts $\ell(\la)$. Here $y^{\da m}:=y(y-1)\ldots(y-m+1)$ is the falling factorial power. These $\mathbf m_{\la}^{*}$'s are the so-called \emph{augmented factorial monomial symmetric polynomials}.

\begin{proposition}
  For all $m=0,\ldots,N-1$, we have the following two expansions:
  \begin{align}\label{expansion_e_m_1}
    e_m(q^{\al_1},\ldots,q^{\al_{N-1}})=
    \sum_{\la\in\Yb}&
    \frac{(q-1)^{|\la|}}{|\la|!}
    \binom{N-1-\ell(\la)}{N-1-m}
    \mathbf{m}_{\la}^*(\al_1,\ldots,\al_{N-1})
    \\=
    q^{\al_1+\ldots+\al_{N-1}}
    \sum_{\la\in\Yb}&
    \frac{(q-1)^{|\la|}}{|\la|!}
    \binom{N-1-\ell(\la)}{m}
    \mathbf{m}_{\la}^*(-\al_1,\ldots,-\al_{N-1}),
    \label{expansion_e_m_2}
  \end{align}
  where the sums are taken over all Young diagrams $\la\in\Yb$.
\end{proposition}
To see that both series in (\ref{expansion_e_m_1}) and (\ref{expansion_e_m_2}) converge, one could argue as follows. Using $e_{m}(q^{\al_1},\ldots,q^{\al_{N-1}})=q^{-Lm}e_{m}(q^{\al_1+L},\ldots,q^{\al_{N-1}+L})$ ($e_m$'s are homogeneous), one could make all $\al_i$'s positive. Since $\mathbf{m}_{\la}^{*}(\be_1,\ldots,\be_{N-1})$ vanishes for nonnegative $\be_i$'s if all $\la_j$'s are large enough, the sum in (\ref{expansion_e_m_1}) for positive $\al_i$'s will be finite. Thus, one can turn (\ref{expansion_e_m_1}) into a power of $q$ (which expands into a convergent series at $q=1$) times a polynomial in $q-1$. The same trick can be performed for (\ref{expansion_e_m_2}).
\begin{proof}
  The two expansions (\ref{expansion_e_m_1}) and (\ref{expansion_e_m_2}) turn one into another because
  \begin{align*}
    e_m(y_1,\ldots,y_{N-1})=y_1 \ldots y_{N-1}\cdot
    e_{N-1-m}(y_1^{-1},\ldots,y_{N-1}^{-1}).
  \end{align*}
  Thus, we only need to prove, say, (\ref{expansion_e_m_1}). 

  Consider the generating function
  \begin{align*}
    F(t;q):=\sum_{m=0}^{N-1}t^{m}e_{m}(q^{\al_1},\ldots,q^{\al_{N-1}})
    =\prod_{r=1}^{N-1}(1+tq^{\al_r}).
  \end{align*}
  Since 
  \begin{align*}
    \frac{\partial^{s}}{\partial q^{s}}
    (1+tq^{\al_r})|_{q=1}=t\al_r^{\da s},\qquad s\ge1,
  \end{align*}
  we have for any $s=0,1,\ldots$ the following expression for partial derivatives of $F(t;q)$:
  \begin{align}\label{expansion_e_m_proof}
    \frac{\partial^{s}}{\partial q^{s}}
    F(t;q)|_{q=1}=
    \sum\nolimits_{\la\in\Yb\colon|\la|=s}
    t^{\ell(\la)}(1+t)^{N-1-\ell(\la)}\mathbf{m}_{\la}^*(\al_1,\ldots,\al_{N-1}).
  \end{align}
  Indeed, every summand corresponding to some partition $\la$ indicates that we apply the derivative $\frac{\partial^{\la_i}}{\partial q^{\la_i}}$ to one of the factors $(1+tq^{\al_{r_{i}}})$ for every $i=1,\ldots,\ell(\la)$. Summing over all possibilities of doing that, we get the polynomial $\mathbf{m}_\la^*$. After this, $N-1-\ell(\la)$ factors are not affected by differentiation, and this gives us the multiplication by $(1+t)^{N-1-\ell(\la)}$. The special case $s=0$ in (\ref{expansion_e_m_proof}) is checked directly.
  
  Since
  \begin{align*}
    [t^{m}]\Big(t^{\ell(\la)}(1+t)^{N-1-\ell(\la)}\Big)=\binom{N-1-\ell(\la)}{N-1-m}
  \end{align*}
  (coefficient by $t^m$), we obtain (\ref{expansion_e_m_1}) by writing the standard Taylor expansion. This concludes the proof.
\end{proof}

% subsection taylor_expansion_of_elementary_symmetric_polynomials (end)

\subsection{$q$-Stirling and classical Stirling numbers} % (fold)
\label{sub:_q_stirling_and_classical_stirling_numbers}

Now, plugging the expansion of $e_{N-i-1}(q^{\al_1},\ldots,q^{\al_{N-1}})$ given by (\ref{expansion_e_m_2}) into (\ref{things_we_need_to_understand_q->1}), we can rewrite
\begin{align}&
  \q E_{n_1}(\al_1,\ldots,\al_{N-1})=
  q^{\al_1+\ldots+\al_{N-1}}
  \sum_{\la\in\Yb}
  \frac{(q-1)^{|\la|}}{|\la|!}
  \mathbf{m}_{\la}^*(-\al_1,\ldots,-\al_{N-1})
  \nonumber
  \times\\&\qquad \qquad\times
  \label{q->1_sum_over_i}
  \sum_{i=0}^{n_1-1}
  \Big(\prod_{r=n_1+1}^{N}(1-q^{r-i-1})\Big)
  (-1)^{N-i-1}
  \binom{N-1-\ell(\la)}{N-1-i}.
\end{align}
In this subsection we will relate the sum over $i$ above to the known $q$-Stirling numbers (of the second kind), and this will help us to write the leading term of the $q\ua 1$ asymptotics of $\q E_{n_1}(\al_1,\ldots,\al_{N-1})$ (Proposition \ref{prop:principal_term_qEn} below).

We will use the standard $q$-notation:
\begin{align*}
  [n]_{q}:=\frac{1-q^{n}}{1-q},\qquad
  [n]_{q}!:=[n]_{q}[n-1]_{q}\ldots[1]_{q}.
\end{align*}
\begin{definition}[\cite{Gould1960}]
  The $q$-Stirling numbers of the second kind $S(n,m;q)$ are defined as the following expansion coefficients (here $|z|$ is sufficiently small):
  \begin{align}\label{q_stirling_def}
    \prod_{r=1}^{n}\frac1{1-[r]_{q}\cdot z}=
    \sum_{m=0}^{\infty}S(n,m;q)z^{m}.
  \end{align}
\end{definition}

\begin{lemma}\label{lemma:q_stirling}
  For every $l=0,\ldots,N-1$, we have
  \begin{align}
    \label{q_stirling}
    \sum_{i=0}^{n_1-1}&
    \Big(\prod_{r=n_1+1}^{N}(1-q^{r-i-1})\Big)
    (-1)^{N-i-1}
    \binom{N-1-l}{N-1-i}\\&=
    (-1)^{N+n_1}(1-q)^{N-l-1}[N-n_1]_{q}!\cdot S(N-n_1,n_1-l-1;q).
    \nonumber
  \end{align}
\end{lemma}
\begin{proof}
  We use the following explicit formula for the $q$-Stirling numbers from \cite{Gould1960}:
  \begin{align*}
    S(n,m;q)=
    (q-1)^{-m}\sum_{p=0}^{m}
    (-1)^{p}\binom{m+n}p
    \frac{[m+n-p]_{q}!}{[n]_{q}![m-p]_{q}!}.
  \end{align*}
  Thus, 
  \begin{align*}&
    S(N-n_1,n_1-l-1;q)\\&=
    (q-1)^{l+1-n_1}
    \sum\nolimits_{p=0}^{n_1-l-1}
    (-1)^{p}\binom{N-l-1}p
    \frac{[N-l-1-p]_{q}!}{[N-n_1]_{q}![n_1-l-1-p]_{q}!}
    \\&=
    (q-1)^{l+1-n_1}
    \sum\nolimits_{p=l}^{n_1-1}
    (-1)^{p-l}\binom{N-l-1}{N-p-1}
    \frac{[N-1-p]_{q}!}{[N-n_1]_{q}![n_1-1-p]_{q}!}.
  \end{align*}
  Note that the summation over $i$ in (\ref{q_stirling}) is also done from $l$ to $n_1-1$ due to the presence of the binomial coefficient $\binom{N-1-l}{N-1-i}$. It remains to match terms in the above formula with those in (\ref{q_stirling}) to see that the claim holds.
\end{proof}

The $q$-Stirling numbers $S(n,m;q)$ converge to the classical Stirling numbers $S(n,m)$ which are defined as the following expansion coefficients (for sufficiently small $|z|$):
\begin{align}\label{stirling_def}
  \prod_{r=1}^{n}\frac1{1-rz}=
  \sum_{m=0}^{\infty}S(n,m)z^{m}.
\end{align}
\begin{remark}
  It can be readily checked that the numbers $S(n,m)$ are related to the more common parametrization of the Stirling numbers of the second kind
  \begin{equation*}
    \begin{Bmatrix}
      n\\m
    \end{Bmatrix}:=
    \frac1{m!}\sum_{p=0}^{m}(-1)^{p}\binom{m}{p}(m-p)^{n}
  \end{equation*}
  via $S(n,m)=\left\{\genfrac{}{}{0pt}{}{n+m}{m}\right\}.$
\end{remark}

\begin{proposition}\label{prop:principal_term_qEn}
  The leading term of the asymptotics of  $\q E_{n_1}(\al_1,\ldots,\al_{N-1})$ (\ref{q->1_sum_over_i}) as $q\ua1$ looks as follows:
  \begin{align}
    \q E_{n_1}(\al_1,\ldots,\al_{N-1})
    \sim{}&
    (1-q)^{N-1}
    (-1)^{N+n_1}
    (N-n_1)!
    \label{principal_term_qEn}
    \times\\&\qquad\times
    \nonumber
    \sum_{l=0}^{n_1-1}
    e_{l}(\al_1,\ldots,\al_{N-1})
    S(N-n_1,n_1-l-1).
  \end{align}
\end{proposition}
\begin{proof}
  Looking at formula (\ref{q->1_sum_over_i}) and Lemma \ref{lemma:q_stirling}, we see that each sum over $i$ in (\ref{q->1_sum_over_i}) behaves as $\sim\mathrm{const}\cdot(1-q)^{N-\ell(\la)-1}$. Therefore, the whole expression $\q E_{n_1}(\al_1,\ldots,\al_{N-1})$ (\ref{q->1_sum_over_i}) behaves as
  \begin{align*}
    \sim\mathrm{const}\cdot (1-q)^{N-1}\sum_{\la\in\Yb}
    \mathrm{const}_{\la}\cdot (q-1)^{|\la|-\ell(\la)}.
  \end{align*}
  It follows that Young diagrams $\la$ with $|\la|>\ell(\la)$ provide a negligible contribution to $\q E_{n_1}(\al_1,\ldots,\al_{N-1})$. Thus, we are left only with one-column Young diagrams, and for each of them the factorial monomial symmetric polynomial $\mathbf{m}_\la^*$ reduces to a multiple of $e_{\ell(\la)}$. Lemma \ref{lemma:q_stirling} then provides the necessary asymptotics which directly leads to (\ref{principal_term_qEn}). Observe that we only need to sum over $0\le l\le n_1-1$ because for bigger $l$, the Stirling numbers $S(N-n_1,n_1-l-1)$ vanish. This concludes the proof.
\end{proof}

\begin{remark}\label{rmk:q_disappearance}
  We see from the above proposition that taking $q\ua1$ limit allows us to drop almost all summands in the sum over $\la\in\Yb$ in (\ref{q->1_sum_over_i}). This is the reason why the $q$-hypergeometric function in the formula for the kernel $\q K$ (\ref{q_kernel_formula}) disappears for $q=1$ (\ref{kernel_formula}). Proposition \ref{prop:oint_principal_term} below also displays this effect.
\end{remark}

% subsection _q_stirling_and_classical_stirling_numbers (end)

\subsection{Completing the proof} % (fold)
\label{sub:completing_the_proof}

Summarizing the development of \S\S \ref{sub:taylor_expansion_of_elementary_symmetric_polynomials}--\ref{sub:_q_stirling_and_classical_stirling_numbers} and translating Proposition \ref{prop:principal_term_qEn} into the language of contour integrals, we have:
\begin{proposition}\label{prop:oint_principal_term}
  The leading term as $q\ua1$ of the contour integral in $w$ in $\q Res_j$ (\ref{Residue_j}) looks as follows:
  \begin{align}&
    {(q^{N-1};q^{-1})_{N-n_1}}
    \oint_{\ga(\infty)} \frac{dw}{w}
    {}_2\phi_1(q^{-1},q^{n_1-1};q^{N-1}\mid q^{-1};w^{-1})
    \prod_{r\ne j}(w-q^{\nu_r-r-x_1})
    \nonumber
    \\&\qquad\sim
    (q-1)^{N-1}
    {(N-n_1)!}
    \oint_{\ga(\infty)}\frac{dw}{(w-x_1)_{N-n_1+1}}\prod_{r\ne j}(w+r-\nu_r).
    \label{oint_principal_term}
  \end{align}
  Here in both integrals the contours in $w$ are counter-clockwise and have sufficiently large radii.
\end{proposition}
\begin{proof}
  A simple change of variables in the integral in the right-hand side of (\ref{oint_principal_term}) and the expansion of the integrand into powers of $w$ lead to appearance of the classical Stirling numbers via (\ref{stirling_def}), and also of the elementary symmetric polynomials $e_{l}(\al_1,\ldots,\al_{N-1})$ (recall (\ref{alphas})). Then it can be readily checked that the claim directly follows from (\ref{things_we_need_to_understand_q->1}) and Proposition \ref{prop:principal_term_qEn}.
\end{proof}

Now we are in a position to compute the limits of the residues $\q Res_j$ (\ref{Residue_j}):
\begin{align}
  \lim_{q\ua1}\q Res_j=:Res_j&=
  \frac{(N-n_1)!}{(N-n_2-1)!}
  \frac{(\nu_j-j-x_2+1)_{N-n_2-1}}
  {\prod_{r\ne j}(\nu_j-j-\nu_r+r)}
  \label{Res_j_uniform}
  \times\\&\qquad \qquad \qquad \times
  \frac{1}{2\pi \i}
  \oint\limits_{\ga(\infty)}\frac{dw}{(w-x_1)_{N-n_1+1}}\prod_{r\ne j}(w+r-\nu_r).
  \nonumber
\end{align}
Indeed, the factor $(q-1)^{N-1}$ in Proposition \ref{prop:oint_principal_term} is exactly what is needed to turn the product $\prod_{r\ne j}({q^{\nu_j-j-x_1}-q^{\nu_r-r-x_1}})^{-1}$ in (\ref{Residue_j}) into $\prod_{r\ne j}(\nu_j-j-\nu_r+r)^{-1}$. Everything else in the above formula also follows from (\ref{Residue_j}) and Proposition \ref{prop:oint_principal_term}.

To complete the proof of Theorem \ref{thm:kernel}, it remains to note that the double contour integral in the formula for the $q$-deformed kernel $\q K(x_1,n_1;x_2,n_2)$ (\ref{q_kernel_formula}) is equal to the sum of $\q Res_j$ (\ref{Residue_j}) over all $j=1,\ldots,N$ such that $\nu_j-j\ge x_2$; and the same is true for the double contour integral for $K(x_1,n_1;x_2,n_2)$ in (\ref{kernel_formula}) and the limiting residues $Res_j$ defined above.

Thus, we have completed the proof of Theorem \ref{thm:kernel}, and also of Theorem \ref{thm:K_intro}.

% subsection completing_the_proof (end)

% section correlation_kernel_for_uniform_gelfand_tsetlin_schemes (end)

\section{Inverse Kasteleyn matrix} % (fold)
\label{sec:inverse_kasteleyn_matrix}
\label{sub:inverse_kasteleyn_matrix}

\subsection{Inverse Kasteleyn matrix and the correlation kernel} % (fold)
\label{sub:inverse_kasteleyn_matrix_and_the_correlation_kernel}

Let us show how the kernel $K$ is related to the inverse of the Kasteleyn matrix for the honeycomb graph $G_{\Pc}$ inside our polygon $\Pc$ (Fig.~\ref{fig:tiling_dimers}, right). We would like to use the affine transform (\ref{lozenges}), so that $\Pc$ will be parametrized by $A_i,B_i$ as in \S \ref{sub:lozenge_tilings_of_polygons} (see Fig.~\ref{fig:polygonal_region_tiling}), or, equivalently (see (\ref{fixed_top_row}) and \S\ref{sub:connection_to_measures_on_tilings}), by a fixed signature $\nu\in\GT_N$ as in Theorem~\ref{thm:kernel}. 
 
The graph $G_\Pc$ is bipartite; its vertices correspond to two types of (triangle) faces in the dual triangular lattice:
\begin{center}
  \includegraphics[width=55pt]{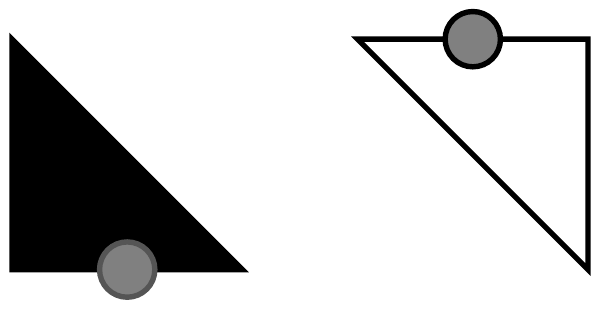}
\end{center}
We will encode each such triangle by the position $(x,n)$ of the mid-point of its horizontal side. The Kasteleyn matrix of the graph $G_{\Pc}$ is its adjacency matrix with rows and columns parametrized by white and black triangles, respectively (e.g., see \cite{Kenyon2007Lecture}). Inside the polygon, this matrix looks as
\begin{align}
  \label{Kasteleyn_matrix}
  \Kast(\wt(x,n);\bt(y,m))=\begin{cases}
    1,&\mbox{if $(y,m)=(x,n)$};\\
    1,&\mbox{if $(y,m)=(x,n-1)$};\\
    1,&\mbox{if $(y,m)=(x+1,n-1)$};\\
    0,&\mbox{otherwise}
  \end{cases}
\end{align}
(see Fig.~\ref{fig:lozenges_triangles}).\begin{figure}[htbp]
  \begin{tabular}{c}
    \includegraphics[width=150pt]{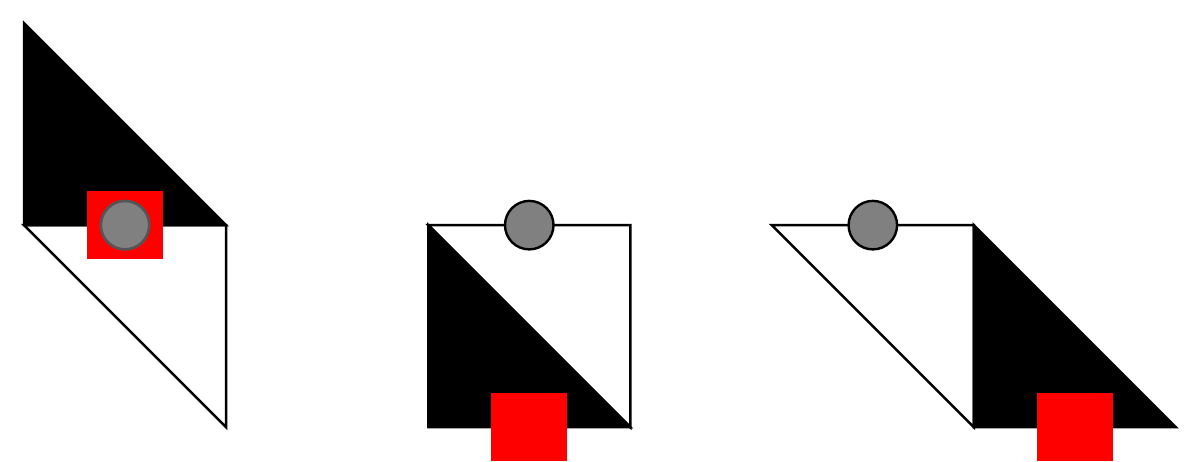}
  \end{tabular}
  \caption{Edges of three directions in the graph $G_\Pc$ encoded by pairs of triangles.}
  \label{fig:lozenges_triangles}
\end{figure} For $\wt(x,n)$ on the boundary of the graph $G_\Pc$, the $\wt(x,n)$--th row of $\Kast$ will contain less than three ones, and the same for the $\bt(y,m)$--th column. 

The inverse matrix $\Kast^{-1}$ has rows and columns indexed by black and white triangles, respectively. As \cite[Cor. 3]{Kenyon2007Lecture} suggests, $\Kast^{-1}$ can also serve as a correlation kernel for the uniform measure on tilings of $\Pc$. Based on the explicit formula of Theorem \ref{thm:kernel} (or Theorem \ref{thm:K_intro}), we establish the following connection between $\Kast^{-1}$ and our correlation kernel $K$:

\begin{theorem}\label{thm:Kasteleyn}
  The inverse Kasteleyn matrix and the correlation kernel $K$ of Theorem \ref{thm:kernel} are related as follows:
  \begin{align*}
    \Kast^{-1}(\bt(y,m);\wt(x,n))=(-1)^{y-x+m-n}K(x,n;y,m).
  \end{align*}
\end{theorem}
\begin{proof}
  We need to check that 
  \begin{align}\label{Kasteleyn_desired}
    \sum\nolimits_{(x,n)}
    (-1)^{y'-x+m'-n}K(x,n;y',m')
    \Kast(\wt(x,n);\bt(y,m))
    =1_{(y,m)=(y',m')}.
  \end{align}
  The black triangles $\bt(y,m)$ in the graph $G_\Pc$ have $0\le m\le N-1$ (e.g., see Fig.~\ref{fig:polygonal_region_tiling}). Depending on the position of a black triangle, (\ref{Kasteleyn_desired}) turns into a three-term or a two-term relation for the correlation kernel $K$. We aim to verify all these relations using the explicit formula (\ref{kernel_formula}) which expresses the correlation kernel $K(x_1,n_1;x_2,n_2)$ as a double contour integral (denote it by $I(x_1,n_1;x_2,n_2)$) plus an additional summand.

  {\bf1.} (inside the polygon) Let $1\le m\le N-1$, and $y$ be such that the $\bt(y,m)$-th column of $\Kast$ has three ones. Then (\ref{Kasteleyn_desired}) becomes the three-term relation:
  \begin{align}\label{kernel_Kast_3term}
    K(y,m;y',m')-K(y,m+1;y',m')+K(y-1&,m+1;y',m')
    =1_{(y,m)=(y',m')}.
  \end{align}
  Using an obvious identity
  \begin{align*}
    \frac{(N-m)!}{(w-y)_{N-m+1}}-\frac{(N-m-1)!}{(w-y)_{N-m}}+
    \frac{(N-m-1)!}{(w-y+1)_{N-m}}=0,
  \end{align*}
  we see that the terms corresponding to the double contour integral part of $K$ sum to zero:
  \begin{align}\label{integrals_sum_to_zero}
    I(y,m;y',m')-I(y,m+1;y',m')+I(y-1,m+1;y',m')=0,
  \end{align}
  because the contours in all these integrals are the same. Note that (\ref{integrals_sum_to_zero}) in fact holds for any $0\le m,m'\le N-1$ and any $y,y'\in\Z$.

  The fact that the terms corresponding to the additional summand in $K$ give the desired result $1_{(y,m)=(y',m')}$ in (\ref{kernel_Kast_3term}) can be checked directly. Thus, (\ref{Kasteleyn_desired}) is established inside the polygon.

  {\bf2.} (boundary) Now we consider boundary black triangles $\bt(y,m)$. Note that the ones which are close to the top boundary of $\Pc$ are already included in the general case {\bf1}. 

  {\bf2a.} (bottom boundary) Let $m=0$, so the triangle $\bt(y,m)$ lies at the bottom of~$\Pc$ (Fig.~\ref{fig:polygonal_region_tiling}). For $m=0$, (\ref{Kasteleyn_desired}) turns into 
  \begin{align}\label{Kast_proof_noname1}
    -K(y,1;y',m')+K(y-1,1;y',m')
    =1_{(y,0)=(y',m')}.
  \end{align}
  One can show that $K(y,0;y',m')$ vanishes: (1) there is no additional summand in (\ref{kernel_formula}) because $m'\ge0$; (2) looking at the integral in $w\in\ga(\infty)$ in (\ref{kernel_formula}), we see that the $w$-integrand decays as $w^{-2}$ at $\infty$ and thus has no residue there, so the integral over $w$ is also zero. Thus, we may include $K(y,0;y',m')$ into (\ref{Kast_proof_noname1}), and use the fact that the resulting three-term relation (\ref{kernel_Kast_3term}) is already established.

  {\bf2b.} (vertical boundary) Let $1\le m\le N-1$, and let $\bt(y,m)$ be close to the vertical boundary of the polygon $\Pc$ (except for the rightmost vertical boundary which falls into the general case {\bf1}). This means that the point $(y-1,m+1)$ is outside $\Pc$, and we must show that 
  \begin{align}\label{Kast_proof_noname2}
    K(y,m;y',m')-K(y,m+1;y',m')
    =1_{(y,m)=(y',m')}.
  \end{align}
  Using (\ref{integrals_sum_to_zero}), we can replace the double contour integral parts $I(y,m;y',m')-I(y,m+1;y',m')$ above by $-I(y-1,m+1;y',m')$. 

  We can compute the resulting double contour integral $I(y-1,m+1;y',m')$. Observe that for the point $(y-1,m+1)$ to be outside $\Pc$, the top row particles must occupy positions $y-N+m,\ldots,y-1$ (see Fig.~\ref{fig:polygonal_region_tiling}). This means that the integrand in $w$ in $I(y-1,m+1;y',m')$ does not have poles inside the contour except for $w=z$. Taking the residue at $w=z$ and using the Lemma \ref{lemma:Res_w=z} below, we see that
  \begin{align*}
    I(y-1,m+1;y',m')=1_{y'<y}\frac{(m-m'+1)_{y-y'-1}}{(y-y'-1)!}.
  \end{align*}
  Taking into account the additional summands that come from the two kernels in (\ref{Kast_proof_noname2}), one can directly check that the desired identity (\ref{Kast_proof_noname2}) holds for every $(y',m')$ inside the polygon $\Pc$. In fact, (\ref{Kast_proof_noname2}) may fail if $(y',m')$ is outside the polygon.

  {\bf2c.} (boundary parallel to the vector $(-1,1)$) Finally, let $1\le m\le N-1$, and $\bt(y,m)$ be close to the boundary of $\Pc$ of direction $(-1,1)$ (except for the leftmost such boundary which falls into the general case {\bf1}). We need to show that 
  \begin{align*}
    K(y,m;y',m')+K(y-1&,m+1;y',m')
    =1_{(y,m)=(y',m')}.
  \end{align*}
  Since $\bt(y,m)$ is at the boundary, and thus the point $(y,m+1)$ is outside $\Pc$, we note that the top row particles must occupy positions $y-N+m-1,\ldots,y$ (see Fig.~\ref{fig:polygonal_region_tiling}). Then one can argue in the same way as in {\bf2b}.

  This concludes the proof of the theorem modulo the below Lemma \ref{lemma:Res_w=z}.
\end{proof}

% subsection inverse_kasteleyn_matrix_and_the_correlation_kernel (end)

\subsection{Computing some contour integrals} % (fold)
\label{sub:computing_some_contour_integrals}

\begin{lemma}\label{lemma:Res_w=z}
  For all $1\le n_1\le N$, $1\le n_2\le N-1$, and $x_1,x_2\in\Z$, we have
  \begin{align}
    \label{Res_w=z}
    \frac{(N-n_1)!}{(N-n_2-1)!}
    \times\frac{1}{2\pi\i}
    \oint\limits_{\Ga(x_2)}
    \frac{(z-x_2+1)_{N-n_2-1}}{(z-x_1)_{N-n_1+1}}dz=
    1_{x_2\le x_1}
    \frac{(n_1-n_2)_{x_1-x_2}}{(x_1-x_2)!},
  \end{align}
  where the contour $\Ga(x_2)$ in $z$ is the same as in Theorem \ref{thm:kernel}: a counter-clockwise contour which encircles points $x_2,x_2+1,\ldots,$ and not points $x_2-1,x_2-2,\ldots$.
\end{lemma}
As was mentioned in the proof of the above Theorem \ref{thm:Kasteleyn}, the integral in (\ref{Res_w=z}) arises if one takes the $w=z$ residue in the $w$ integral in the formula for our correlation kernel $K(x_1,n_1;x_2,n_2)$ (\ref{kernel_formula}).
\begin{proof}
  Denote the integral in (\ref{Res_w=z}) by $J$. Let us denote $\d x=x_1-x_2$ and $\d n=n_1-n_2$. Shifting the variable $z$ by $x_1$, we have
  \begin{align*}
    J=
    \frac{(N-n_1)!}{(N-n_2-1)!}
    \times\frac{1}{2\pi\i}
    \oint_{\Ga(-\d x)}
    \frac{(z+\d x+1)_{N-n_2-1}}{(z)_{N-n_1+1}}dz.
  \end{align*}
  The integrand here has poles $0,-1,\ldots,-(N-n_1)$. We see that if $\d x<0$, the contour of integration $\Ga(-\d x)$ has no poles inside, and thus $J=0$. 

  In the rest of the proof we assume that $\d x\ge0$. The integral $J$ then becomes a sum over the poles $0,-1,\ldots,-\d x$ of the corresponding residues which we write as follows:
  \begin{align*}
    J=
    \frac{\Gamma(N-n_2+\d x)}{\Gamma(\d x+1)\Gamma(N-n_2)}
    \sum_{j=0}^{\d x}
    \Bigg[{}&{}
    \frac{\Gamma(\d x+1)}{\Gamma(z+\d x+1)}
    \frac{\Gamma(N-n_1+1)}{\Gamma(z+N-n_1+1)}
    \times\\
    &\qquad\times
    \frac{\Gamma(z+\d x+N-n_2)}{\Gamma(\d x+N-n_2)}
    (z+j)\Gamma(z)\Bigg]_{z=-j}.
  \end{align*}
  Using a simple observation that 
  \begin{align}\label{Pochhammer_gamma_trick}
    \frac{\Gamma(-A+1)}{\Gamma(-A+1-j)}=
    \frac{\Gamma(-A+1+j-j)}{\Gamma(-A+1-j)}
    =
    (-A+1-j)_{j}=(-1)^{j}(A)_{j},
  \end{align}
  and the residue of the Gamma function $(z+j)\Gamma(z)|_{z=-j}=(-1)^{j}/j!$, we can rewrite as follows:
  \begin{align*}
    J&=
    \frac{\Gamma(N-n_2+\d x)}{\Gamma(\d x+1)\Gamma(N-n_2)}
    \sum_{j=0}^{\d x}
    \frac{(-\d x)_{j}
    (n_1-N)_{j}}
    {(-\d x-N+n_2+1)_{j}}\frac{1}{j!}
    \\&=
    \frac{\Gamma(N-n_2+\d x)}{\Gamma(\d x+1)\Gamma(N-n_2)}
    {}_2F_1
    \left(
    \left.\begin{array}{c}
    -\d x,n_1-N\\  
    -\d x-N+n_2+1
    \end{array}\right|
    1
    \right).
  \end{align*}
  Here ${}_2F_1$ is the Gauss hypergeometric function. We may use the Gauss summation formula \cite[2.8.(46)]{Erdelyi1953} for it if we assume for a while that $N$ is a nonreal complex number:
  \begin{align*}
    {}_2F_1
    \left(
    \left.\begin{array}{c}
    -\d x,n_1-N\\  
    -\d x-N+n_2+1
    \end{array}\right|
    1
    \right)&=
    \frac{\Gamma(-\d x-N+n_2+1)}
    {\Gamma(-N+n_2+1)}
    \frac{\Gamma(-\d n+1)}
    {\Gamma(-\d x-\d n+1)}
    \\&=
    \frac{(\d n)_{\d x}}{(N-n_2)_{\d x}}
  \end{align*}
  (we have used (\ref{Pochhammer_gamma_trick}) again). Then we can set the complex $N$ to be equal to an integer again because both sides of the above identity are rational functions in $N$.

  Putting all together, we have 
  \begin{align*}
    J=\frac{\Gamma(N-n_2+\d x)}{\Gamma(\d x+1)\Gamma(N-n_2)}\frac{(\d n)_{\d x}}{(N-n_2)_{\d x}}=
    \frac{(\d n)_{\d x}}{\Gamma(\d x+1)}.
  \end{align*}
  This concludes the proof of the lemma and completes the proof of Theorem~\ref{thm:Kasteleyn}.
\end{proof}

Let us compute several related contour integrals which will be useful for the asymptotics at the edge (\S \ref{sec:asymptotics_at_the_edge}).

\begin{lemma}\label{lemma:Res_w=z_1}
  For all $1\le n_1\le N$, $1\le n_2\le N-1$, and $x_1,x_2\in\Z$, we have
  \begin{align}
    \label{Res_w=z_1}
    \frac{(N-n_1)!}{(N-n_2-1)!}
    \times\frac{1}{2\pi\i}&
    \oint\nolimits_{\Ga'(x_2)}
    \frac{(z-x_2+1)_{N-n_2-1}}{(z-x_1)_{N-n_1+1}}dz
    \\&=\nonumber
    (1_{x_1\ge x_2}-1_{n_1>n_2})
    \binom{n_1-n_2+x_1-x_2-1}{x_1-x_2},
  \end{align}
  where the contour $\Ga'(x_2)$ in $z$ is a clockwise contour which encircles points $x_2-1,x_2-2,\ldots$, and not points $x_2,x_2+1,\ldots$.
\end{lemma}
Here and below we understand that for integer $m$ and $r$, one has
$\binom{m+r}{r}=0$ if $m,r<0$. Otherwise if, say, $r\ge 0$, 
we agree that $\binom{m+r}{r}=\frac{(m+1)_r}{r!}$.
\begin{proof}
  The proof is analogous to that of Lemma \ref{lemma:Res_w=z}.
\end{proof}
Note that for the two integrals in (\ref{Res_w=z}) and (\ref{Res_w=z_1}) we have $\oint_{\Ga(x_2)}-\oint_{\Ga'(x_2)}=\oint_{\ga(\infty)}$, where the last integral is taken over a large enough counter-clockwise contour. Since in (\ref{Res_w=z}) and (\ref{Res_w=z_1}) the binomial coefficients in the right-hand sides are the same, we also have
\begin{align}
  \label{Res_w=z_2}
  \frac{(N-n_1)!}{(N-n_2-1)!}
  \times\frac{1}{2\pi\i}&
  \oint\limits_{\ga(\infty)}
  \frac{(z-x_2+1)_{N-n_2-1}}{(z-x_1)_{N-n_1+1}}dz
  =1_{n_1>n_2}\frac{(x_1-x_2+1)_{n_1-n_2-1}}{(n_1-n_2-1)!}.
\end{align}

\begin{lemma}\label{lemma:Res_w=z_3}
  For all $1\le n_1\le N$, $1\le n_2\le N-1$, and $x_1,x_2\in\Z$, we have 
  \begin{align}
    \frac{(N-n_1)!}{(N-n_2-1)!}\times
    \frac1{(2\pi\i)^{2}}
    \oint\limits_{\ga(\infty)}dz\oint\limits_{\ga(\infty)}dw
    &
    \nonumber
    \frac{(z-x_2+1)_{N-n_2-1}}{(w-x_1)_{N-n_1+1}}
    \frac{1}{w-z}
    \prod_{r=1}^{N}\frac{w+r-\nu_r}{z+r-\nu_r}
    \\& 
    =1_{n_1>n_2}\frac{(x_1-x_2+1)_{n_1-n_2-1}}{(n_1-n_2-1)!}.
    \label{Res_w=z_3}
  \end{align}
  This is the same integral as in the correlation kernel (\ref{kernel_formula}), but with the contour $\Ga(x_2)$ in $z$ replaced by a sufficiently large counter-clockwise contour which lies inside the $w$ contour.
\end{lemma} 
\begin{proof}
  Let us write the integral in $z$ in (\ref{Res_w=z_3}) as a sum over the residues at $z=\nu_j-j$, $j=1,\ldots,N$:
  \begin{align*}
    &
    \frac1{(2\pi\i)^{2}}
    \oint\limits_{\ga(\infty)}dz\oint\limits_{\ga(\infty)}dw
    \frac{(z-x_2+1)_{N-n_2-1}}{(w-x_1)_{N-n_1+1}}
    \frac{1}{w-z}
    \prod_{r=1}^{N}\frac{w+r-\nu_r}{z+r-\nu_r}
    \\&\qquad=
    \frac1{2\pi\i}
    \oint\limits_{\ga(\infty)}
    \frac{dw}{(w-x_1)_{N-n_1+1}}
    \sum_{j=1}^{N}
    (\nu_j-j-x_2+1)_{N-n_2-1}
    \prod_{r\ne j}\frac{w+r-\nu_r}{\nu_j-j+r-\nu_r}.
  \end{align*}
  The sum here is the Lagrange interpolation polynomial of the function $w\mapsto (w-x_2+1)_{N-n_2-1}$ with nodes $\nu_j-j$, $j=1,\ldots,N$. Since that function in $w$ is itself a polynomial of degree $\le N-1$, the interpolation is exact, and
  \begin{align*}
    \sum_{j=1}^{N}
    (\nu_j-j-x_2+1)_{N-n_2-1}
    \prod_{r\ne j}\frac{w+r-\nu_r}{\nu_j-j+r-\nu_r}
    =(w-x_2+1)_{N-n_2-1}.
  \end{align*}
  The claim now follows from (\ref{Res_w=z_2}).
\end{proof}

% subsection computing_some_contour_integrals (end)

% section inverse_kasteleyn_matrix (end)

\section{Asymptotics in the bulk, limit shape, and frozen boundary} % (fold)
\label{sec:asymptotics_in_the_bulk_limit_shape_and_frozen_boundary}

\subsection{Parameters of the polygon} % (fold)
\label{sub:parameters_of_the_polygon}

In this and the next section we perform asymptotic analysis of the uniform measure $\Pp_{\Pc(N)}$ on lozenge tilings of the polygon $\Pc(N)$ as $N\to\infty$. We will use the parameters $k$ and $\{A_i(N),B_i(N)\}_{i=1}^{k}\subset\Z'=\Z+\frac12$ describing the polygon $\Pc(N)$ (\S \ref{sec:model_and_results}). We will think that $k=2,3,\ldots$ is fixed, and $A_i(N),B_i(N)$ scale linearly with $N$ as in (\ref{scale_Ai_Bi}), depending on new continuous parameters $\{a_i,b_i\}_{i=1}^{k}$ with $a_1<b_1<\ldots<a_k<b_k$, and $\sum_{i=1}^{k}(b_i-a_i)=1$. These parameters describe the limiting polygon $\Pl$ in the new coordinates $(\chi,\eta)$ (Fig.~\ref{fig:frozen_boundary}). 

We fix a global position $(\chi,\eta)\in\Pl$, and obtain local asymptotics of the measures $\Pp_{\Pc(N)}$ around $(\chi,\eta)$ via asymptotic analysis of the correlation kernel $K$ of Theorem \ref{thm:K_intro}. We then discuss how these local asymptotics describe global properties of random lozenge tilings such as the limit shape and the frozen boundary. In this section we prove Theorem \ref{thm:bulk_intro} and Propositions \ref{prop:complex_Burgers_intro}, \ref{prop:frozen_boundary_intro}, and \ref{prop:tangent_points_intro}.

% subsection parameters_of_the_polygon (end)

\subsection{Asymptotic expression for the kernel in the bulk regime} % (fold)
\label{sub:asymptotics_of_the_kernel}

As a first step, we establish an asymptotically equivalent expression for the kernel $K$ which will allow to employ saddle point analysis in the spirit of \cite{Okounkov2002} (see also \cite{okounkov2003correlation}, \cite{Okounkov2005}, \cite{BorodinKuan2007U}). 

We will look at the kernel $K(x_1,n_1;x_2,n_2)$ with the parameters 
(depending on~$N$)
scaled as
\begin{align}\label{bulk_regime_12_1}
  {x_{1,2}}(N)/N\to \chi,\qquad {n_{1,2}}(N)/N\to \eta,
  \qquad N\to\infty,
\end{align}
such that the differences
\begin{align}\label{bulk_regime_12_2}
  \d x:=x_1-x_2\in\Z,\qquad
  \d n:=n_1-n_2\in\Z
\end{align}
stabilize. This is the so-called `bulk' limit regime which describes local asymptotic behavior of the measures $\Pp_{\Pc(N)}$ around the global position $(\chi,\eta)$.

\begin{definition}\label{def:action}
  Define the \emph{action} by
  \begin{align}
    S(w;\chi,\eta)&:=(w-\chi)\ln(w-\chi)
    -(w-\chi+1-\eta)\ln(w-\chi+1-\eta)
    \label{action}
    \\&\qquad
    \nonumber
    +(1-\eta)\ln(1-\eta)+\sum_{i=1}^{k}
    \Big[(b_i-w)\ln(b_i-w)-(a_i-w)\ln(a_i-w)\Big].
  \end{align}
  Unless otherwise stated, we assume that that the branches of all logarithms have cuts looking in negative direction along the real line. Note that the real part $\Re S(w;\chi,\eta)$ is well-defined and continuous for all $w\in\C$.
\end{definition}

\begin{proposition}\label{prop:asymptotics_of_the_kernel}
  In the regime (\ref{bulk_regime_12_1})--(\ref{bulk_regime_12_2}), the correlation kernel $K$ of Theorem~\ref{thm:K_intro} is asymptotically equivalent to
  \begin{align}&
    \nonumber
    K(x_1,n_1;x_2,n_2)
    \sim-1_{\d n>0}1_{\d x\ge0}\frac{(\d x+1)_{\d n-1}}{(\d n-1)!}
    +
    \frac{(1-\eta)^{1-\d n}}{(2\pi\i)^{2}}
    \oint\limits_{\Ga(\chi-)}dz\oint
    \limits_{\ga(\infty)}dw
    \times\\&\quad
    \times
    \frac{(w-\chi)^{-\d x-\frac12}(w-\chi+1-\eta)^{\d x+\d n-\frac12}}
    {(z-\chi)^{\frac12}
    (z-\chi+1-\eta)^{\frac12}}\cdot
    \frac{e^{N\big[S(w;\tfrac{x_2}N,\tfrac{n_2}N)-
    S(z;\tfrac{x_2}N,\tfrac{n_2}N)
    \big]}}{w-z}.
    \label{K_asymptotically_equivalent}
  \end{align}
  The branches of the square roots and other noninteger powers here are assumed to have cuts looking in negative direction along the real line.

  The $z$ contour $\Ga(\chi-)$ in (\ref{K_asymptotically_equivalent}) is counter-clockwise, it starts inside the segment $(\chi+\eta-1,\chi)$, goes in the upper half plane, crosses the real line again to the right of $b_k$,\footnote{Note that $(\chi,\eta)\in\Pl$ in particular means that $\chi<b_k$.} and returns (in the lower half plane) back to where it started. The counter-clockwise $w$ contour $\ga(\infty)$ contains $\Ga(\chi-)$ without intersecting it, and is sufficiently large. See contours on Figure \ref{fig:saddle_points_bulk}, left.
\end{proposition}
Note that by our choice of branches, the integrand in (\ref{K_asymptotically_equivalent}) is continuous on our contours except for possibly one real point lying on a cut.

\smallskip

In the rest of this subsection we prove Proposition \ref{prop:asymptotics_of_the_kernel}.

By scaling the variables of integration as $\tilde z=z/N$, $\tilde w=w/N$ in formula (\ref{K_intro}) for $K(x_1,n_1;x_2,n_2)$ (and renaming back), we can write
\begin{align}&
  K(x_1,n_1;x_2,n_2)=
  -1_{\d n>0}1_{\d x\ge0}\frac{(\d x+1)_{\d n-1}}{(\d n-1)!}
  \label{K_scale_of_variables}
  \\&
  \nonumber
  \qquad
  \qquad
  +
  \frac{1}{(2\pi\i)^{2}}
  \oint\limits_{\Ga(\chi-)}dz\oint\limits_{\ga(\infty)}dw
  \frac{1}{w-z}
  \frac{(1-\frac{n_1}N)}{(w-\frac{x_1}N)(w-\frac{x_1}N+1-\frac{n_1}N)}\frac{P(w;x_1,n_1)}{P(z;x_2,n_2)},
\end{align}
where 
\begin{align*}
  P(w;x,n):=\frac{(N-n-1)!}{(Nw-x+1)_{N-n-1}}
  \prod_{i=1}^{k}{(A_i+\tfrac12-Nw)_{B_i-A_i}}.
\end{align*} 

Let us explain why we can choose the contour $\Ga(\chi-)$ for $z$ in (\ref{K_scale_of_variables}). After the scaling, the new $z$ contour will contain inside it the points $\frac{x_2}N, \frac{x_2+1}N, \ldots \frac{B_k-1/2}{N}$, and the points $\frac{x_2-1}{N},\frac{x_2-2}{N},\ldots$ will be outside. It is possible to drag the left end of the contour slightly to the left because the factor $(Nz-x_2+1)_{N-n_2-1}$ inside $\frac{1}{P(z;x_2,n_2)}$ in (\ref{K_scale_of_variables}) compensates the corresponding poles. We can also drag the right end of the contour slightly to the right. 
Thus, we arrive at the ``macroscopic'' contour $\Ga(\chi-)$ for $z$ 
(in the sense that it does not depend on $N$). 
Clearly, one can chose the $w$ contour $\ga(\infty)$
to also be independent of $N$.
The new $z$ and $w$ contours in (\ref{K_scale_of_variables}) 
are given on Fig.~\ref{fig:saddle_points_bulk}, left.

Before going further, we need the following statement:
\begin{lemma}\label{lemma:aux_asympt}
  Let
  $\al$ and $\be$ depend on $N$ in such a way that
  $\al/N\to\tilde\al$,
  $\be/N\to\tilde\be$, where $\tilde\al,\tilde\be\in\R$. 
  Assume also that $\al-\be\in\Z$.
  Then we have the following two equivalences: 
  \begin{align}\label{gamma_eq_1}
    \frac{\Gamma(Nw-\al)}
    {\Gamma(Nw-\be)}&=
    \left(\frac{w- \frac{\be}{N}}
    {w- \frac{\al}{N}}\right)^{\frac12}
    \exp\Big\{
      N\Big[
      (\tfrac{\be}N-\tfrac{\al}N)
      (\ln N-1)
      +
      \\&
      \hspace{60pt}+
      (w-\tfrac{\al}N)\ln(w-\tfrac{\al}N)
      -
      (w-\tfrac{\be}N)\ln(w-\tfrac{\be}N)\Big]
      +O(\tfrac1N)
    \Big\},
    \nonumber
  \end{align}
  and 
  \begin{align}\label{gamma_eq_2}
    \frac{\Gamma(Nw-\al)}
    {\Gamma(Nw-\be)}&=
    \left(\frac{w- \frac{\be}{N}}
    {w- \frac{\al}{N}}\right)^{\frac12}
    \exp\Big\{
      N\Big[
      (\tfrac{\be}N-\tfrac{\al}N)
      (\ln N-1+\i\pi)
      +
      \\&
      \hspace{60pt}+
      (\tfrac{\be}N-w)\ln(\tfrac{\be}N-w)
      -
      (\tfrac{\al}N-w)\ln(\tfrac{\al}N-w)
      \Big]
      +O(\tfrac1N)
    \Big\}.
    \nonumber
  \end{align}
  Here
  $w\in\C\setminus(-\infty, \max(\tilde\al,\tilde\be)]$
  in \eqref{gamma_eq_1}, 
  and
  $w\in\C\setminus[\min(\tilde\al,\tilde\be),+\infty]$
  in \eqref{gamma_eq_2} (we do not consider
  the case $w\in[\min(\tilde\al,\tilde\be),\max(\tilde\al,\tilde\be)]$). 
  The quantities $O(\frac1N)$ are uniform 
  in $w$ belonging to compact subsets of the corresponding
  domains.
  The logarithms and the square root 
  are assumed to have 
  cuts looking in negative (in \eqref{gamma_eq_1}) 
  or positive (in \eqref{gamma_eq_2})
  directions along the real line.
\end{lemma}
  The right-hand sides of \eqref{gamma_eq_1} and \eqref{gamma_eq_2}
  are equal for $w\in\C\setminus\R$.
\begin{proof}
  To obtain \eqref{gamma_eq_1}, we directly
  apply the Stirling approximation 
  (which holds for $y\notin(-\infty,0]$):
  \begin{align}\label{Stirling_gamma}
    \Gamma(y)=\exp\left(
    (y-\tfrac12)\ln y-y+\tfrac12\ln(2\pi)+O(\tfrac1y)
    \right),\qquad |y|\to\infty.
  \end{align}
  The $O(\frac1y)$ is uniform in $y$ belonging to compact
  subsets of $\C\setminus (-\infty,0]$.
  
  Equivalence \eqref{gamma_eq_2} also 
  follows from the Stirling approximation
  if one rewrites 
  $
    \frac{\Gamma(Nw-\al)}{\Gamma(Nw-\be)}
    =
    (-1)^{\be-\al}\frac{\Gamma(\be+1-Nw)}{\Gamma(\al+1-Nw)}
  $
  which is possible for $w\in\C\setminus\R$ because
  $\al-\be$ is an integer.
\end{proof}

\begin{lemma}
  \label{lemma:asymptotics_of_P}
  Let $x/N\to\chi$ and $n/N\to\eta$ as $N\to\infty$
  (where $x$ and
  $n$ depend on $N$ in some way). 
  Then
  \begin{align*}
    P(w;x,n)&=\mathrm{const}_N
    \left(\frac{(w-\tfrac xN)(w-\tfrac xN+1-\tfrac nN)}
    {1-\tfrac nN}\right)^{\frac12}
    \exp\Big( NS(w;\tfrac xN,\tfrac nN)+O(\tfrac 1N)\Big)
  \end{align*}
  for $w\in\C\setminus\R$. Here $S$ is given by (\ref{action}), and
  $\mathrm{const}_N$ denotes a constant which does not depend on $w,x,n$, but may depend on $N$ and $\{A_i,B_i\}$. 
  Clearly, such a constant in $P$ does not affect the integrand in (\ref{K_scale_of_variables}).
\end{lemma}
\begin{proof}
  This follows (after a simplification) 
  from Lemma \ref{lemma:aux_asympt} if one writes 
  each Pochhammer symbol as a ratio of two Gamma functions, namely, 
  $(\al)_m=\Gamma(\al+m)/\Gamma(\al)$.
  In particular, the choice of the constants $\delta_i,\delta_i'$ in 
  \eqref{scale_Ai_Bi} does not affect the asymptotical expression for 
  $P(w;x,n)$ apart from $\mathrm{const}_N$.
\end{proof}

\begin{lemma}\label{lemma:exponential_factor_in_P/P}
  In the regime (\ref{bulk_regime_12_1})--(\ref{bulk_regime_12_2}), 
  we have for $w\in\C\setminus\R$:
  \begin{align*}&
    \exp\Big(N\cdot S(w;\tfrac{x_1}N,\tfrac{n_1}N)\Big)
    \sim
    \frac{
    (w-\frac{x_2}N+1-\frac{n_2}N)^{\d x+\d n}}
    {(1-\frac{n_2}N)^{\d n}
    (w-\frac{x_2}N)^{\d x}}
    \exp\Big(N\cdot S(w;\tfrac{x_2}N,\tfrac{n_2}N)\Big).
  \end{align*}
\end{lemma}
That is, knowing the exact distance between
$(x_1,n_1)$ and $(x_2,n_2)$ from 
(\ref{bulk_regime_12_1})--(\ref{bulk_regime_12_2}),
we can express the exponential term depending on $(x_1,n_1)$
in terms of $(x_2,n_2)$.
\begin{proof}
  We can write
  \begin{align*}
    S(w;\tfrac{x_1}N,\tfrac{n_1}N)
    &=
    S(w;\tfrac{x_2}N+\tfrac{\d x}N,\tfrac{n_2}N+\tfrac{\d n}N)
    \\&=
    S(w;\tfrac{x_2}N,\tfrac{n_2}N)+
    \tfrac{\d x}N S_\chi(w;\tfrac{x_2}N,\tfrac{n_2}N)
    +
    \tfrac{\d n}N S_\eta(w;\tfrac{x_2}N,\tfrac{n_2}N)+O(\tfrac1{N^2}),
  \end{align*}
  where the derivatives are given by 
  \begin{align*}
    S_\chi(w;\chi,\eta)
    =\tfrac{\partial}{\partial\chi}S_\chi(w;\chi,\eta)&=\ln(w-\chi+1-\eta)-\ln(w-\chi),\\
    S_\eta(w;\chi,\eta)
    =\tfrac{\partial}{\partial\eta}S_\chi(w;\chi,\eta)
    &=\ln(w-\chi+1-\eta)-\ln(1-\eta).
  \end{align*}
  This concludes the proof.
\end{proof}

\par\noindent
\emph{Proof of Proposition \ref{prop:asymptotics_of_the_kernel}.}
Fix macroscopic contours for $w$ and $z$ as in \eqref{K_scale_of_variables}
which do not depend on $N$, and then on the contours apply our asymptotical equivalences of Lemmas \ref{lemma:asymptotics_of_P} and \ref{lemma:exponential_factor_in_P/P}. 
To justify the application of equivalences under the contour
integrals, one can split each of the contours for $w$ and $z$ into two parts
by two of points of the form $\si\pm\i t$, $t>0$. 
On each part of each contour, there is an estimate with 
constant $O(\frac1N)$ uniform in the integration variable.
Such an estimate follows by taking an appropriate 
analytic expression in the right-hand side of either
\eqref{gamma_eq_1} or \eqref{gamma_eq_2} for each of the Pochhammer symbols under the integral (in fact, this is allowed by the contours
in \eqref{K_scale_of_variables}). 
The resulting equivalence 
can be written in one form \eqref{K_asymptotically_equivalent}
because
different analytic expressions coming from \eqref{gamma_eq_1} 
or \eqref{gamma_eq_2} coincide for $w,z\in\C\setminus\R$, and one can ignore
two real points on each of the contours.
Outside the exponent in, we clearly can replace $\frac{x_{1,2}}N$ and $\frac{n_{1,2}}{N}$ by $\chi$ and $\eta$, respectively.
\qed

% subsection asymptotics_of_the_kernel (end)

\subsection{Critical points of the action $S(w;\chi,\eta)$} % (fold)
\label{sub:critical_points_of_the_action_s_w_chi_eta_}

As the saddle point technique suggests \cite{Okounkov2002}, to analyse the asymptotics of the double contour integral for the pre-limit kernel (\ref{K_asymptotically_equivalent}), we need to deal with the \emph{critical points} of the action $S(w;\chi,\eta)$ (Definition \ref{def:action}), i.e., solutions to
\begin{align}\label{critical_points_equation}
  \tfrac{\partial}{\partial w}S(w,\chi,\eta)=0.
\end{align}
\begin{proposition}\label{prop:number_of_roots}
  For every point $(\chi,\eta)$ inside the polygon $\Pl$ (defined by $\{a_i,b_i\}$, see Fig.~\ref{fig:frozen_boundary}), the function $S(w,\chi,\eta)$ in $w$ has either 2 or 0 nonreal critical points.
\end{proposition}
\begin{proof}
  The equation (\ref{critical_points_equation}) for the critical points is equivalent to the following algebraic equation of degree $k$:
  \begin{align}\label{alg_equation_for_w_c}
    (w-\chi)\prod\nolimits_{i=1}^{k}(w-a_i)=
    (w-\chi+1-\eta)\prod\nolimits_{i=1}^{k}(w-b_i).
  \end{align}
  Let us denote by $Q_a(w)$ and $Q_b(w)$ the polynomials on the left and on the right, respectively (an example shown on Fig.~\ref{fig:2polys}).\begin{figure}[htbp]
    \begin{tabular}{cc}
    \includegraphics[width=170pt]{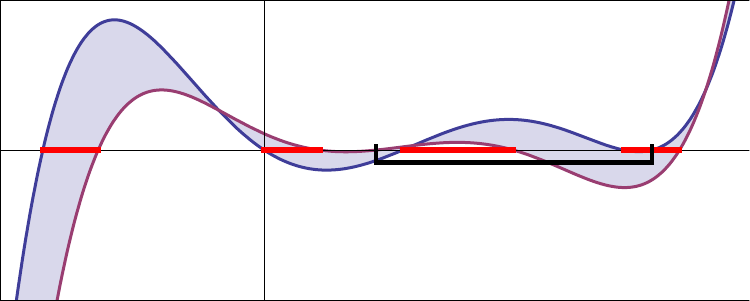}&
    \includegraphics[width=170pt]{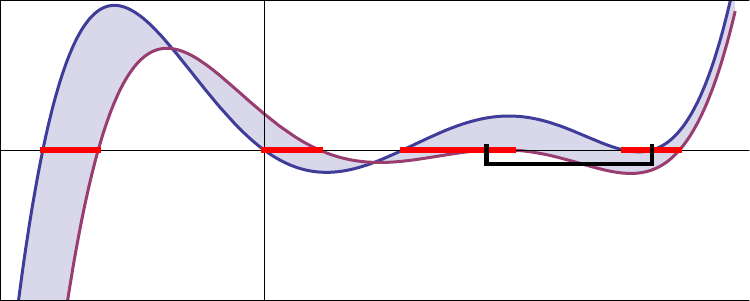}
    \end{tabular}
    \caption{Example with $k=4$. The polynomials $Q_a(w)$ and $Q_b(w)$ intersect $4$ (left; the fourth intersection point is far to the right) or $2$ (right) times depending on $(\chi,\eta)$.}
    \label{fig:2polys}
  \end{figure} We aim to count the number of intersections of graphs of $Q_a(w)$ and $Q_b(w)$, $w\in\R$. Since
  $\{a_i\}$ and $\{b_i\}$ interlace, at each of the 
  $k-1$ segments $[b_i,b_{i+1}]$
  (with possible exceptions for
  the segments containing $\chi$ and $\chi+\eta-1$) 
  there is at least one real root
  of \eqref{alg_equation_for_w_c}. 
  This gives us $k-3$ roots. Since the degree of equation
  is $k$, we conclude that there is at 
  most one pair of complex conjugate
  roots.
\end{proof}

\begin{remark}
  In figures below where we show contours of integration or other relevant objects, the segments $[a_i,b_i]$ and the distinguished segment $[\chi+\eta-1,\chi]$ are also present. We will always follow the way these segments are shown on Fig.~\ref{fig:2polys}: there are several red segments $[a_i,b_i]$ lying on the horizontal line, and one black segment $[\chi+\eta-1,\chi]$ which is shown slightly under that line. Its the endpoints $\chi+\eta-1$ and $\chi$ are indicated by small vertical marks crossing the horizontal line.
\end{remark}

\begin{definition}\label{def:liquid_region_omc}
  According to Proposition \ref{prop:number_of_roots}, let $\D\subset\Pl$ be the (in fact, open) set of pairs $(\chi,\eta)$ for which $S(w;\chi,\eta)$ has nonreal critical points. This set $\D$ is called the \emph{liquid region}, it lies inside the \emph{frozen boundary curve} $\partial\D$. See Fig.~\ref{fig:frozen_boundary} and~\S\S \ref{sub:asymptotics_in_the_bulk}--\ref{sub:limit_shape_complex_burgers_equation_frozen_boundary}. 

  For $(\chi,\eta)\in\D$, let $\om=\om(\chi,\eta)$ denote the unique nonreal critical point of $S(w;\chi,\eta)$ lying in the upper half plane.
\end{definition}

% subsection critical_points_of_the_action_s_w_chi_eta_ (end)

\subsection{Moving the contours} % (fold)
\label{sub:moving_the_contours}

Fix a point $(\chi,\eta)\in\D$. We aim to move the contours in the double contour integral for the correlation kernel $K(x_1,n_1;x_2,n_2)$ (\ref{K_asymptotically_equivalent}) so that the exponent $\exp\Big({N\big[S(w;\tfrac{x_2}N,\tfrac{n_2}N)-S(z;\tfrac{x_2}N,\tfrac{n_2}N)\big]}\Big)$ will make the whole integral go to zero. This is achieved by making the $z$ and $w$ contours cross at the two simple nonreal critical points $\om,\omb$ of $S(w;\chi,\eta)$. These
points are also the saddle points of $\Re S(w;\chi,\eta)$, so on the new transformed contours we have
\begin{align}\label{contours_Re_condition}
  \Re S(w;\tfrac{x_2}N,\tfrac{n_2}N)< 
  \Re S(\om;\chi,\eta)
  <\Re S(z;\tfrac{x_2}N,\tfrac{n_2}N), \qquad
  z,w\ne\om
\end{align}
(see, e.g.,~\cite[\S3]{Okounkov2002} for more detail). In the course of moving the contours, certain residues corresponding to $w=z$ will appear.

\begin{proposition}\label{prop:moving_bulk}
  It is always possible to transform the $w$ and $z$ contours in the double contour integral for $K(x_1,n_1;x_2,n_2)$ (\ref{K_asymptotically_equivalent}) such that (\ref{contours_Re_condition}) holds (Fig.~\ref{fig:saddle_points_bulk}, right). After this transformation, an additional term that converges to
  \begin{align}\label{bulk_sine_residue}
    \frac{(1-\eta)^{1-\d n}}{2\pi\i}
    \int_{\omb}^{\om}
    (z-\chi)^{-\d x-1}(z-\chi+1-\eta)^{\d x+\d n-1}dz
  \end{align}
  will arise. In (\ref{bulk_sine_residue}), the contour of integration crosses $(\chi,+\infty)$.
\end{proposition}
\begin{figure}[htbp]
  \begin{tabular}{cc}
    \includegraphics[width=145pt]{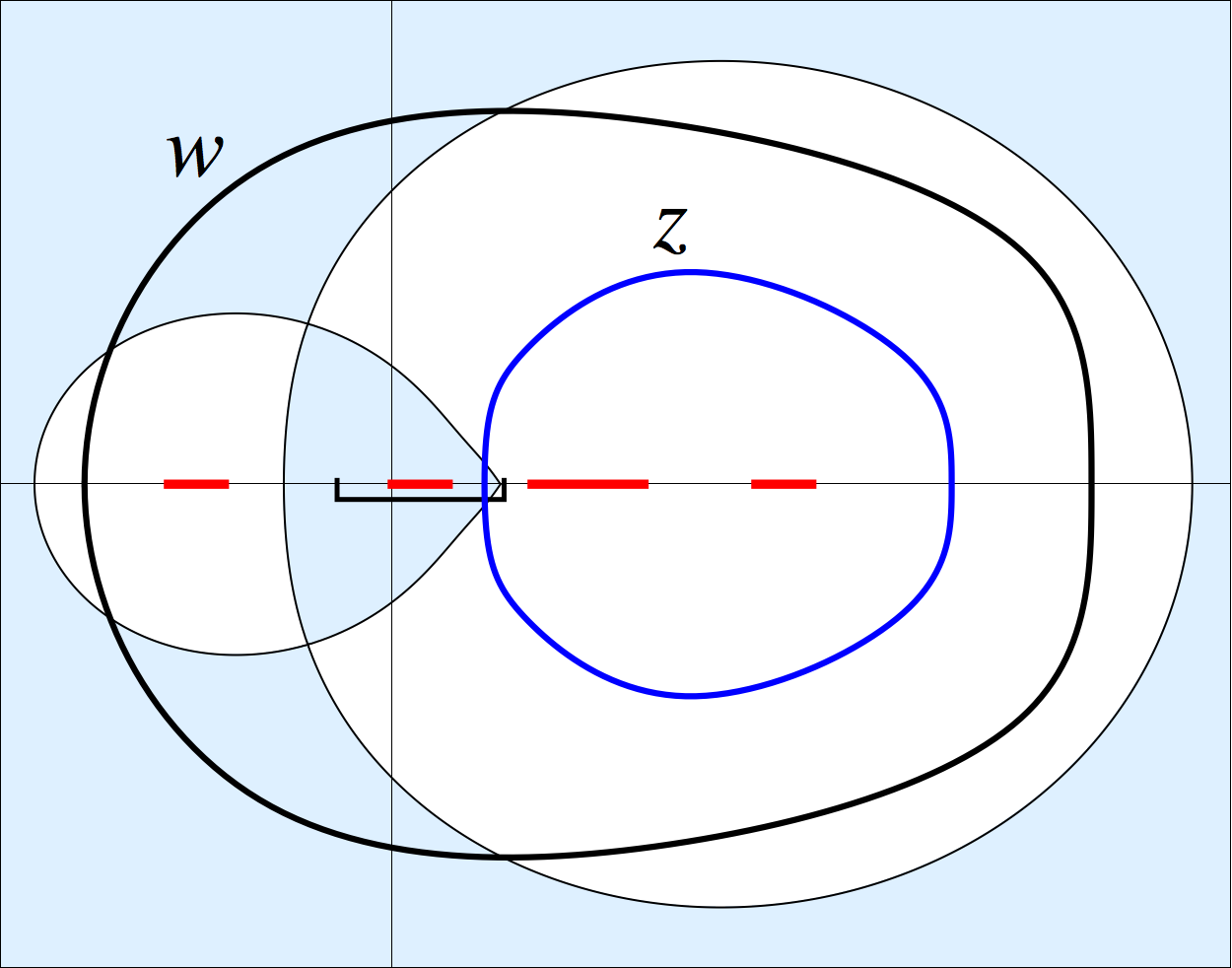}&
    \includegraphics[width=145pt]{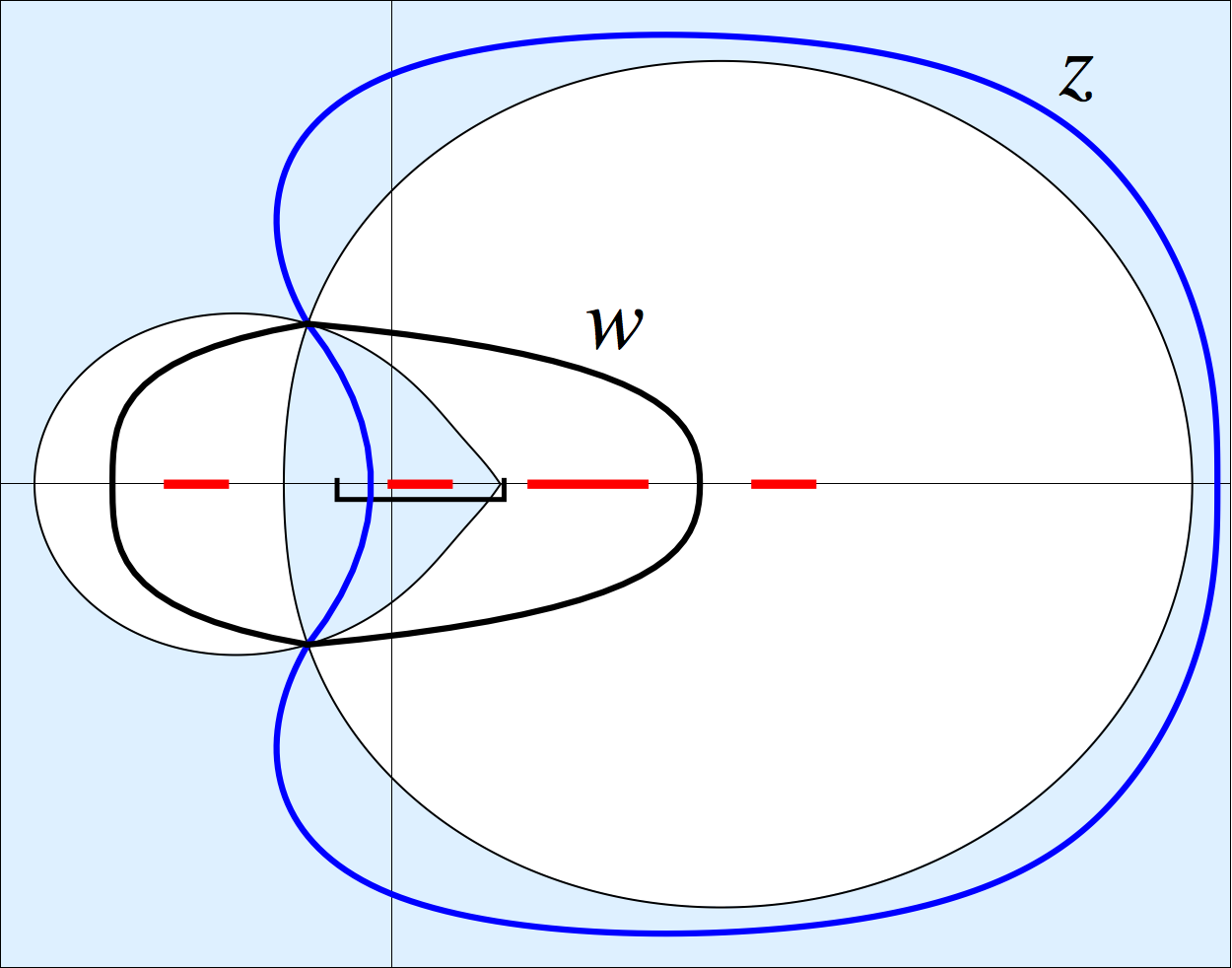}
  \end{tabular}
  \caption{The contours before (left) and after (right) transformation. Shaded is the region where $\Re(S(w;\chi,\eta)-S(\om;\chi,\eta))>0$.}
  \label{fig:saddle_points_bulk}
\end{figure}
\begin{proof}
  First, observe that 
  one can drag the $z$ contour $\Ga(\chi-)$ in 
  \eqref{K_asymptotically_equivalent} everywhere along the real line except for the segments $[a_i,b_i]$ which are not covered by $[\chi+\eta-1,\chi]$. Likewise, one can drag the $w$ contour $\ga(\infty)$ in \eqref{K_asymptotically_equivalent} 
  everywhere except for the part of $[\chi+\eta-1,\chi]$ outside the segments $[a_i,b_i]$.
  These possibilities 
  to move the contours 
  become especially obvious
  if one looks at the original formula \eqref{K_intro} for the kernel and 
  considers zeroes and poles of the integrand (the restrictions
  arise because we do not want to pick up any 
  residues while dragging the contours). 
  One can think that we first drag the contours of \eqref{K_intro}
  into new positions (Fig.~\ref{fig:saddle_points_bulk}),
  and then apply uniform asymptotic equivalences as explained
  in the proof of Proposition \ref{prop:asymptotics_of_the_kernel}
  in \S \ref{sub:asymptotics_of_the_kernel}.

  \begin{figure}[htbp]
    \begin{tabular}{c}
      \includegraphics[width=180pt]{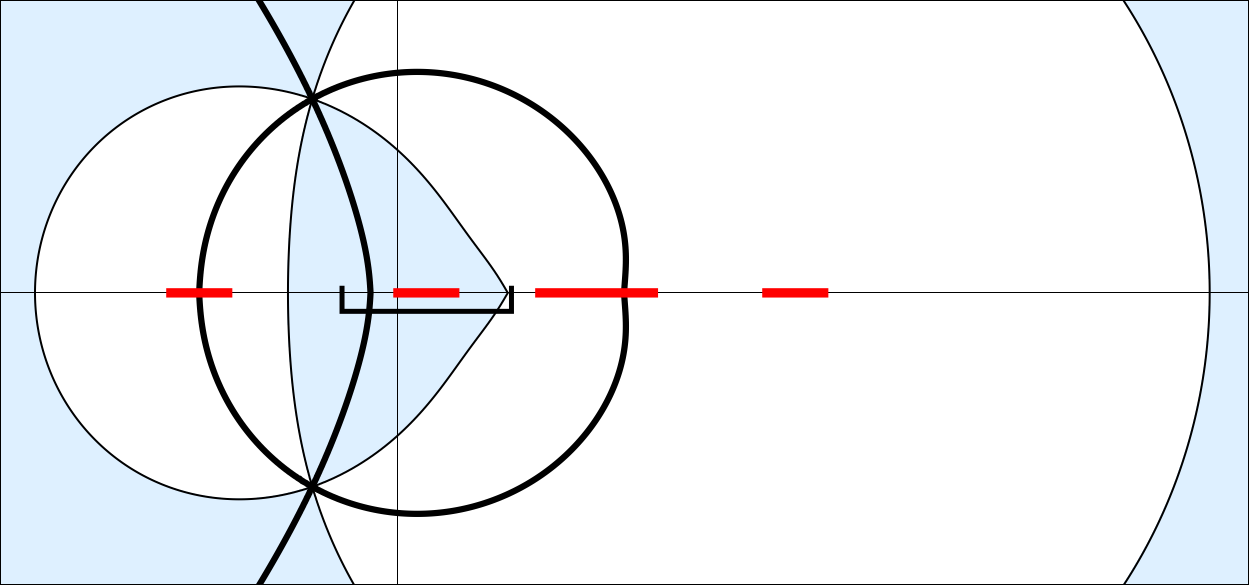}
    \end{tabular}
    \caption{Curves along which $\Im(S(w;\chi,\eta)-S(\om;\chi,\eta))=0$. Each such curve from $\omb$ to $\om$ curve completely belongs to a shaded or not shaded region.}
    \label{fig:Im_bulk}
  \end{figure}
  We now aim to justify that the picture of shaded regions where 
  $\Re(S(w)-S(\om))>0$ looks exactly as on Fig.~\ref{fig:saddle_points_bulk}
  and \ref{fig:Im_bulk}, and also describe the points where the four contours $\{u\colon\Im S(u)=\Im S(\om)\}$ intersect the real line (see Fig.~\ref{fig:Im_bulk}).
  First, note that for any $(\chi,\eta)\in\D$, the real part $\Re S(u;\chi,\eta)$ goes to $+\infty$ as $|u|\to\infty$ because $\Re S(u;\chi,\eta)\sim \eta\ln|u|$. This implies that far away on Fig.~\ref{fig:saddle_points_bulk}
  one sees a shaded region.

  Next, since $S(u)$ is holomorphic in the upper half plane, 
  along each of the four contours $\{u\colon\Im S(u)=\Im S(\om)\}$ 
  (see Fig.~\ref{fig:Im_bulk}) the sign of $\Re(S(u)-S(\om))$ must be constant. This implies that each curve on Fig.~\ref{fig:Im_bulk} 
  from $\om$ to $\omb$ must be completely inside a shaded or not shaded region.

  Now let us look at the function $\Im (S(u)-\Im S(\om))$ for 
  $u\in\R+\i \epsilon$ for fixed small $\epsilon>0$. Observe that
  \begin{align*}
    \Im( (t+\i \epsilon)\ln(t+\i\epsilon))=
    \epsilon\ln|t+\i\epsilon|+t \arg(t+\i\epsilon)
    \sim 
    \pi \cdot(t)_-:=\pi \cdot t 1_{t<0},
  \end{align*}
  where $t\in\R$, as $\epsilon\to0+$. Thus,
  \begin{align*}
    \tfrac{1}{\pi}\Im S(t+\i\epsilon)\sim(t-\chi)_{-}-
    (t-\chi+1-\eta)_{-}
    +\sum\nolimits_{j=1}^{k}[(b_j-t)_{-}-(a_j-t)_{-}].
  \end{align*}
  The graph of this function is shown on Fig.~\ref{fig:Im_S}.\begin{figure}[htbp]
    \includegraphics[width=220pt]{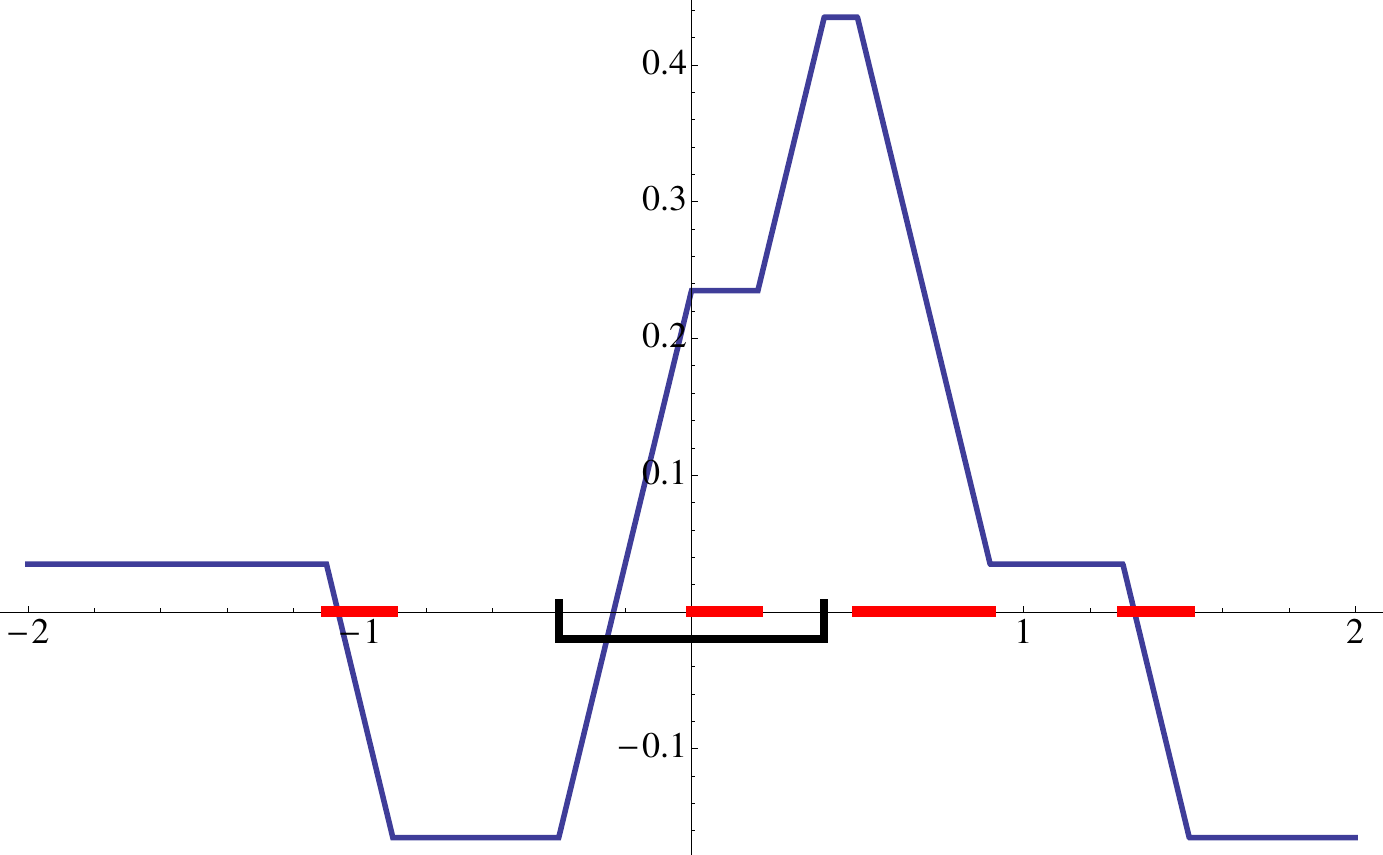}
    \caption{Graph of $\tfrac{1}{\pi}\Im (S(t+\i\epsilon)-S(\om))$, $t\in\R$, for small $\epsilon>0$. The segments $[a_j,b_j]$ (red) and $[\chi+\eta-1,\chi]$ (black) are displayed.}
    \label{fig:Im_S}
  \end{figure} 
  Clearly,
  $$\tfrac{\partial}{\partial t}\tfrac{1}{\pi}\Im S(t+\i\epsilon)\sim1_{t\in[\chi+\eta-1,\chi]}-\sum\nolimits_{j=1}^{k}1_{t\in[a_j,b_j]}.$$ 
  Looking at the slopes of the graph on Fig.~\ref{fig:Im_S}, we see that the four contours $\{u\colon\Im S(u)=\Im S(\om)\}$ can intersect the real line in at most three points.\footnote{In fact, the case when there are infinitely many such points (i.e., when a horizontal part of the graph on Fig.~\ref{fig:Im_S} is lying at the horizontal coordinate line) corresponds to $(\chi,\eta)$ belonging to the frozen boundary, when $S$ has a real double critical point.} Because of the relation between these contours and the shaded regions on Fig.~\ref{fig:saddle_points_bulk} explained above, there are exactly three such points of intersection:
  \begin{enumerate}[$\bullet$]
    \item $t^+\in[\chi+\eta-1,\chi]$, where $\Im S(t^+)=\Im S(\om)$ and $\Re S(t^+)>\Re S(\om)$;
    \item $t_l^-<t_r^-$, both belonging to the union of the segments $[a_j,b_j]$, where $\Im S(t^-_{l,r})=\Im S(\om)$ and $\Re S(t^-_{l,r})<\Re S(\om)$.
  \end{enumerate}
  Moreover, from Fig.~\ref{fig:Im_S} we see that $t_l^-<t^+<t_r^-$. The fourth contour $\{u\colon\Im S(u)=\Im S(\om)\}$ (with $\Re S(u)>\Re S(\om)$) runs to infinity.

  Dragging the $w$ contour through the $z$ one into the new positions
  as on Fig.~\ref{fig:saddle_points_bulk} (see also Fig.~\ref{fig:Im_bulk}),
  which is possible as explained above,
  in the $w$ integral in (\ref{K_asymptotically_equivalent}) we pick up the residue at $w=z$ for $z$ between $\omb$ and $\om$. Then we integrate this residue over $z$, which results in (\ref{bulk_sine_residue}).
  This concludes the proof.
\end{proof}

% subsection moving_the_contours (end)

\subsection{Proof of Theorem \ref{thm:bulk_intro}} % (fold)
\label{sub:asymptotics_in_the_bulk_and_theorem_ref}

Now we can finalize the proof of Theorem \ref{thm:bulk_intro} formulated in \S \ref{sub:asymptotics_in_the_bulk}. Recall the incomplete beta kernel $\B_{\Om}(m,l)$ with complex parameter $\Om\in\C$, $\Im\Om\ge 0$ (Definition \ref{def:incomplete_beta}). We will prove the convergence of the correlation kernel $K(x_1,n_1;x_2,n_2)$ (\ref{K_intro}) to the incomplete beta kernel which will readily imply Theorem \ref{thm:bulk_intro} via (\ref{correlation_kernel_intro}).

\begin{proposition}
  For $(\chi,\eta)\in\D$, in the bulk limit regime (\ref{bulk_regime_12_1})--(\ref{bulk_regime_12_2}), we have the following convergence:
  \begin{align*}
    K(x_1,n_1;x_2,n_2)\to\B_{\Om}(n_2-n_1,x_1-x_2),
  \end{align*}
  where the parameter $\Om=\Om(\chi,\eta)$ of the incomplete beta kernel is given by
  \begin{align}\label{Om_definition}
    \Om(\chi,\eta):=
    \frac{\om(\chi,\eta)-\chi}{\om(\chi,\eta)-\chi+1-\eta}.
  \end{align}
\end{proposition}
\begin{proof}
  Transforming the contours as in Proposition \ref{prop:moving_bulk}, with the help of Proposition \ref{prop:asymptotics_of_the_kernel} we see that 
  \begin{align*}&
    K(x_1,n_1;x_2,n_2)
    \sim-1_{\d n>0}1_{\d x\ge0}\frac{(\d x+1)_{\d n-1}}{(\d n-1)!}
    \\&\qquad \qquad+
    \frac{(1-\eta)^{1-\d n}}{2\pi\i}
    \int_{\omb}^{\om}
    \frac{(z-\chi)^{-\d x-1}}{(z-\chi+1-\eta)^{-\d x-\d n+1}}dz
    \\&\qquad \qquad+
    \frac{(1-\eta)^{1-\d n}}{(2\pi\i)^{2}}
    \oint\oint
    \frac{(w-\chi)^{-\d x-\frac12}(w-\chi+1-\eta)^{\d x+\d n-\frac12}}
    {(z-\chi)^{\frac12}
    (z-\chi+1-\eta)^{\frac12}}
    \times\\&
    \qquad\qquad\qquad\qquad\qquad\qquad\qquad\qquad\qquad\qquad
    \times
    \frac{e^{N\big[S(w;\tfrac{x_2}N,\tfrac{n_2}N)-
    S(z;\tfrac{x_2}N,\tfrac{n_2}N)
    \big]}}{w-z}dzdw.
  \end{align*}
  The first two summands give $\B_\Om(-\d n, \d x)$:
  (1) for $\d n>0$, drag the contour in $\int_{\omb}^{\om}$ through the point $\chi$, this will result in the residue $1_{\d x\ge0}\frac{(\d x+1)_{\d n-1}}{(\d n-1)!}$ which will compensate the first summand; (2) in the integral substitute $u=\frac{z-\chi}{z-\chi+1-\eta}$, this will give the integral for $\B_\Om$ of Definition \ref{def:incomplete_beta}.

  The contours in the third summand above are as on Fig.~\ref{fig:saddle_points_bulk}, right. On these contours, 
  the exponent has the form $N$ times a function 
  having negative real part (and nonzero second derivative). 
  Thus, 
  the third integral goes to zero.
  See \cite[\S3.1]{Okounkov2002} for more detail.
\end{proof}

For $(\chi,\eta)\notin\D$, in the bulk regime, the kernel $K(x_1,n_1;x_2,n_2)$ also converges to the incomplete beta kernel $\B_\Om$, but with real $\Om$. This can be seen also by moving the contours as in Proposition \ref{prop:moving_bulk}, but now all saddle points would be located on the real line. One of the transformed contours should pass through a saddle point, and it will contain or will be contained inside the other contour. The residue picked up after the transformation will be either 0, or the integral (\ref{bulk_sine_residue}) over a closed contour. This exactly corresponds to $\B_\Om$ with real $\Om$. A detailed treatment of contours for the case of real saddle points for a similar double contour integral is performed in~\cite[\S4.4]{BorodinKuan2007U}.

We have now established Theorem \ref{thm:bulk_intro}.

% subsection asymptotics_in_the_bulk_and_theorem_ref (end)

\subsection{Limit shape and the complex Burgers equation} % (fold)
\label{sub:limit_shape_and_the_complex_burgers_equation}

Here we prove Proposition \ref{prop:complex_Burgers_intro}. From (\ref{Om_definition}) and (\ref{alg_equation_for_w_c}) it readily follows that the complex slope $\Om=\Om(\chi,\eta)$ satisfies the algebraic equation
\begin{align*}
  \Om\cdot\prod\limits_{i=1}^{k}
  \big(
  (a_i-\chi+1-\eta)\Om-(a_i-\chi)
  \big)=
  \prod\limits_{i=1}^{k}
  \big(
  (b_i-\chi+1-\eta)\Om-(b_i-\chi)
  \big).
\end{align*}
This equation has degree $k+1$ in contrast with the equation (\ref{alg_equation_for_w_c}) for $\om$ which has degree $k$. However, the above equation for $\Om$ always has a root $\Om=1$, so it can be reduced to a certain degree $k$ equation. 

\begin{proposition}[Complex Burgers equation]
\label{prop:complex_Burgers_maintext}
  In addition to the above algebraic equation, $\Om(\chi,\eta)$ also satisfies the complex Burgers equation \cite{OkounkovKenyon2007Limit}:  
  \begin{align*}
    \Om(\chi,\eta)\frac{\partial\Om(\chi,\eta)}{\partial\chi}=
    -(1-\Om(\chi,\eta))\frac{\partial\Om(\chi,\eta)}{\partial\eta}.
  \end{align*}
\end{proposition}
\begin{proof}
  It is more suitable to start with a differential equation for the parameter $\om(\chi,\eta)$ of Definition \ref{def:liquid_region_omc}. The equation for the critical points of the action $S(w;\chi,\eta)$ (\ref{action}) which is satisfied by $\om$ can be written as 
  \begin{align}\label{alg_equation_for_w_c2}
    \ln(w-\chi)-\ln(w-\chi+1-\eta)=\sum_{i=1}^{k}
    \big(\ln(b_i-w)-\ln(a_i-w)\big)
  \end{align}
  (it is of course equivalent to (\ref{alg_equation_for_w_c})). Differentiating it with respect to $\chi$ and $\eta$, we obtain:
  \begin{align*}
    \frac{(\om)_{\chi}-1}{\om-\chi}-\frac{(\om)_{\chi}-1}{\om-\chi+1-\eta}
    &=
    (\om)_{\chi}\cdot
    \Sigma(\om),\\
    \frac{(\om)_{\eta}}{\om-\chi}-\frac{(\om)_{\eta}-1}{\om-\chi+1-\eta}
    &=
    (\om)_{\eta}\cdot
    \Sigma(\om)
  \end{align*}
  (see (\ref{frozen_param2}) for notation). From these two equations one can deduce that $\om$ satisfies
  \begin{align*}
    \frac{\om(\chi,\eta)-\chi}{1-\eta}\cdot
    \frac{\partial\om(\chi,\eta)}{\partial\chi}=-
    \frac{\partial\om(\chi,\eta)}{\partial\eta}.
  \end{align*}
  Then, using the connection between $\Om(\chi,\eta)$ and $\om(\chi,\eta)$ (\ref{Om_definition}), we see that the desired equation holds.
\end{proof}
There are certain boundary conditions satisfied by the complex slope $\Om(\chi,\eta)$ which can be deduced from the study of local asymptotics (Theorem \ref{thm:bulk_intro}). 
Namely (see \S \ref{sub:frozen_boundary} below for more detail), 
the limit shape has frozen facets outside the curve $\partial\D$ which is the same curve obtained as a frozen boundary in \cite{OkounkovKenyon2007Limit}. Together with these boundary conditions, the complex Burgers equation ensures that the complex slope $\Om(\chi,\eta)$ of the limiting ergodic translation invariant Gibbs measure coincides with that of the tangent plane to the limit shape of \cite{CohnKenyonPropp2000}, \cite{OkounkovKenyon2007Limit} at $(\chi,\eta)$. That is, the `hypothetical' limit shape for our model inferred from the asymptotics of correlation functions and, in particular, from the limiting density function $\B_\Om(0,0)=\arg(\Om)/\pi$ (see Fig.~\ref{fig:lozenge_types}), coincides with the actual limit shape.

% subsection limit_shape_and_the_complex_burgers_equation (end)

\subsection{Frozen boundary} % (fold)
\label{sub:frozen_boundary}

Let us now discuss the frozen boundary curve $\partial\D$. We obtain its explicit rational parametrization, and identify tangent points on it, thus proving Propositions \ref{prop:frozen_boundary_intro} and \ref{prop:tangent_points_intro}.

\begin{figure}[htbp]
  \begin{tabular}{c}
    \includegraphics[height=110pt]{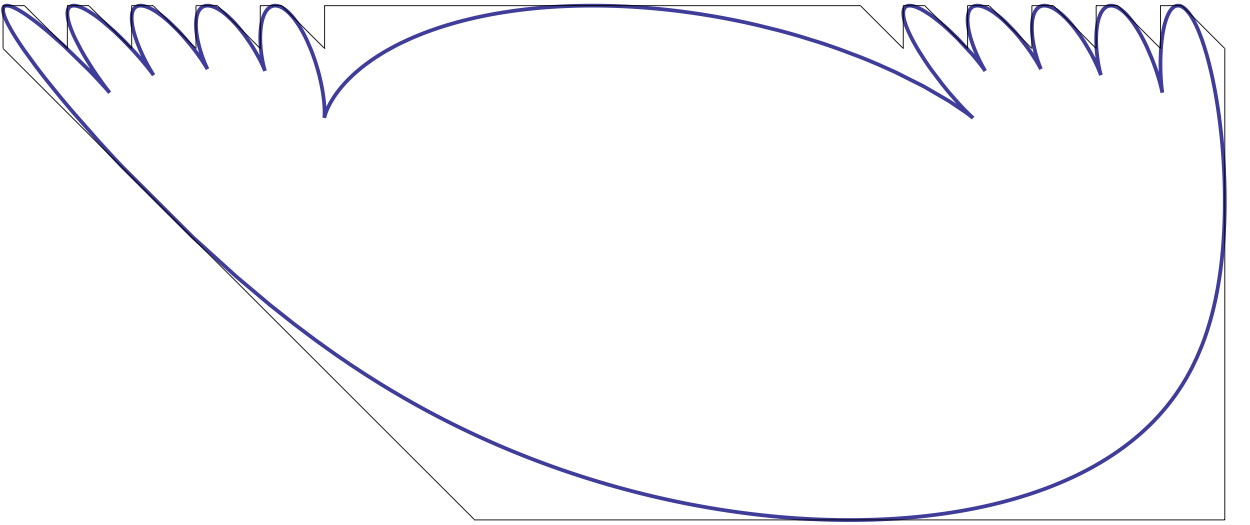}
  \end{tabular}
  \caption{Example of the frozen boundary curve with $k=12$ (inscribed in a $36$-gon).}
  \label{fig:frozen_boundary_big}
\end{figure}

\medskip

\noindent\textit{Proof of Proposition \ref{prop:frozen_boundary_intro}.}
When a point $(\chi,\eta)\in\D\subset\Pl$ approaches the frozen boundary curve $\partial\D$, the parameter $\om(\chi,\eta)$ --- the critical point of the action $S(w;\chi,\eta)$ (\ref{action}) in the upper half plane --- becomes real and merges with $\omb(\chi,\eta)$. This means that the frozen boundary can be characterized as a set $(\chi,\eta)$ such that $S(w;\chi,\eta)$ has a double critical point. In fact, it is more convenient to take the double critical point $\om$ itself as the parameter on $\partial\D$, and express $\chi$ and $\eta$ through it. We will now establish the desired parametrization (\ref{frozen_param1})--(\ref{frozen_param2}).

The equations $\frac{\partial}{\partial w}S(w;\chi,\eta)=\frac{\partial^{2}}{\partial w^{2}}S(w;\chi,\eta)=0$ for the double critical points of $S(w;\chi,\eta)$ can be rewritten as (see (\ref{frozen_param2}) for notation):
\begin{align*}
  \frac{\om-\chi}{\om-\chi+1-\eta}&=\Pi(\om),\qquad
  \frac{1}{\om-\chi}-\frac{1}{\om-\chi+1-\eta}=\Sigma(\om).
\end{align*}
The first equation is just (\ref{alg_equation_for_w_c}), and the second one is obtained by differentiating (\ref{alg_equation_for_w_c2}) with respect to $w$. The above two equations can be reduced to linear ones in $\chi$ and $\eta$, and solving them one arrives at (\ref{frozen_param1}). All other assertions of Proposition \ref{prop:frozen_boundary_intro} are readily checked.
\qed

\medskip

From now on we will think of $\om$, $-\infty\le\om<\infty$, as of the parameter of $\partial\D$. In particular, the parametrization of $\partial\D$ by the double critical point $\om$ (\ref{frozen_param1})--(\ref{frozen_param2}) can be used to draw frozen boundary curves, see Fig.~\ref{fig:frozen_boundary} and \ref{fig:frozen_boundary_big}. 

\medskip

\noindent\textit{Proof of Proposition \ref{prop:tangent_points_intro}.}
Let us turn to identification of tangent points on the frozen boundary curve. The slope of the tangent vector $(\dot\chi(\om),\dot\eta(\om))$ to $\partial\D$ is given by 
\begin{align*}
  \frac{\dot\chi(\om)}{\dot\eta(\om)}=
  \frac{\om-\chi(\om)}{1-\eta(\om)}
  =\frac{\Pi(\om)}{1-\Pi(\om)}
  =:\tg(\om).
\end{align*}
Indeed, this can be obtained by differentiating (\ref{frozen_param1}) with respect to $\om$ and noting that $\dot\Pi(\om)=\Pi(\om)\Sigma(\om)$. One can then check that $\frac{\dot\chi(\om)}{\dot\eta(\om)}=\frac{\Pi(\om)}{1-\Pi(\om)}$. It is also easy to verify using (\ref{frozen_param1})--(\ref{frozen_param2}) that $\frac{\om-\chi(\om)}{1-\eta(\om)}$ simplifies to the same expression. 

The rest of Proposition \ref{prop:tangent_points_intro} is not hard to check.
For example, 
the tangent slope $\frac{\Pi(\om)}{1-\Pi(\om)}$ is zero precisely when 
$\om=b_i$, these are points when the tangent vector to the frozen
boundary is vertical. 
In this case $\Sigma(\om)=\infty$, and so 
by \eqref{frozen_param1}--\eqref{frozen_param2}
we have $\chi(\om)=b_i$, $\eta(\om)=1-\frac{1}{\Pi(\om)\Sigma(\om)}=\frac{1}{b_i-a_i}\prod_{j\ne i}\frac{b_i-b_j}{b_i-a_j}$. We see that in this case 
the frozen boundary indeed intersects the vertical side of the polygon.
The ``diagonal'' and ``horizontal'' cases may be considered in a similar manner.
\qed

\medskip

\begin{figure}[htbp]
  \begin{tabular}{c}
    \includegraphics[width=320pt]{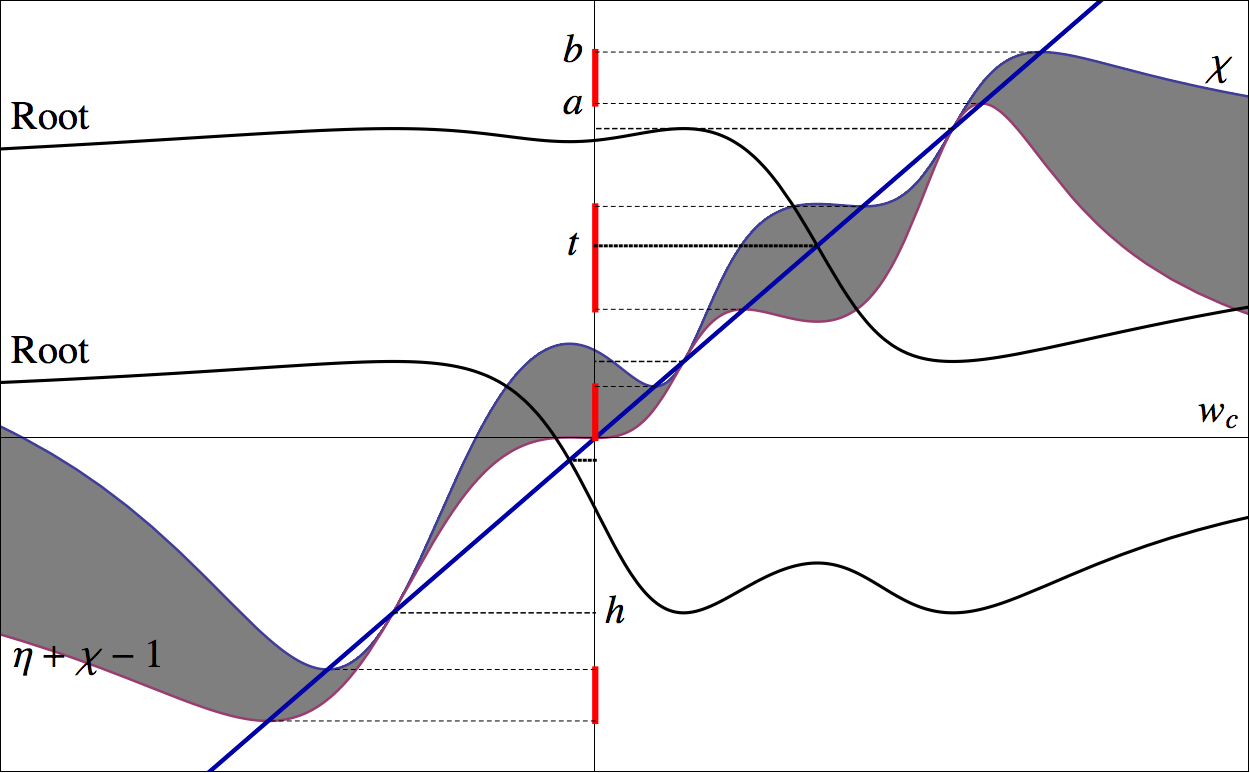}
  \end{tabular}
  \caption{Various parameters along the frozen boundary curve.}
  \label{fig:along_frozen_boundary}
\end{figure}
\begin{figure}[htbp]
  \begin{tabular}{c}
    \includegraphics[width=280pt]{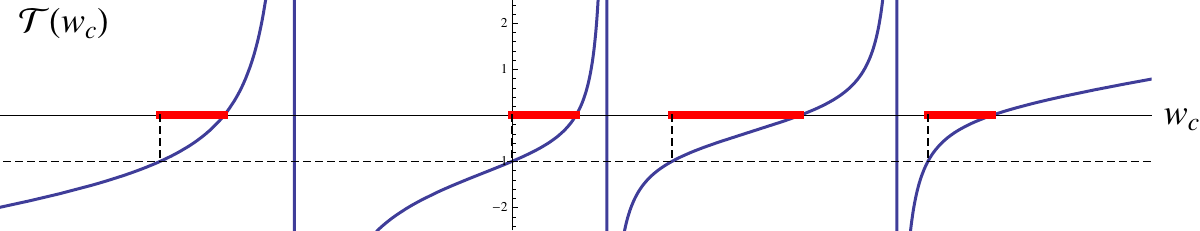}\\
    \includegraphics[width=280pt]{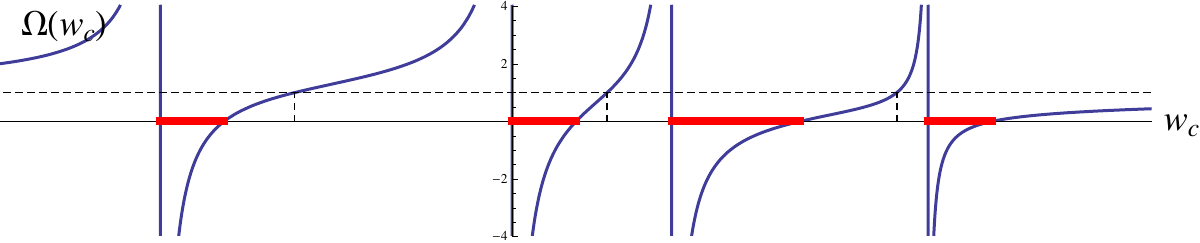}\\
    \includegraphics[width=280pt]{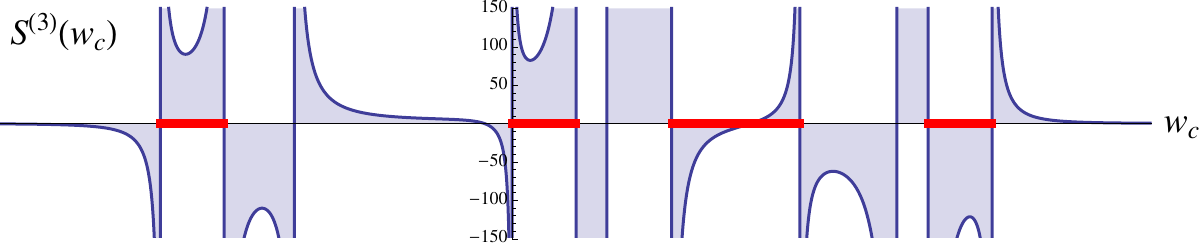}
  \end{tabular}
  \caption{Functions $\tg(\om)$, $\Om(\om)$, and $\szzz(\om)$ along $\partial\D$.}
  \label{fig:tau_om_S3}
\end{figure}
We conclude the discussion of the frozen boundary curve in this subsection by giving an illustration of how various parameters change along the curve. On Fig.~\ref{fig:along_frozen_boundary} there is an example for $k=4$. The horizontal axis is the parameter $\om$. The straight line with slope $1$ represents the double root $\om$ of the equation $\frac{\partial}{\partial w}S(w;\chi,\eta)=0$. Two other roots are also shown. There are two points where $\om$ is a triple root, they correspond to turning points of the frozen boundary (one of them is marked as `$t$'). The segments shown on the vertical line are $[a_i,b_i]$. The shaded region shows the segment $[\chi(\om)+\eta(\om)-1,\chi(\om)]$. When $\om=a_i,b_i$, or $h_i$ (see the marks `$a$', `$b$', and `$h$'), the frozen boundary is tangent to a side of the polygon according to Proposition \ref{prop:tangent_points_intro}. 

Fig.~\ref{fig:tau_om_S3} shows the functions $\tg(\om)$, $\Om(\om)$, and the third derivative $\szzz(\om):=\frac{\partial^{3}}{\partial z^{3}}S(z;\chi(\om),\eta(\om))|_{z=\om}$ as functions of the parameter $\om$ along the frozen boundary curve. The function $\tg(\om)$ describes the direction of the tangent vector to the curve, and $\Om(\om)$ shows the type of the neighboring frozen facet outside the liquid region, see Proposition \ref{prop:tangent_points_intro} and the discussion in \S \ref{sub:asymptotics_at_the_edge_of_the_limit_shape} (in particular, Fig.~\ref{fig:Airy}). The third derivative changes its sign at every tangent and turning point where it is infinite or zero, respectively.

% subsection frozen_boundary (end)

% section asymptotics_in_the_bulk_limit_shape_and_frozen_boundary (end)

\section{Asymptotics at the edge} % (fold)
\label{sec:asymptotics_at_the_edge}

This section is devoted to studying the asymptotics of random lozenge tilings on the edge of the limit shape. See \S \ref{sub:asymptotics_at_the_edge_of_the_limit_shape} for a preliminary discussion and illustrations.

\subsection{Limit regime} % (fold)
\label{sub:formulation_airy}

Fix a point $(\chi,\eta)=(\chi(\om),\eta(\om))\in\partial\D$ which is neither a tangent nor a turning point. Here $\om$ is the parameter of the frozen boundary curve (Proposition \ref{prop:frozen_boundary_intro}). Let the points $(x_1,n_1),\ldots,(x_s,n_s)$ be in the vicinity of the point $(\chi,\eta)$, i.e.,
\begin{align*}
  {x_i}/{N}\to\chi,\qquad
  n_i/N\to\eta,\qquad N\to\infty \qquad (i=1,\ldots,s).
\end{align*}
In contrast to the bulk limit regime (Theorem \ref{thm:bulk_intro}), at the edge to obtain nontrivial asymptotics, the differences $x_i-x_j$ and $n_i-n_j$ grow as $N\to\infty$. More precisely, we assume the following scaling:
\begin{align}
  \label{xpnp}
  x_i=\chi(\om) N+\tg(\om)\cdot n_i'N^{2/3}+x_i' N^{1/3},\qquad
  n_i=\eta(\om) N+n_i' N^{2/3},
\end{align}
where 
\begin{equation}
  \label{tau_sigma}
  \begin{array}{ll}
    n_i'=-2^{1/3}(\szzz)^{2/3}(1-\eta)^{2}\Om(1-\Om)^{-2}\tau_i,\\
    \rule{0pt}{13pt}
    x_i'=2^{-1/3}(\szzz)^{1/3}(1-\eta)\Om(1-\Om)^{-2}(\si_i-\tau_i^{2})
  \end{array}
\end{equation}
(recall that $\szzz=\szzz(\om)=\frac{\partial^{3}}{\partial z^{3}}S(z;\chi(\om),\eta(\om))|_{z=\om}$). Here $\tau_i$ and $\si_i$, $i=1,\ldots,s$, are the new coordinates around $(\chi,\eta)$. Because of the factor $\tg(\om)$ in (\ref{xpnp}) and due to (\ref{tangent_vector}), one sees that these new coordinates describe fluctuations of order $N^{2/3}$ in tangent direction to $\partial\D$ and of order $N^{1/3}$ in normal direction. The exact form of the scaling (\ref{tau_sigma}) is chosen to ensure that the limiting process would be the standard Airy process with kernel (\ref{extended_Airy}).

% subsection formulation_airy (end)

\subsection{Formulation of the result} % (fold)
\label{sub:formulation_of_the_result}

Let $K$ be the correlation kernel of our model (\ref{K_intro}). For $\Om(\om)>0$, let the scaled correlation kernel be defined by
\begin{align}\label{K_scaled_1}
  \mathcal{K}(x_1,n_1;x_2,n_2):=
  N^{1/3}|\szzz/2|^{1/3}(1-\eta)|\Om|(1-\Om)^{-2}
  {K}(x_1,n_1;x_2,n_2),
\end{align}
and for $\Om(\om)<0$ we set
\begin{align}\label{K_scaled_2}
  \mathcal{K}(x_1,n_1;x_2,n_2):={}&
  N^{1/3}|\szzz/2|^{1/3}(1-\eta)
  \times\\&\qquad\times
  |\Om|(1-\Om)^{-2}
  \big[
  1_{x_1=x_2}1_{n_1=n_2}
  -{K}(x_1,n_1;x_2,n_2)
  \big].
  \nonumber
\end{align}
According to Proposition \ref{prop:tangent_points_intro}, this means that for the DV part of the frozen boundary (Fig.~\ref{fig:Airy}) we are replacing particles (on Fig.~\ref{fig:polygonal_region_tiling}) by holes and vice versa. For a general discussion of the particle-hole involution, e.g., see \cite[Appendix 3]{Borodin2000b}. In our model this involution is done because the neighboring frozen facet at the DV part of $\partial\D$ is densely packed with particles, and thus we must look at fluctuations of holes.

\begin{theorem}\label{thm:Airy_full}
  In the scaling (\ref{xpnp})--(\ref{tau_sigma}), we have the convergence
  \begin{align*}
    \det[\mathcal{K}(x_i,n_i;x_j,n_j)]_{i,j=1}^{s}\to
    \det[\A(\tau_i,\si_i;\tau_j,\si_j)]_{i,j=1}^{s},
  \end{align*}
  where $\A$ is the extended Airy kernel (\ref{extended_Airy}).
\end{theorem}

The interpretation of this result in terms of fluctuations of nonintersecting paths is discussed in \S \ref{sub:asymptotics_at_the_edge_of_the_limit_shape}. In the rest of this section we prove Theorem \ref{thm:Airy_full}.

% subsection formulation_of_the_result (end)

\subsection{Kernel $\tilde K$} Arguing as in the proof of Proposition \ref{prop:asymptotics_of_the_kernel}, we see that in the regime (\ref{xpnp})--(\ref{tau_sigma}) the kernel $K(x_1,n_1;x_2,n_2)$ (\ref{K_intro}) has the following asymptotics:
\begin{align}&
  K(x_1,n_1;x_2,n_2)\sim
  -1_{n_2<n_1}1_{x_2\le x_1}\frac{(x_1-x_2+1)_{n_1-n_2-1}}{(n_1-n_2-1)!}
  +\frac{1-\eta}{(2\pi\i)^{2}}
  \label{K_edge_1}
  \times\\&\qquad\times
  \nonumber
  \oint\limits_{\Ga(\chi-)}dz
  \oint\limits_{\ga(\infty)}dw
  \frac{1}{w-z}
  \frac{\exp\Big(
  N\big[S(w;\tfrac{x_1}N,\tfrac{n_1}N)-
  S(z;\tfrac{x_2}N,\tfrac{n_2}N)\big]\Big)}
  {\sqrt{(w-\chi)(w-\chi+1-\eta)}\sqrt{(z-\chi)(z-\chi+1-\eta)}}.
\end{align}
The contours $\Ga(\chi-)$ and $\ga(\infty)$ are the same as described in Proposition \ref{prop:asymptotics_of_the_kernel}.
As explained in the proof of that proposition
(\S \ref{sub:asymptotics_of_the_kernel}), 
the asymptotically equivalent expression under the integral
may be taken to be uniform
in $w$ and $z$ belonging to 
compact subsets of 
$\C\setminus(-\infty,t]$ or $\C\setminus[t,+\infty)$ for appropriate $t$.
This uniformity is needed to justify convergence discussed in 
\S \ref{sub:transforming_the_contours} below.

It is convenient to conjugate the kernel $K$ as
\begin{align}\label{tilde_K}
  \tilde K(x_1,n_1;x_2,n_2):=
  e^{N\big(
  S(\om;\tfrac{x_2}N,\tfrac{n_2}N)-
  S(\om;\tfrac{x_1}N,\tfrac{n_1}N)
  \big)}
  K(x_1,n_1;x_2,n_2).
\end{align}
In the definition of the action $S(\om;\tfrac {x_{1,2}}N,\tfrac {n_{1,2}}N)$ (\ref{action}) we must take cuts in logarithms in, say, the lower half plane because $\om\in\R$. Of course, the new kernel $\tilde K$ can also serve as a correlation kernel of our measure $\Pp_{\Pc(N)}$. 

\subsection{Expansion of the action at the edge} % (fold)
\label{sub:expansion_of_action_at_the_edge}

In the integral in $\tilde K$ (\ref{tilde_K}), there is an exponent of the difference of two expressions like
\begin{align*}
  N\big(S(w;\tfrac xN,\tfrac nN)-S(\om;\tfrac xN,\tfrac nN)\big).
\end{align*}
Let us expand such an expression in powers of $N$:

\begin{lemma}\label{lemma:edge_S_expansion}
  Let $w=\om+w'N^{-1/3}$. For $N(S(w;\frac xN,\frac nN)-S(\om;\frac xN,\frac nN))$ to be asymptotically bounded, it suffices to take
  \begin{align*}
    x=\chi N + n'\tg(\om)N^{2/3}+x'N^{1/3},\qquad
    n=\eta N+n' N^{2/3}
  \end{align*}
  for some $x',n'\in\R$. Then
  \begin{align*}
    N(S(w;\tfrac xN,\tfrac nN)-S(\om;\tfrac xN,\tfrac nN))
    &=\tfrac16
    \szzz w'^3+
    \tfrac12 
    (1-\eta)^{-2}(1-\Om)^{2}\Om^{-1}
    n'w'^2\\&\qquad \qquad
    -
    (1-\eta)^{-1}(1-\Om)^{2}\Om^{-1}x'w'+o(1).
  \end{align*}
\end{lemma}
\begin{proof}
  Observe that for partial derivatives at $(\om,\chi,\eta)$ one has
  \begin{align*}
    S_{w}=S_{ww}=0,\qquad
    S_{w\chi\chi}=-S_{ww\chi},\qquad
    S_{w\eta\eta}=S_{w\chi\eta}=-S_{ww\eta}.
  \end{align*}
  Denoting $d\chi=\frac{x-N\chi}{N}$, $d\eta=\frac{n-N\eta}{N}$ and taking the Taylor expansion at $(\om;\chi,\eta)$, we obtain
  \begin{align*}
    &
    N\Big(S(\om+w'N^{-1/3};\chi+d\chi,\eta+d\eta)-S(\om;\chi+d\chi,\eta+d\eta)\Big)
    \\&=
    N^{2/3}w'
    \Big[(S_{w\chi}d\chi +S_{w\eta}d\eta)
    + 
    \tfrac12 
    S_{w\chi\chi}
    d\chi^2 
    +
    S_{w\eta\eta}d\chi d\eta 
    +\tfrac12
    S_{w\eta\eta}
    d\eta^2 
    \Big]\\&\qquad-
    \tfrac12N^{1/3}w'^{2}\Big[
    S_{w\chi\chi}d\chi
    +S_{w\eta\eta}d\eta
    \Big]+\tfrac16
    \szzz w'^3+O(N^{-1/3}).
  \end{align*}
  We see that if $d\chi,d\eta\sim\mathrm{const}\cdot N^{-1/3}$, and that also $S_{w\chi}d\chi +S_{w\eta}d\eta\sim\mathrm{const}\cdot N^{-2/3}$, then the above expression is asymptotically bounded. Then it is readily checked that the desired expansion holds.
\end{proof}

Lemma \ref{lemma:edge_S_expansion} justifies the scaling (\ref{xpnp}) of the coordinates $(x,n)$, and also suggests scaling of variables in the double contour integral in the kernel $\tilde K$.

% subsection expansion_of_action_at_the_edge (end)

\subsection{Extended Airy kernel} % (fold)
\label{sub:extended_airy_kernel}
\begin{figure}[htbp]
  \begin{tabular}{c}
    \includegraphics[height=150pt]{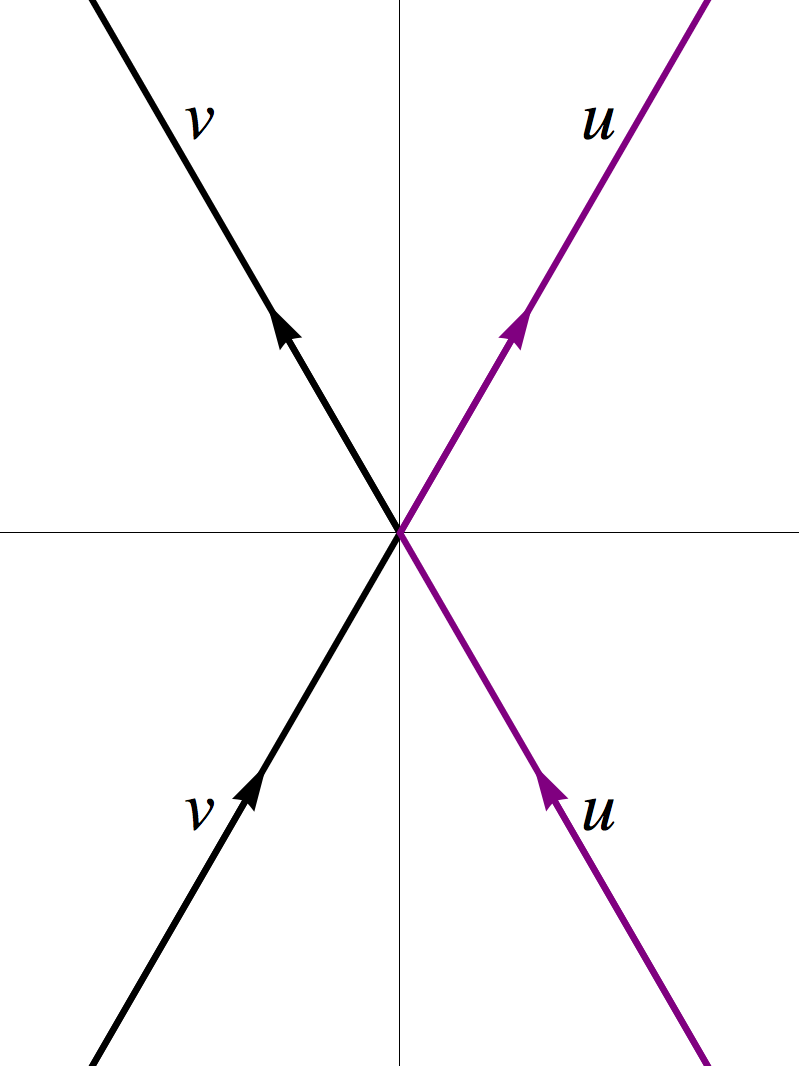}
  \end{tabular}
  \caption{Contours for the extended Airy kernel.}
  \label{fig:Airy_contours}
\end{figure}
We will use another formula for the extended Airy kernel $\A(\tau_1,\si_1;\tau_2,\si_2)$ (\ref{extended_Airy}) which is written out in \cite[\S4.6]{BorodinKuan2007U}:
\begin{align}&
  \label{ext_Airy_long}
  \A(\tau_1,\si_1;\tau_2,\si_2)\\&=
  -\frac{1_{\tau_1<\tau_2}}{\sqrt{4\pi(\tau_2-\tau_1)}}
  \exp\left(
  - \frac{(\si_1-\si_2)^2}{4(\tau_2-\tau_1)}
  -\frac12(\tau_2-\tau_1)(\si_1+\si_2)
  +\frac1{12}(\tau_2-\tau_1)^3
  \right)
  \nonumber
  \\&\qquad+
  \frac1{(2\pi\i)^2}\int\int
  \exp\Big(
  \tau_1\si_1-\tau_2\si_2
  -\frac{1}{3}\tau_1^3+\frac{1}{3}\tau_2^3
  -(\si_1-\tau_1^2)u+(\si_2-\tau_2^2)v
  \nonumber
  \\&\qquad \qquad \qquad \qquad \qquad \qquad
  \qquad \qquad \qquad \qquad
  -\tau_1 u^2+\tau_2 v^2
  +
  \frac13(u^3-v^3)
  \Big)
  \frac{dudv}{u-v}.\nonumber
\end{align}
Both contours come from infinity as straight lines, then meet at $0$, and go to infinity in other directions. In $u$ the contour goes from $e^{-\i\pi/3}\infty$ to $0$ to $e^{\i\pi/3}\infty$, and in $v$ it is from $e^{-2\i\pi/3}\infty$ to $0$ to $e^{2\i\pi/3}\infty$ (see Fig.~\ref{fig:Airy_contours}).

Comparing cubic polynomials in the exponent under the integral in (\ref{ext_Airy_long}) with the expansion of Lemma \ref{lemma:edge_S_expansion}, we see that the parameters $x'_{1,2}$ and $n'_{1,2}$ should necessarily have the form (\ref{tau_sigma}) in order for our pre-limit kernel $\tilde K$ to converge to the extended Airy kernel (\ref{ext_Airy_long}). Indeed, under (\ref{xpnp})--(\ref{tau_sigma}), we have:
\begin{align*}
  N\big[S(w;\tfrac{x_1}N,\tfrac{n_1}N)&
  -
  S(\om;\tfrac{x_1}N,\tfrac{n_1}N)-
  S(z;\tfrac{x_2}N,\tfrac{n_2}N)
  +S(\om;\tfrac{x_2}N,\tfrac{n_2}N)
  \big]\\&=
  -(\si_1-\tau_1^2)u+(\si_2-\tau_2^2)v
  -\tau_1 u^2+\tau_2 v^2
  +
  \tfrac13(u^3-v^3),
\end{align*}
where $u$ and $v$ are the new variables, 
\begin{align}\label{new_vars_u_v}
  w=\om+N^{-\frac13}(\szzz/2)^{-\frac13}u,\qquad
  z=\om+N^{-\frac13}(\szzz/2)^{-\frac13}v.
\end{align}
Thus, up to the factor $e^{\tau_1\si_1-\tau_2\si_2-\frac13\tau_1^3+\frac13\tau_2^3}$ (which simply corresponds to another conjugation of the correlation kernel), we get in $\tilde K$ (\ref{K_edge_1})--(\ref{tilde_K}) the same exponent as in the extended Airy kernel (\ref{ext_Airy_long}). 

To make the new variables $u$ and $v$ (\ref{new_vars_u_v}) local around the double critical point $\om$ (as on Fig.~\ref{fig:Airy_contours}), we need to transform the $z$ and $w$ contours so that they will cross at $\om$ and locally look as on Fig.~\ref{fig:Airy_contours_zw}. After that, the leading contribution to the double contour integral for $\tilde K$ (\ref{K_edge_1})--(\ref{tilde_K}) will come from a neighborhood of $\om$, and we will see the desired convergence to the correlation functions of the Airy process (Theorem \ref{thm:Airy_full}).

\begin{figure}[htbp]
  \begin{tabular}{ccc}
    \includegraphics[height=130pt]{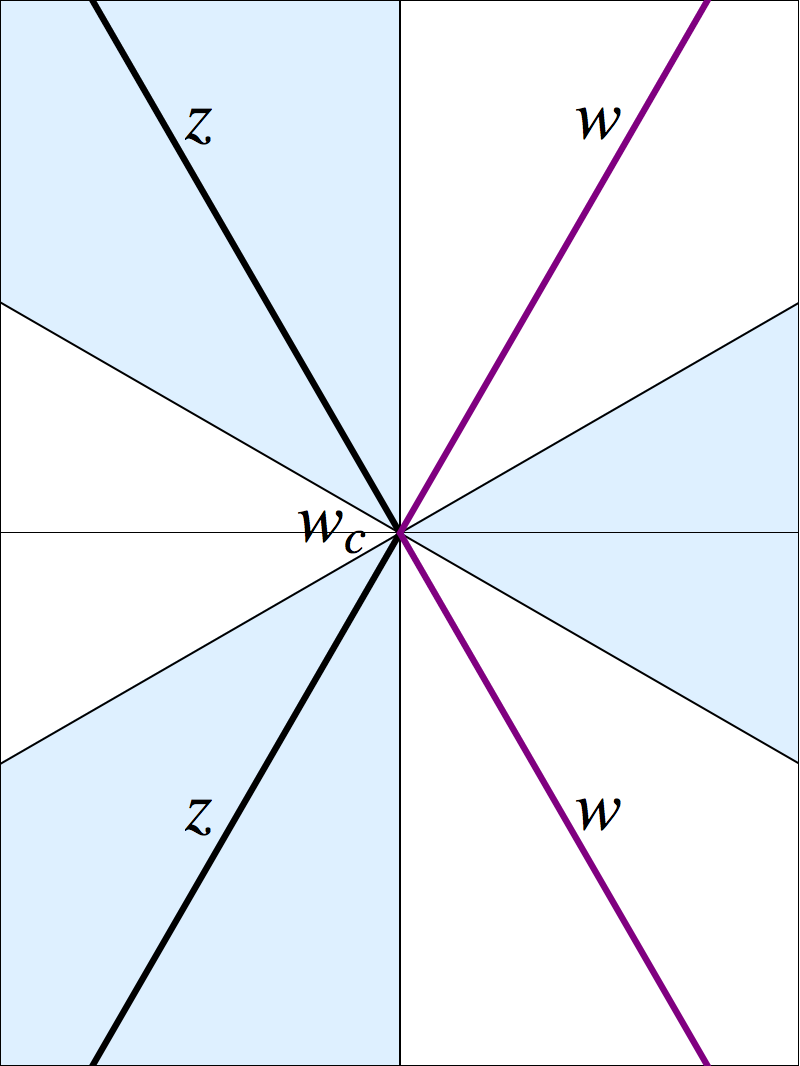}&
    \rule{50pt}{0pt}
    &
    \includegraphics[height=130pt]{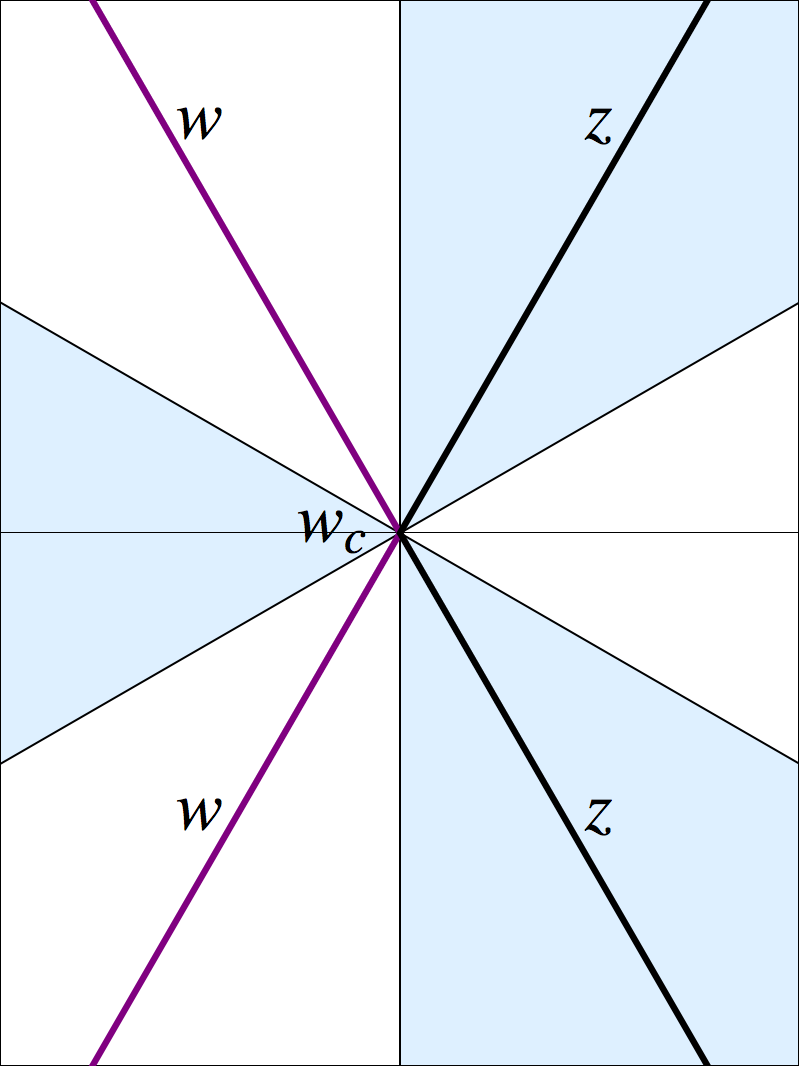}
  \end{tabular}
  \caption{The $z$ and $w$ contours around the double critical point $\om$ when $\szzz$ is positive (left) or negative (right). Shaded is the region where $\Re(S(w;\chi,\eta)-S(\om;\chi,\eta))>0$.}
  \label{fig:Airy_contours_zw}
\end{figure}

% subsection extended_airy_kernel (end)

\subsection{Transforming the contours and completing the proof} % (fold)
\label{sub:transforming_the_contours}

Here we explain how the contours of integration for the kernel $\tilde K$ (\ref{K_edge_1})--(\ref{tilde_K}) should be transformed. In order for the integration outside a neighborhood of $\om$ to be negligible, we require that on the contours $\Re S(w)<\Re S(\om)<\Re S(z)$ for $z,w\ne \om$. In the process of moving the contours, certain residues may be picked, and as a result we obtain the desired convergence as in Theorem~\ref{thm:Airy_full}.

The rules for dragging the $z$ and $w$ contours are explained in the beginning of the proof of Proposition \ref{prop:moving_bulk}. We consider three cases corresponding to various directions of the frozen boundary as indicated on Fig.~\ref{fig:Airy} and in Proposition \ref{prop:tangent_points_intro}. It is also informative to consult Fig.~\ref{fig:along_frozen_boundary} and \ref{fig:tau_om_S3}. 

\subsubsection{Case VH: $0<\Om(\om)<1$ and $b_i<\om<h_{i+1}$} % (fold)
\label{ssub:case1}
In this case the double critical point $\om$ lies to the right of $\chi$ because $\tg(\om)>0$. Looking at curves $\{w\colon \Im S(w;\chi,\eta)=\Im S(\om;\chi,\eta)\}$ and counting arguments (similarly to the proof of Proposition \ref{prop:moving_bulk}), one can see that it is possible to transform the contours in a desired way as shown on Fig.~\ref{fig:case1_S3}. 

\begin{figure}[htbp]
  \begin{tabular}{cc}
    \includegraphics[width=170pt]{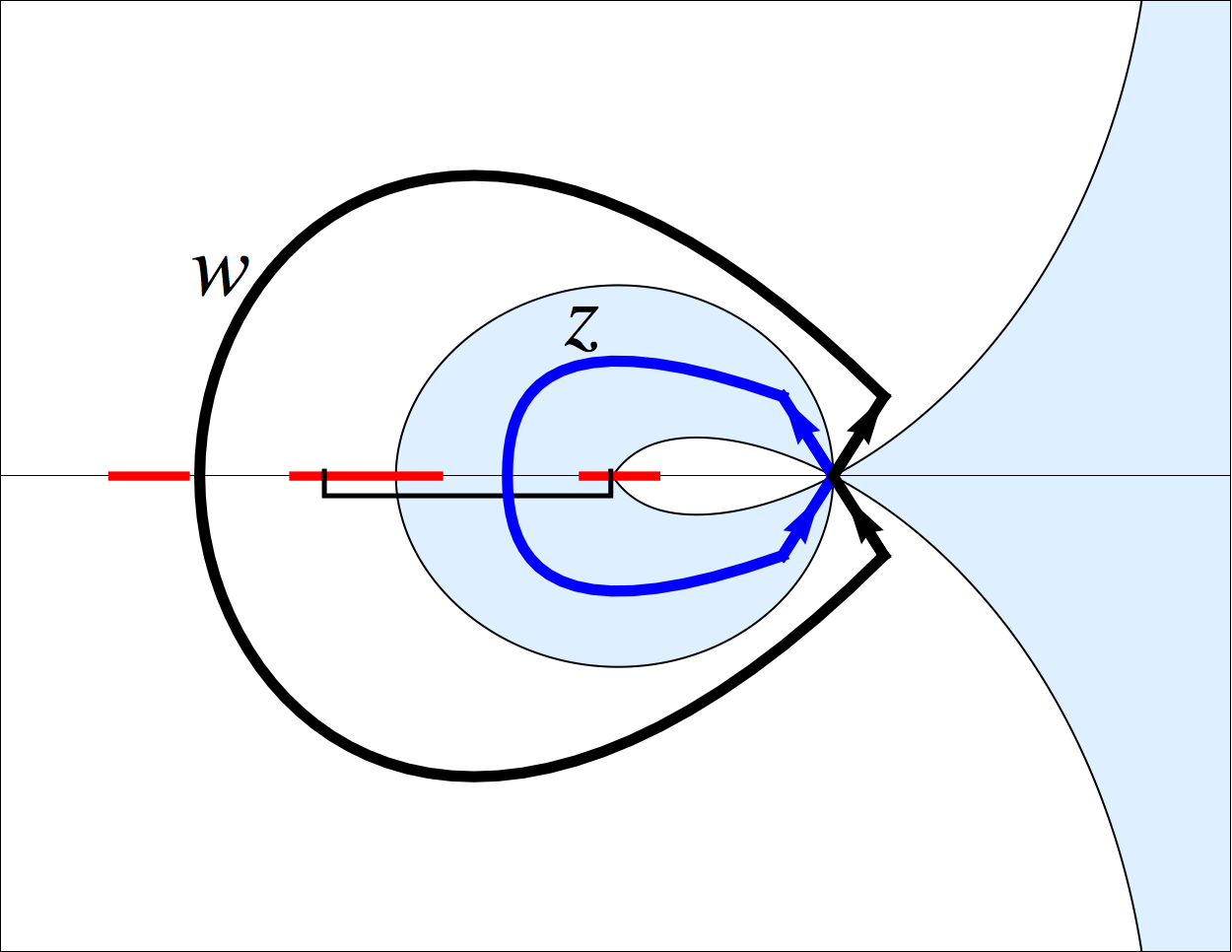}&
    \includegraphics[width=170pt]{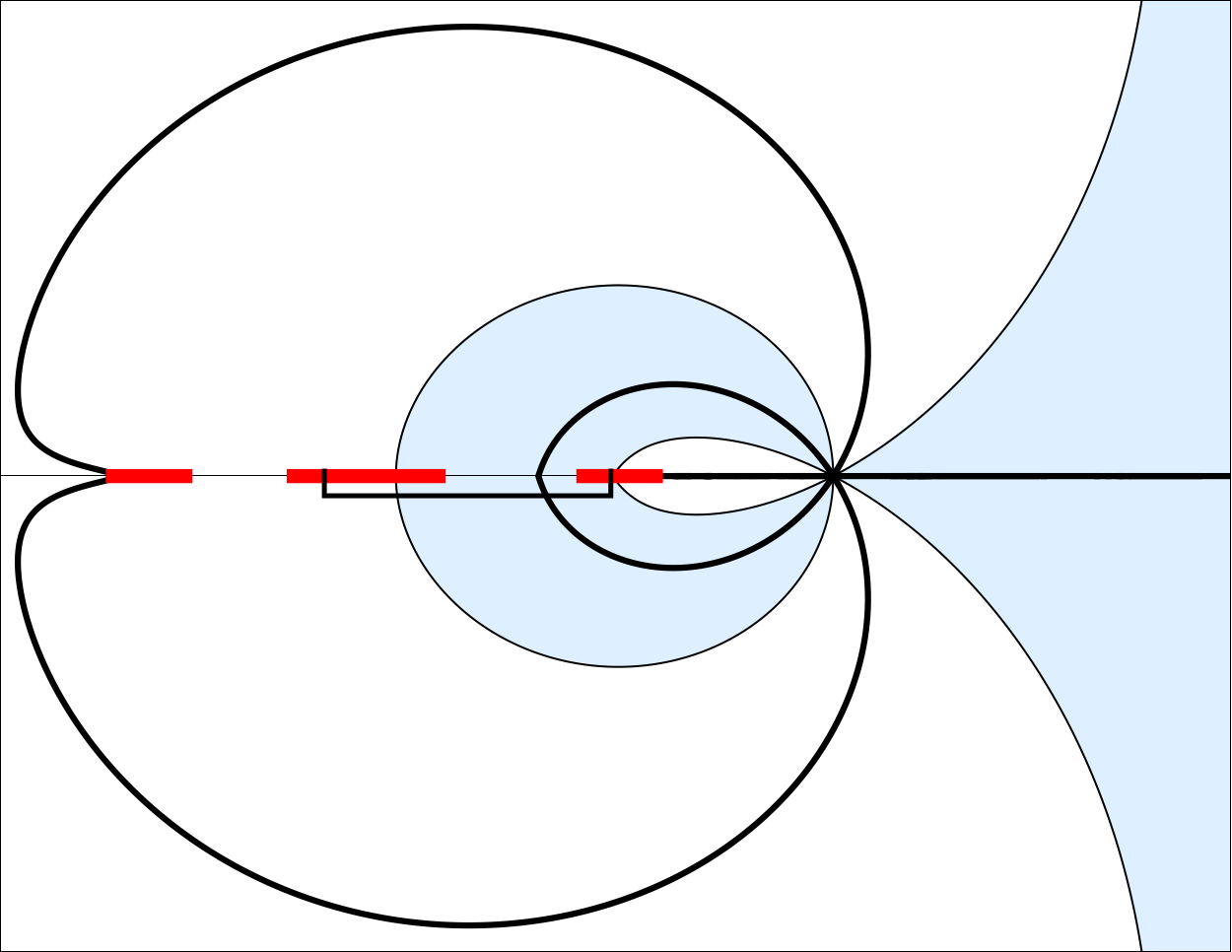}\\
    \includegraphics[width=170pt]{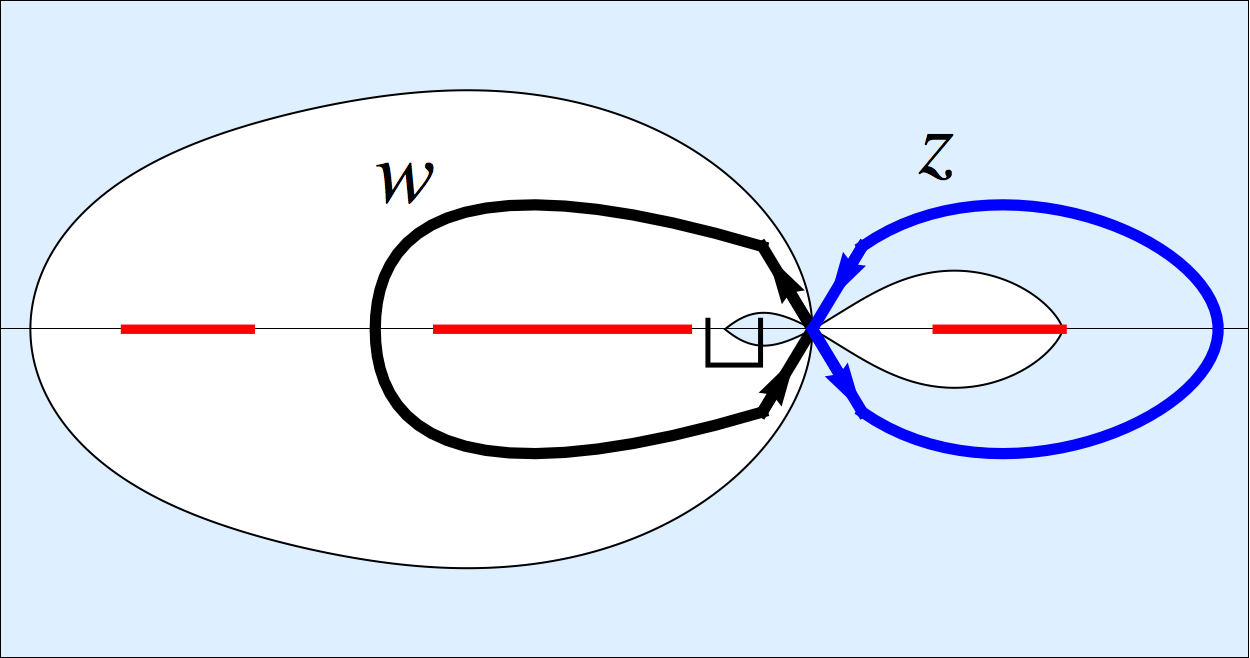}&
    \includegraphics[width=170pt]{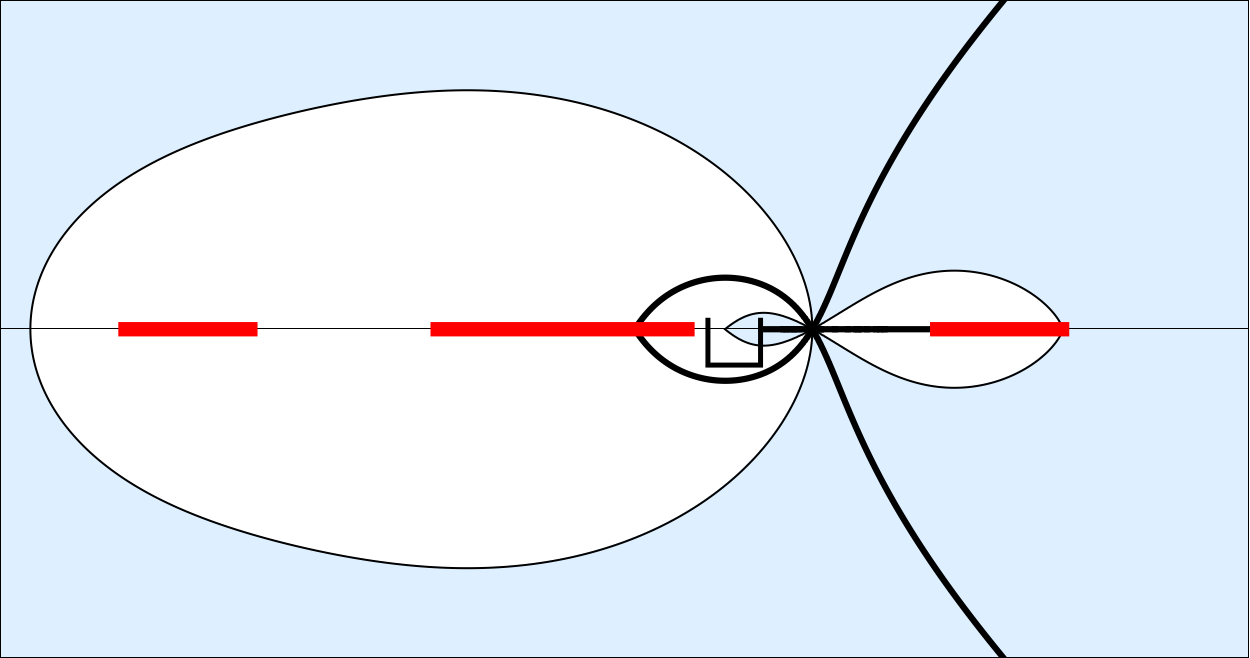}
  \end{tabular}
  \caption{Case VH (\S\ref{ssub:case1}). $\szzz>0$ (top) and $\szzz<0$ (bottom). Transformed contours (left) and curves where $\Im S(w;\chi,\eta)=\Im S(\om;\chi,\eta)$ (right).}
  \label{fig:case1_S3}
\end{figure}

For some configurations, it may be needed to drag the $w$ contour through the $z$ contour (as on Fig.~\ref{fig:case1_S3}, bottom). As a result, a residue at $w=z$ will be picked, but that residue will be integrated over a $z$ contour not containing its poles which are inside $[\chi+\eta-1,\chi]$. Thus, no summand in addition to the one in (\ref{K_edge_1}) will arise. For $\szzz<0$, locally the contours around $\om$ will have directions that produce a minus sign (see Fig.~\ref{fig:case1_S3}, bottom). However, this sign is absorbed by the change of variables (\ref{new_vars_u_v}): recall that under the integral we have $\frac{dzdw}{w-z}$. Due to the discussion in \S \ref{sub:extended_airy_kernel}, we have now established the desired convergence of the integral in $\tilde K$ (\ref{K_edge_1})--(\ref{tilde_K}) to the integral in the extended Airy kernel (\ref{ext_Airy_long}) (with correction as in (\ref{case1_convergence_correction}) below).

Let us turn to the additional summand. In scaling (\ref{xpnp})--(\ref{tau_sigma}), the fact that both conditions $n_2<n_1$ and $x_2\le x_1$ are satisfied translates to the condition $\tau_1<\tau_2$. Then it can be shown using the Stirling approximation for the Gamma function (\ref{Stirling_gamma}) that the additional summand in $\tilde K$ (\ref{K_edge_1})--(\ref{tilde_K}) behaves as
\begin{align*}&
  -1_{n_2<n_1}1_{x_2\le x_1}
  N^{1/3}|{\szzz}/2|^{1/3}
  \frac{|\Om|(1-\eta)}{(1-\Om)^{2}}
  e^{N\big(
  S(\om;\tfrac{x_2}N,\tfrac{n_2}N)-
  S(\om;\tfrac{x_1}N,\tfrac{n_1}N)
  \big)}
  \times\\& \qquad\times
  \frac{(x_1-x_2+1)_{n_1-n_2-1}}{(n_1-n_2-1)!}
  \to
  -1_{\tau_1<\tau_2}
  \frac{\exp\left(-\frac14
  {\left(\si_1-\si_2-\tau_1^2+\tau_2^2\right)^2}/{(\tau_2-\tau_1)}
  \right)}
  {\sqrt{4\pi(\tau_2-\tau_1)}}.
\end{align*}
Summarizing, we conclude that in the case VH the rescaled kernel (\ref{K_scaled_1}) converges to the extended Airy kernel as follows:
\begin{align}\label{case1_convergence_correction}
  \mathcal{K}(x_1,n_1;x_2,n_2)\to
  e^{-\tau_1\si_1+\tau_2\si_2+\frac13\tau_1^3-\frac13\tau_2^3}
  \A(\tau_1,\si_1;\tau_2,\si_2).
\end{align}
Because the exponent factor is just a conjugation not affecting the determinant, this implies Theorem \ref{thm:Airy_full} in this case.

% subsubsection case1 (end)

\subsubsection{Case HD: $\Om(\om)>1$ and $h_i<\om<a_i$} % (fold)
\label{ssub:case2}

First, observe that the scaling (\ref{xpnp})--(\ref{tau_sigma}) implies that the two conditions $n_2<n_1$ and $x_2\le x_1$ cannot hold simultaneously because $\tg(\om)<0$. Thus, there is no additional summand in $\tilde K$ coming from (\ref{K_edge_1}). 

Since in this case $\tg(\om)<-1$, the double critical point $\om$ lies to the left of $\chi+\eta-1$. We transform the contours as on Fig.~\ref{fig:case2_S3} (one can also draw curves with $\Im S(w;\chi,\eta)=\Im S(\om;\chi,\eta)$ as on Fig.~\ref{fig:case1_S3}, but we omit this): 
\begin{figure}[htbp]
  \begin{tabular}{cc}
    \includegraphics[height=98pt]{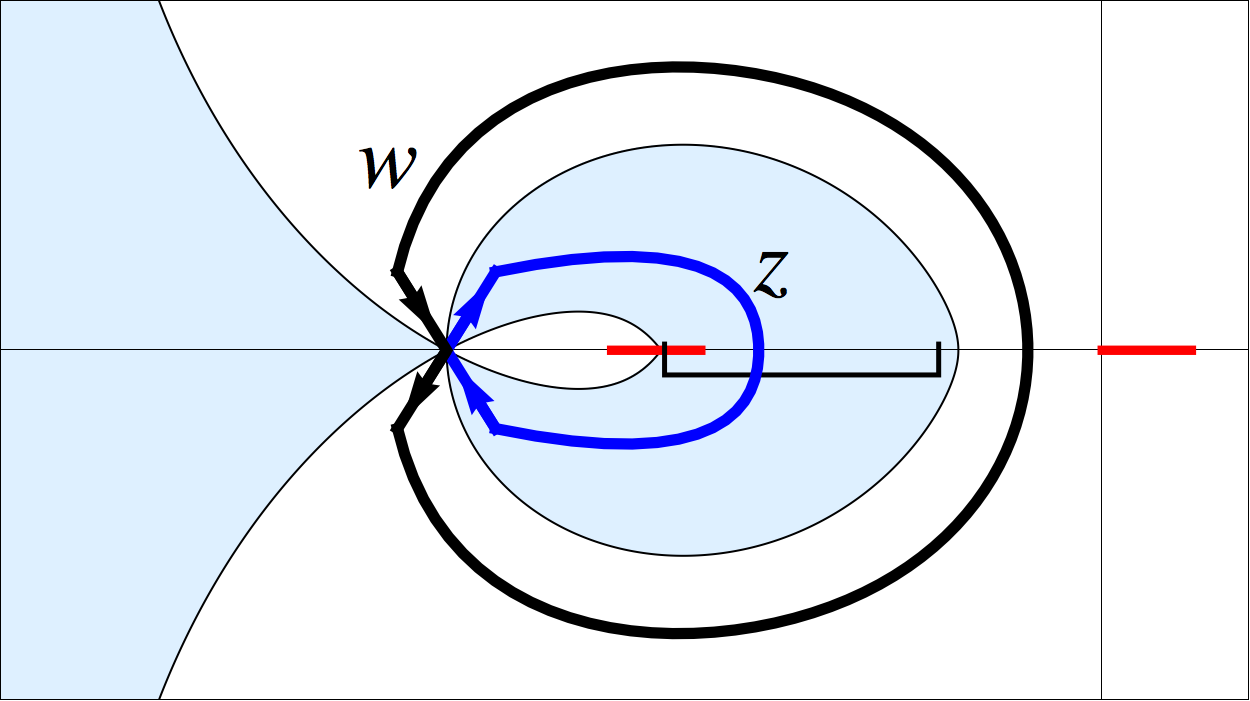}
    &
    \includegraphics[height=98pt]{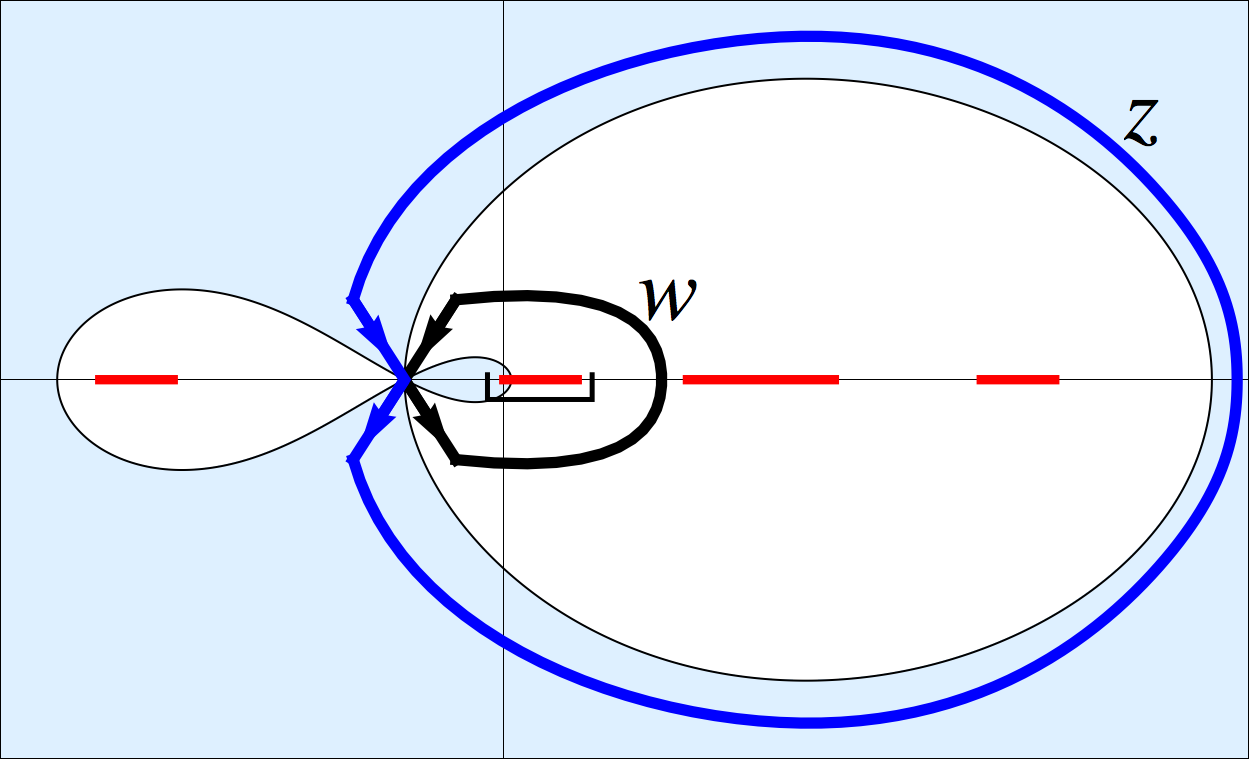}
  \end{tabular}
  \caption{Case HD (\S\ref{ssub:case2}). 
  Transformed contours for $\szzz<0$ (left) and $\szzz>0$ (right).}
  \label{fig:case2_S3}
\end{figure}
\begin{enumerate}[$\bullet$]
  \item For $\szzz<0$, we replace the counter-clockwise $z$ contour coming from $\Ga(x_2)$ by a clockwise contour corresponding to $\Ga'(x_2)$ (see Lemmas \ref{lemma:Res_w=z} and \ref{lemma:Res_w=z_1} for notation). This leads to appearance of the residue (\ref{Res_w=z_3}). The change of sign in the integral due to local directions of the contours around $\om$ will be again compensated by the change of variables (\ref{new_vars_u_v}).

  \item For $\szzz>0$, the $w$ contour is dragged inside the $z$ one, and the resulting residue at $w=z$ gives (\ref{Res_w=z_2}) after the integration. Alternatively, one could also transform the contours in the same way as in the $\szzz<0$ case.
\end{enumerate}
We see that either way, a transformation of contours results in the additional summand $1_{n_1>n_2}\frac{(x_1-x_2+1)_{n_1-n_2-1}}{(n_1-n_2-1)!}$ in (\ref{K_edge_1}). By (\ref{xpnp})--(\ref{tau_sigma}), the condition $n_1>n_2$ turns into $\tau_1<\tau_2$. It can be shown that for our case $\Om(\om)>1$, the rescaled additional summand 
\begin{align*}&
  1_{n_1>n_2}
  N^{1/3}|{\szzz}/2|^{1/3}
  \frac{|\Om|(1-\eta)}{(1-\Om)^{2}}
  e^{N\big(
  S(\om;\tfrac{x_2}N,\tfrac{n_2}N)-
  S(\om;\tfrac{x_1}N,\tfrac{n_1}N)
  \big)}
  \frac{(x_1-x_2+1)_{n_1-n_2-1}}{(n_1-n_2-1)!}
\end{align*}
converges to $-1_{\tau_1<\tau_2}{\exp\left(-\frac14{\left(\si_1-\si_2-\tau_1^2+\tau_2^2\right)^2}/{(\tau_2-\tau_1)}\right)}/{\sqrt{4\pi(\tau_2-\tau_1)}}$. This fact and the transformed contours ensure the desired convergence (\ref{case1_convergence_correction}), which implies that Theorem \ref{thm:Airy_full} holds in the HD case as well.

% subsubsection case2 (end)

\subsubsection{Case DV: $\Om(\om)<0$ and $a_i<\om<b_i$} % (fold)
\label{ssub:case3}

As in \S \ref{ssub:case2}, there is no additional summand in $\tilde K$ coming from (\ref{K_edge_1}) because $\tg(\om)<0$, see (\ref{xpnp})--(\ref{tau_sigma}). 

Because $-1<\tg(\om)<0$ and $a_i<\om<b_i$, we see that the double critical point $\om$ belongs to the intersection of the segments $(\chi+\eta-1,\chi)\cap(a_i,b_i)$. We transform the contours as shown on Fig.~\ref{fig:case3_S3} (we also omit the curves where one can also draw curves with $\Im S(w;\chi,\eta)=\Im S(\om;\chi,\eta)$).\begin{figure}[htbp]
  \begin{tabular}{cc}
    \includegraphics[height=112pt]{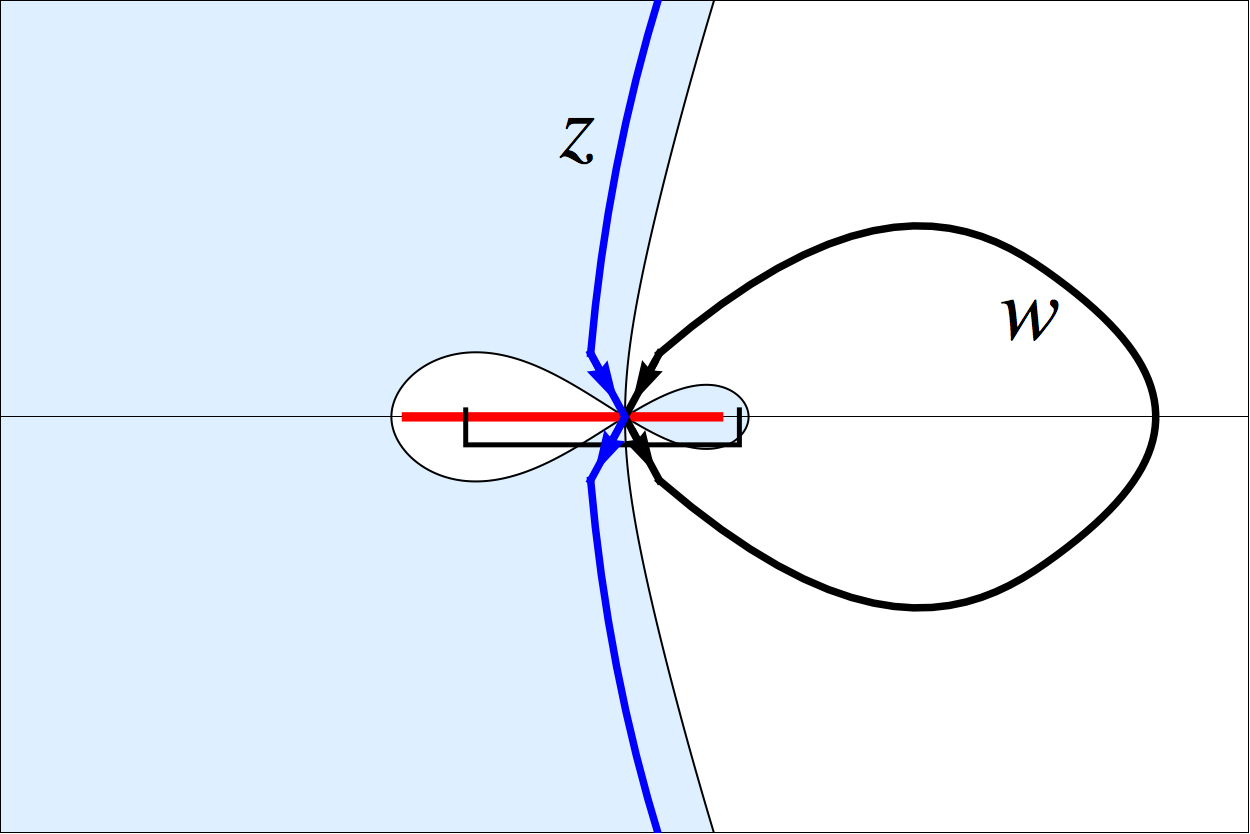}&
    \includegraphics[height=112pt]{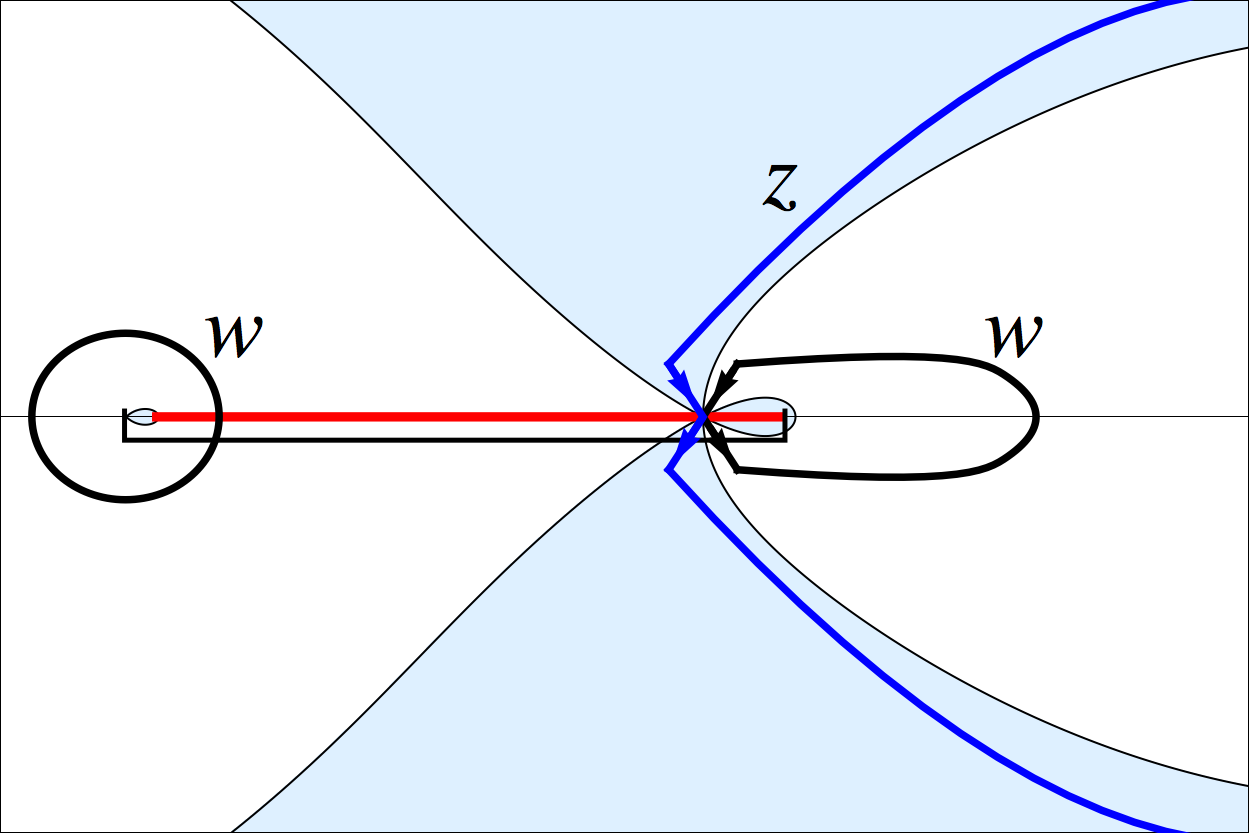}
    \\
    \multicolumn{2}{c}
    {\includegraphics[height=108pt]{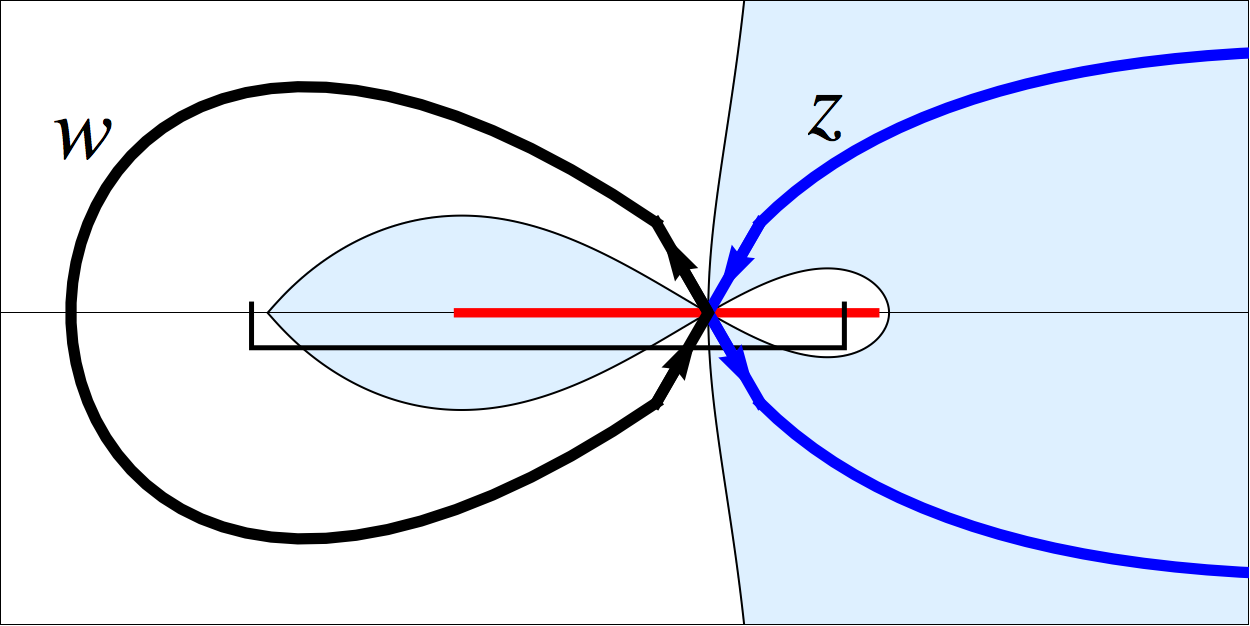}}
  \end{tabular}
  \caption{Case DV (\S\ref{ssub:case3}). 
  Transformed contours for $\szzz>0$ (top; there are variations depending on relative positions of $[a_i,b_i]$ and $[\chi+\eta-1,\chi]$) and $\szzz<0$ (bottom).}
  \label{fig:case3_S3}
\end{figure} As before, for $\szzz<0$, the change of sign due to directions of the contours is compensated by the change of variables (\ref{new_vars_u_v}). There would also be another minus sign in front of the integral in $\tilde K$ (\ref{K_edge_1})--(\ref{tilde_K}) due to the factor $((w-\chi)(w-\chi+1-\eta)(z-\chi)(z-\chi+1-\eta))^{-\frac12}$ in (\ref{K_edge_1}). This justifies the need of the particle-hole involution for $\Om(\om)<0$ (see (\ref{K_scaled_2})) at the level of formulas. So, scaling the kernel $\tilde K$ after the involution as in (\ref{K_scaled_2}), we conclude that for the double contour integral part of the kernel, the desired convergence holds.

While transforming the contours, we drag the $w$ contour through or inside the $z$ one. This results in the residue at $w=z$. It will be integrated as in Lemma \ref{lemma:Res_w=z}, and this will give an additional summand $1_{x_2\le x_1}\frac{(n_1-n_2)_{x_1-x_2}}{(x_1-x_2)!}$. Applying the particle-hole involution (\ref{K_scaled_2}), we see that
\begin{align*}
  1_{x_1=x_2}1_{n_1=n_2}-1_{x_2\le x_1}\frac{(n_1-n_2)_{x_1-x_2}}{(x_1-x_2)!}=
  -1_{x_1>x_2}
  \frac{(n_1-n_2)_{x_1-x_2}}{(x_1-x_2)!}
  -1_{x_1=x_2}1_{n_1\ne n_2}.
\end{align*}
For $\Om(\om)<0$, the scaling (\ref{xpnp})--(\ref{tau_sigma}) implies that the conditions $x_1=x_2$ and $n_1\ne n_2$ cannot hold simultaneously. One can then check that the following convergence holds:
\begin{align*}
  -1_{x_1>x_2}&
  N^{1/3}|{\szzz}/2|^{1/3}
  \frac{|\Om|(1-\eta)}{(1-\Om)^{2}}
  e^{N\big(
  S(\om;\tfrac{x_2}N,\tfrac{n_2}N)-
  S(\om;\tfrac{x_1}N,\tfrac{n_1}N)
  \big)}
  \frac{(n_1-n_2)_{x_1-x_2}}{(x_1-x_2)!}
  \\&\qquad \qquad \qquad \qquad 
  \to
  -1_{\tau_1<\tau_2}\frac{\exp\left(-\frac14{\left(\si_1-\si_2-\tau_1^2+\tau_2^2\right)^2}/{(\tau_2-\tau_1)}\right)}{\sqrt{4\pi(\tau_2-\tau_1)}}.
\end{align*}
Thus, Theorem \ref{thm:Airy_full} is established in the last remaining DV case.

% subsubsection case3 (end)

% subsection transforming_the_contours (end)

% section asymptotics_at_the_edge (end)

\providecommand{\bysame}{\leavevmode\hbox to3em{\hrulefill}\thinspace}
\providecommand{\MR}{\relax\ifhmode\unskip\space\fi MR }
\providecommand{\MRhref}[2]{%
  \href{http://www.ams.org/mathscinet-getitem?mr=#1}{#2}
}
\providecommand{\href}[2]{#2}

\end{document}